\documentstyle[11pt]{article}
\newtheorem{lem}{Lemma}[section]

\newtheorem{theo}{Theorem}

\newtheorem{fact}{Fact}[section]
\newtheorem{coro}{Corollary}[section]
\newtheorem{rem}{Remark}[section]

\newtheorem{prop}{Proposition}[section]

\newtheorem{defi}{Definition}[section]
\newenvironment{proof}
{{\bf Proof:}\noindent}{\begin{flushright}$\Box$\end{flushright}}
\def\mathrm{\rm}
\font\mathfonta=msam10 at 11pt
\font\mathfontb=msbm10 at 11pt
\def\Bbb#1{\mbox{\mathfontb #1}}
\def\lesssim{\mbox{\mathfonta.}}

\def\grtsim{\mbox{\mathfonta\&}}
\def\gtrsim{\mbox{\mathfonta\&}}

\def\Sig{{\bf\Sigma}}

\def\diam{{\mathrm{diam}}}
\def\Jac{{\mathrm{Jac}}}

\def\dist#1{{{\mathrm{dist}}\br{#1}}}
\def\qhdist#1{{{\mathrm{dist_{qh}}}\br{#1}}}

\def\del{\Delta}
\def\om{\Omega}

\def\dd{\partial}

\def\Cal#1{{\cal#1}}
\def\J{J_F}
\def\C{\Bbb C}
\def\N{\Bbb N}
\def\CC{\hat{\Bbb C}}

\def\R{\Bbb R}

\def\Crit{{\mathrm{Crit}}}
\def\ha{{\cal H}}

\def\br#1{\left(#1\right)}

\def\brs#1{\left\{#1\right\}}
\def\abs#1{\left|#1\right|}

\def\len#1{{{\mathrm{length}}\br{#1}}}
\def\qhlen#1{{{\mathrm{length_{qh}}}\br{#1}}}

\def\HD{{\mathrm{HDim}}}
\def\MD{{\mathrm{MDim}}}
\def\MDsup{{\overline{\mathrm{MDim}}}}
\def\MDinf{{\underline{\mathrm{MDim}}}}
\def\HH{{\mathrm{HypDim}}}
\def\dwhit{{\delta_{Whit}}}
\def\dpoin{{\delta_{Poin}}}
\def\dconf{{\delta_{conf}}}

\def\const{{\mathrm{const}}}
\def\fr{\frac}
\def\itemm#1{ \item{\makebox[0.3in][l]{#1}}}

\def\m#1{\mu(#1)}
\def\mum{\mu}
\def\mmax{\mu_{max}}
\def\l2t{L}
\def\supF{{\sup|F'|}}
\def\xx#1#2{{{\cal H}_{#1}(#2)}}
\def\diam{{\mathrm{diam\;}}}
\def\degg#1#2{{\# \br{#1,#2}}}
\def\Jls{{J_*}}
\def\Jlso{{J_{*,\epsilon}}}
\def\typei{{\bf I}}
\def\typeii{{\bf II}}

\def\gr{{\mathrm g}}

\def\Fe{{F^{-1}_{\mathrm{rl}}}}

\begin{document}
\vspace*{1cm}
\begin{tabular}{||c||}\hline\hline
                            \\
{\Huge{\bf Non-uniform Hyperbolicity}} \\ \\
{\Huge{\bf in Complex Dynamics}}\\
 \\ 
{\normalsize { by}}\\
 \\ 
{\Large { JACEK GRACZYK\footnotemark
 \& STANISLAV SMIRNOV\footnotemark
 }} \\ 
              \\  \hline\hline
\end{tabular}
\addtocounter{footnote}{-1}
\footnotetext{Universit\'e de Paris-Sud, Mathematique, 91405 Orsay, France.
(e-mail: graczyk@math.u-psud.fr)} 
\addtocounter{footnote}{1}%
\footnotetext{Universit\'e de Gen\`eve, Section de Math\'ematiques, 2-4 rue du Li\`evre,
CP 64, 1211 Gen\`eve 4, Switzerland
(e-mail: Stanislav.Smirnov@math.unige.ch)}
\addtocounter{footnote}{1}
\footnotetext{Both authors are supported by   EU Research Training Network  CODY.
The second author is supported by the Swiss National Science Foundation.}
\pagestyle{empty}
\clearpage
\begin{abstract}
We say that a rational function $F$ 
satisfies  the summability condition
with exponent $\alpha$ if for every critical point $c$ which
belongs to the Julia set $J$ there
exists a positive integer $n_{c}$ so that 
$\sum_{n=1}^{\infty}
|(F^{n})'(F^{n_{c}}(c))|^{-\alpha}<\infty$ 
and $F$ has no parabolic periodic cycles.
Let $\mmax$ be the maximal multiplicity of the critical points.

The objective is to study the Poincar\'e series for a large class
of rational maps and establish 
 ergodic and regularity   properties of  conformal
measures.
If $F$
is summable with exponent $\alpha< \frac{\dpoin(J)}{\dpoin(J)+\mmax}$
where $\dpoin(J)$ is the Poincar\'e exponent of the Julia set 
then there exists a unique, ergodic, and  non-atomic conformal 
measure $\nu$ with
exponent $\dpoin(J)=\HD(J)$.
If $F$ is polynomially summable with the exponent $\alpha$,
$\sum_{n=1}^{\infty}
n |(F^{n})'(F^{n_{c}}(c))|^{-\alpha}<\infty$ 
and $F$ has no parabolic periodic cycles, then $F$
has an absolutely continuous invariant measure
with respect to $\nu$. This leads also to  a new result
about the existence of absolutely continuous invariant 
measures for multimodal maps of the interval.

We prove that if $F$ is summable with an exponent 
$\alpha< \frac{2}{2+\mmax}$ then the Minkowski dimension of $J$
is strictly less than $2$ if $J\not=\C$ and $F$ is unstable. 
If $F$ is a polynomial or
Blaschke product then $J$ is conformally removable.  
If $F$ is summable with $\alpha< \frac{1}{1+\mmax}$
then connected components of the
boundary of every invariant Fatou component
are  locally connected. To study
continuity
of Hausdorff dimension of Julia sets,
 we introduce the concept of the 
uniform summability.  

Finally, we derive  a conformal analogue of Jakobson's 
(Benedicks-Carleson's) theorem
and prove the external continuity of the Hausdorff dimension of Julia
sets for almost all points $c$ from the Mandelbrot set with respect
to the harmonic measure.
\end{abstract}

\begin{small}
\begin{quote}
\begin{center}
{ \bf R\'esum\'e}
\end{center}
Nous  disons qu'une application rationnelle $F$ satisfait la condition
de la sommabilit\'e avec un exposant $\alpha$ si pour tout point critique
$c$ qui appartient \`a l'ensemble de Julia $J$, il y a un entier positif
$n_{c}$ tel que  $\sum_{n=1}^{\infty}|(F^{n})'(F^{n_{c}}(c))|^{-\alpha}<\infty$ 
et $F$ n'a pas de points p\'eriodiques paraboliques.
Soit $\mmax$ la multiplicit\'e maximale des points critiques de $F$.

L'objectif est d'\'etudier les s\'eries
de Poincar\'e pour une large classe  d'applications rationnelles
et  d'\'etablir les propri\'et\'es ergodiques et la regularit\'e
des mesures conformes.
Si $F$ est sommable avec un exposant 
$\alpha< \frac{\dpoin(J)}{\dpoin(J)+\mmax}$,
o\`u $\dpoin(J)$
est l'exposant de Poincar\'e de l'ensemble de Julia, alors il existe
une unique mesure conforme $\nu$ avec l'exposant $\dpoin(J)=\HD(J)$
qui est invariante, ergodique, et non-atomique. En plus, $F$
poss\`ede une mesure invariante absolument continue par rapport \`a $\nu$
pourvu que
$\sum_{n=1}^{\infty}
n |(F^{n})'(F^{n_{c}}(c))|^{-\alpha}<\infty$ 
(la sommabilit\'e de type polyn\^{o}mial) et que 
$F$ n'a pas de points p\'eriodiques paraboliques.
Cela aboutit \`a un nouveau r\'esultat sur l'existence des mesures invariantes
absolument continues pour des applications multimodales d'un intervalle.

Nous d\'emontrons que si $F$ est sommable avec un exposant  
$\alpha< \frac{2}{2+\mmax}$,
alors la dimension de Minkowski de $J$, si $J\not=\C$, est 
strictement plus petite que $2$ et $F$ est instable. 
Si $F$ est un polyn\^{o}me ou le produit de Blaschke, 
alors $J$ est conform\'ement effac\c\negthinspace able. Si $F$ est sommable
avec $\alpha< \frac{1}{1+\mmax}$,
 alors toute composante connexe de la fronti\`ere de chaque composante
de Fatou invariante est localement connexe. Pour \'etudier la continuit\'e
de la dimension de Hausdorff des ensembles de Julia, nous introduisons le concept de la sommabilit\'e
uniforme.

Enfin, nous en d\'eduisons  un analogue conforme du th\'eor\`eme de Jakobson
et Benedicks-Carleson. Nous montrons la continuit\'e externe de la dimension
de Hausdorff des ensembles de Julia pour presque tout point de l'ensemble
de Mandelbrot par rapport \`a la mesure harmonique. 
\end{quote}
\end{small}
\clearpage
\pagestyle{myheadings}
\pagenumbering{arabic}
\setcounter{page}{1}

\tableofcontents
\clearpage

\pagenumbering{arabic}
\setcounter{page}{3}

\section{Introduction}
\subsection{Overview}
The Poincar\'e series
is a basic tool in the theory of Kleinian groups. It is used
to construct and study conformal densities and dimensions of the limit set.
Patterson and Sullivan  proved 
that the critical exponent
(the Poincar\'e exponent) is equal to the Hausdorff dimension
of the limit set for Fuchsian and non-elementary geometrically finite
Kleinian groups.

We focus  on estimates of  the Poincar\'e series in rational dynamics.
From this perspective, we address the problem of 
regularity  of conformal measures. We propose to study rational maps 
satisfying the summability
condition, which requires, roughly speaking, only a polynomial
growth of the derivative along critical orbits. 
Rational maps with parabolic periodic points are
non-generic and for simplicity we exclude them from our considerations.

In the class
of rational maps satisfying the summability condition,
we prove the counterpart of Sullivan's result that
conformal measures with minimal exponent are ergodic (hence unique)
and non-atomic.
To pursue properties of the Poincar\'e series for rational maps
we introduce a  notion of a  restricted Poincar\'e series
which is also well-defined for points in the  Julia set. 
This notion leads to new estimates, particularly implying 
that the convergence property of the Poincar\'e series is ``self-improving.'' 
This turns out to be an underlying reasons for 
regularity properties of conformal measures  on Julia sets. 
Also, the divergence of the Poincar\'e series with 
the Poincar\'e exponent (infimum of exponents
with converging Poincar\'e series) 
is an immediate consequence. A different definition of the Poincar\'e
exponent and its relation with various dynamical dimensions can be found
in \cite{przytycki-exponent}.

One of the central problems
in the theory of iteration of rational functions is to estimate
the Hausdorff dimension of Julia sets which are not the whole sphere
and investigate their fractal  properties.
It is believed that rational functions with metrically small
Julia sets  should possess certain weak expansion property. 
We prove that the Poincar\'e exponent
coincides with the Hausdorff dimension of the Julia set $J$ 
and $\HD(J)<2$ unless $J=\CC$ for rational
functions satisfying the summability condition with an exponent
$\alpha< \frac{2}{\mmax+2}$. These results bear some relationship
with a recent result of C. Bishop and P. Jones \cite{bijo} which says
that for finitely generated Kleinian groups
if the limit set has zero area then the Poincar\'e exponent
is equal to the Hausdorff dimension of the limit set.

Perhaps, the most famous problem in the iteration theory of rational
functions is  whether a given system 
can be perturbed to a hyperbolic one or not. It is
widely believed that this should be possible (the Fatou conjecture),
at least in the class of polynomials.  
It is well known, \cite{MSS},
that if the Julia set of a polynomial is of Lebesgue measure
zero then the polynomial can be perturbed to a hyperbolic one.
In general, despite much effort, a very limited progress was achieved
towards proving the Fatou conjecture. The real Fatou conjecture was proved
in  \cite{KozShenStrien} .
 We use the recent result of \cite{josm} to prove strong instability of polynomials satisfying 
the summability condition. This  both strengthens and generalizes
the results of \cite{przytycki-rohde-rigidity} in the class of polynomials 
and Blaschke products.  
 
The summability condition was
proposed in one-dimensional real unimodal dynamics \cite{nost}
as a weak condition which  would  guarantee
the existence of absolutely continuous measures with
respect to the one-dimensional Lebesgue measure.
M. Aspenberg proved in ~\cite{Aspenberg} that the class of 
rational functions which satisfy the Collet-Eckmann condition
is of positive Lebesgue measure in the space of all rational functions
of a given degree (see also a closely related paper of M. Rees ~\cite{rees})

From the point of view of
measurable dynamics and ergodic theory, the existence of regular
invariant measures is of crucial importance. A dynamical analogue
of the one-dimensional Lebesgue measure on Julia set is given by
the ``geometric measures'' 
(conformal measures with minimal exponents). 
We study regularity and ergodic properties of conformal
measures and determine whether the dynamics admits the existence
of absolutely continuous invariant measures with respect to a given
conformal measure. The problem is twofold and involves
both  dynamical  and measure theoretical estimates.

Another problem we look at is local connectivity of Julia sets and
the existence of wandering continua. In order to pursue the continuity
of the Hausdorff dimension of Julia sets we introduce a uniform  convergene
for rational maps satisfying the summability condition. 
Finally, we discuss applications of our theory to the study of 
 complex unicritical polynomials $z^{d}+c$.
In this setting, we formulate  a complex analogue of Jacobson
and  Benedicks-Carleson's theorems.    

\paragraph{Non-uniform hyperbolicity.}
The concept of non-uniform hyperbolicity is slightly vague and depends
on varying backgrounds and motivations. 
It is difficult to find a single formulation of this property. 
Our approach emphasizes  measure theoretical aspects of the system,
which should be hyperbolic on the average. 
Loosely speaking, given a non-hyperbolic system,
 one tries to make it hyperbolic by taking only pieces of the phase
space and a high iterate of the map on each piece. If it is possible
to find such  pieces almost everywhere, we say that the system
induces hyperbolicity or is non-uniformly hyperbolic
with respect to a given measure. 
Of course, we are interested only in natural measures such 
as the Lebesgue measure when $J=\CC$ and geometric measures 
(see Definition~\ref{def:mes} and the following discussion) 
when $J\not =\C$.
This approach originates from the work of Jakobson~\cite{jak} on the abundance
of absolutely continuous invariant measures and was also followed
in a similar way by Benedicks and Carleson~\cite{Beca,Becaa}. 
The concept of induced hyperbolicity plays also 
a central role in the proof of the real Fatou conjecture, see ~\cite{book}.

For rational functions
$F$ satisfying the summability condition
we prove induced hyperbolicity  with respect 
to unique geometric measure on $J$
(Theorem~\ref{theo:largescale}). 
The induced hyperbolicity yields that the Julia set is
of Lebesgue measure zero whenever $J\not=\CC$ , see also~\cite{BruinStrien}.
In many cases (e.g. under polynomial summability condition)
 we prove a stronger version of non-uniform hyperbolicity, namely 
the existence of a unique absolutely continuous invariant measure $\sigma$
with respect to the geometric measure. The measure $\sigma$ is ergodic,
mixing, and has positive Lyapunov exponent.


\subsection{Main concepts and statements of results}

\paragraph{Summability conditions and maximal multiplicity.}
Before stating our main theorems,
we make a few technical remarks.
For simplicity we assume that no critical point
belongs to another critical orbit.
Otherwise all theorems remain valid with the following amendment:
a ``block'' of critical points
$$F:~~c_1~{\mapsto}~\dots~{\mapsto}~c_2
~{\mapsto}~\dots~\dots
~{\mapsto}~c_k~,$$
of multiplicities $\mu_1,\,\mu_2,\,\dots,\,\mu_k$
enters the statements as if it is a single
critical point of multiplicity $\prod\mu_j$.

If the Julia set is not the whole sphere,
we use the usual Euclidean metric on the plane, changing the coordinates 
by a M\"obius transformation so that $\infty$ belongs to a periodic Fatou
component, and doing all the reasoning on a large compact containing 
the Julia set. 
Alternatively (and also when $J=\CC$) one can use the spherical metric. 

Define $\sigma_n\,:=\,\min_{c\in\Crit}\brs{\abs{\br{F^n}'(Fc)}}$,
where minimum is taken over all critical points in the Julia set
whose forward orbits do not contain other critical points.
Many properties will take into account $\mmax$ -- the maximal multiplicity
of critical points in the Julia set
(calculated as above, if there are any critical orbit relations).

\begin{defi}
Suppose that $F$ is a rational function without parabolic periodic
points. We say that $F$  satisfies the {\em summability condition} 
with exponent $\alpha$ if
$$\sum_{j=1}^{\infty}\,(\sigma_j)^{-\alpha}~<~\infty~.$$
If a stronger inequality
$$\sum_{j=1}^{\infty}\,j\cdot(\sigma_j)^{-\alpha}~<~\infty~,$$
holds, then we say that $F$ satisfies the {\em polynomial summability condition}
with exponent $\alpha$.
\end{defi} 
We recall that $F$ satisfies the {\em Collet-Eckmann condition} if there exist
$C>0$ and $\Lambda>1$ such that 
$\sigma_{n}\geq C\Lambda^{n}$ for every positive $n$.
Contrary to the Collet-Eckmann case, the 
summability condition allows strong recurrence of the critical points.
Generally, it is not true that the critical value of a 
summable rational map has infinitely many univalent passages to 
the large scale (counterexample given e.g. by a quadratic Fibonacci polynomial), 
compare Theorem~\ref{theo:largescale}.

\paragraph{Poincar\'e series.}
We call a point $z$
admissible if it does not belong to $\bigcup_{i=0}^{\infty}F^{n}(\Crit)$. 
Take an admissible point $z$ and assume that $F$ has
no elliptic Fatou components and $J\not = \hat{\C}$.
We  define the Poincar\'e series by
$$\Sig_\delta(z)~:=
~\sum_{n=1}^\infty~\sum_{y\in F^{-n}z}\abs{\br{F^n}'(y)}^{-\delta}~.$$
The series converges for every $\delta > \dpoin(z)$ and
the minimal such $\dpoin(z)$ is called the Poincar\'e exponent 
(of $F$ at the point $z$).
By standard distortion considerations, if ${\cal F}$ is a component of the Fatou set,
then for all admissible $z\in {\cal F}$
Poincar\'e exponents coincide, so we set
\[\dpoin({\cal F}):=\dpoin(z)~.\] 
 We define the {\em Poincar\'e exponent} by
$$\dpoin(J)~:=~\max~\brs{\dpoin({\cal F})\, },$$
(As Theorem~\ref{theo:poincare} shows,
 we can alternatively set $\dpoin(J):=\inf\{\dpoin(x):\,x\in\CC\}$).
A well-known area estimate shows that $\dpoin(J)\leq 2$.
A natural question arises if the Poincar\'e exponents $\dpoin({\cal F})$
in different Fatou components coincide and if $\dpoin(J)<2$.
By the analogy with the theory of Kleinian groups, we say that $F$ 
is of {\em divergent} ({\em convergent}) {\em type} if
the Poincar\'e series $\Sig_\dpoin(z)$ diverges (converges) 
for every component ${\cal F}$ of the Fatou set and
every admissible $z\in {\cal F}$.
Clearly, hyperbolic rational maps satisfy the summability condition
with any positive exponent $\alpha$. By Theorem~\ref{theo:poincare},
they are all of divergent type.
In general,  rational maps of  the divergent type can be viewed
as weakly hyperbolic systems. It would be interesting to study this
property from an abstract point of view. 

To address the above problems, we consider the Poincar\'e series 
as a function of $z\in \hat{\C}\setminus J$ and
study its limiting behavior when  $z$ approaches the Julia set.
We use the concept of a restricted Poincar\'e series to study
the dynamics of inverse branches of $F$ independently from the fact  whether
their domains intersect $J$ or not.

Let $\xx{}{z,\Delta}$ be the set of all preimages of $z$
such that the ball $B(z,\Delta)$ can be pulled back univalently
along the corresponding branch.
\begin{defi}
The {\em restricted Poincar\'e series} for a number $\Delta>0$
and $z\in\CC$ is defined by
$$
\Sigma^{\Delta}_\delta(z)~:=
~\sum_{n=1}^{\infty}~\sum_{y\in F^{-n}(z)\cap\xx{}{z,\Delta}}\,\abs{\br{F^n}'(y)}^{-\delta}~.
$$
\end{defi}
The definitions of the Poincar\'e exponents assume 
that the complement of $J$ is non-empty.
Should $J$ coincide with the whole sphere,
we set $\dpoin(J)\,:=\,2$.
We will prove that under the summability condition
the convergence of the Poincar\'e series
is an open property. This means that  $F$ is of divergent type.
The proof will use the restricted Poincar\'e series and 
a generalized ``area estimate.''  
Intriguingly, our technique allows us to compare
different $\dpoin({\cal F})$ through perturbations of the Poincar\'e
series near the critical points $c\in J$ of the maximal multiplicity. 
These points appear to be in the stability locus of $\dpoin(z)$.
 
\begin{theo}\label{theo:poincare}
Suppose that $F$ satisfies the summability condition with an exponent
$$\alpha~<~\fr{\dpoin(J)}{\mmax+\dpoin(J)}~.$$ We have that
\begin{itemize}
\item the Poincar\'e series 
with the critical exponent $\dpoin(J)$ diverges
for every  point $z\in \CC$,
\item if $z$ is at a positive distance to  the critical orbits 
and $c$ is a critical point of the maximal multiplicity then
$$\dpoin(z)~=\dpoin(c)~=~
\dpoin(J)~=~
\inf\left\{\dpoin(x):\,x\in\CC\right\}~.$$
\end{itemize}
\end{theo}
If $J=\CC$ then by our convention, $\dpoin(J)=2$. Hence, 
the equality $$\dpoin(J)=
\inf\left\{\dpoin(x):\,x\in\CC\right\}~$$
 of Theorem~\ref{theo:poincare} can
be regarded as an alternative  definition of the Poincar\'e exponent
when $J=\CC$.

\paragraph{Conformal and geometric measures.}
Conformal or {\em {Patterson-Sullivan}} measures are  dynamical analogues
of  Hausdorff measures and capture important (hyperbolic) 
features of the underlying dynamics.   
\begin{defi}\label{def:mes}
Let $F$ be a  rational map with the Julia set $J$.
A Borel measure $\nu$ supported on $J$ 
is called {\em conformal with an exponent} $p$
(or {\em $p$-conformal}) if for every Borel set
$A$ on which $F$ is injective one has
\[ \nu(F(B))=\int_{B}|F'(z)|^{p}~d\nu~.\] 
\end{defi}
As observed in \cite{sullivan-rio}, 
the set of pairs $(p,\nu)$ with $p$-conformal measure $\nu$
is compact (in the weak-$*$ topology). 
Hence, there  exists  a conformal measure with the {\em minimal exponent} 
\[\dconf~:=~\inf\{p: \exists\; \mbox{a $p$-conformal measure on $J$}\}.\]  
The minimal exponent $\dconf$ is also called 
a {\em conformal dimension} of $J$.

However, conformal measures might have extremely bad analytical
properties, in particular  they can  be atomic. In this context 
it is rather surprising that  the most important conformal measures, 
namely these with minimal exponents, have many good analytical properties 
in the class of rational maps which satisfy the summability condition.

A hyperbolic Julia set has the Hausdorff dimension strictly
less than $2$ and a finite positive Hausdorff measure in its dimension, 
\cite{sullivan-rio}. 
In the hyperbolic case,  there is a unique conformal measure which coincides
with the normalized Hausdorff measure.
For non-hyperbolic maps the situation is much more complicated. 
A construction of conformal measures for Kleinian groups
was proposed by Patterson. The same construction was implemented
by Sullivan for rational functions \cite{sullivan-rio}. 
In Patterson-Sullivan approach  conformal measures are constructed
through the  dynamics in the Fatou set.
This external construction favors conformal measures with ``inflated'' 
exponents and can be briefly summarized as follows: 
Assume that $J\not =\hat{\C}$ and  $F$ has no neutral cycles. 
Let $z\in \hat{\C}\setminus J$.
If $\Sig_\dpoin(z)$ diverges then, for
any $p>\dpoin(z)$ we consider $v_{p}$ defined by 
\begin{equation}\label{equ:po1}
\nu_{p}:=  
\frac{1}{\Sig_p}\;
\sum_{n=1}^{\infty}\sum_{y\in F^{-n}(z)} 
\frac{1_{y}}{|(F^{n})'(y)|^{p}}~,
\end{equation}
where $1_{y}$ is a Dirac measure at $y$.
A weak limit of such atomic measures
when $p\to\dpoin(z)$ defines a $\dpoin(z)$-conformal measure on the Julia set. 
If $\Sig_\dpoin(z)$ converges then the construction can be repeated
in the same way after multiplying each term of $\Sig_{p}$ 
by a factor $h(|(F^{n})'(y)|^{-p})$, where $h$ is the Patterson function.
The function $h$ tends to $1$ subexponentialy
but still makes  the series  $\Sig_p$ divergent for $p=\dpoin(z)$. 
Clearly,
$$\dconf\leq \dpoin\;.$$

An alternative ``internal'' construction  
was carried out by Denker and Urbanski, 
\cite{denker-urbanski-sullconf,denker-urbanski-existconf}. It uses increasing
 forward invariant subcompacts inside the Julia set, free of critical points,
to distribute  probabilistic Borel measures $\nu_{n}$ with 
Jacobians bounded respectively by $|F'(z)|^{p_{n}}$. 
In the limit, these approximating measures become
conformal with exponent equal to $\sup p_{n}=\dconf$, 
\cite{denker-urbanski-sullconf,denker-urbanski-existconf}. 
Recall that the {\em hyperbolic dimension} $\HH(J)$ 
of the Julia set $J$ is equal to the supremum of the Hausdorff dimensions
of all hyperbolic subsets of $J$.
An important geometric consequence of the Denker-Urbanski
construction  is
\[\dconf=\HH(J)~\;.\]
If additionally $\dconf=\HD(J)$ 
then  every  $\dconf$-conformal measure $\nu$ is called 
a {\em geometric measure}. 
It is an important open question, whether $\dconf=\HD(J)$ always holds,
and in general very little is known
about the existence and regularity properties of geometric
measures (cf. \cite{shishikura,przytycki-ce,urbanski-bams}).
It is not even known if a geometric  measure  has to be unique and non-atomic. 
A possible  pathology is due to  
the presence of  critical points in $J$ of convergent type. Indeed, if 
$\Sig_{p}(c)$ converges for $p\leq 2$ and $c\in \Crit$ then
$\nu_{p}$ defined by (\ref{equ:po1}) with $y=c$ is a $p$-conformal atomic
measure.

We prove that if a rational function satisfies the 
summability condition with an exponent $\alpha< \frac{2}{2+\mmax}$, 
then, for every $p >\dconf$, there exists an atomic $p$-conformal
measure with an atom at a critical point. 
In contrast to that, the geometric measures are non-atomic in this class.
The  majority of work in the area is based on Denker-Urbanski construction.
We come back to the origins and focus on the Patterson-Sullivan approach.
 
Every conformal measure of  $F$ has an exponent $p\in [\dconf, \infty)$.
The set of all such $p$ forms a {\em conformal spectrum} of $F$. 
We distinguish an atomic  part of the spectrum consisting of
all $p\in [\dconf,\infty)$ for which there exist only
atomic conformal measures, a {\em continuous} part corresponding
to exponents for which there exist only non-atomic conformal measures, and
a {\em mixed part} (possibly empty) gathering all $p$
for which there exist both atomic and non-atomic $p$-conformal measures.

\begin{defi}
We say that the conformal spectrum of $F$
is hyperbolic if its mixed part is empty and continuous part is equal
to  $\{\dconf\}$.
\end{defi}
\begin{theo}\label{theo:4}
Suppose that a rational function $F$ satisfies the summability condition
with an exponent 
\[\alpha<\frac{\dpoin(J)}{\mmax+\dpoin(J)}\;.\] 
Then 
\begin{itemize}
\item there is a unique
 non-atomic conformal measure $\mu$,
\item  measure  $\mu$ is ergodic and its exponent
is equal to 
$\dconf=\dpoin(J)$,
\item if $F$ has critical points in the Julia set  
then for every $\delta >\dconf$ 
 there exists an atomic $\delta$-conformal 
measure supported on the backward orbits
of the critical points,
\item every conformal measure has no atoms outside of the set of the 
critical points of $F$.
\end{itemize}
\end {theo}
In particular, every rational function which satisfies the summability
condition with an exponent $\alpha<\frac{\dpoin(J)}{\mmax+\dpoin(J)}$
has a hyperbolic conformal spectrum. Since every hyperbolic rational map has
no critical points in the Julia set, its conformal spectrum is trivial
and consists of only one point $\{\dconf\}$. 

Observe that if we know that every geometric measure is ergodic,
then in fact it must be unique. Indeed, if there were two such
measures $\nu_{1}$ and $\nu_{2}$, then $\nu_{3}:=\frac{\nu_{1}+\nu_{2}}{2}$
is obviously a non-ergodic geometric measure, a contradiction.

The problem of ergodicity of conformal measures was raised before.
In  \cite{przytycki-ce} it is proved that a number of ergodic
components of conformal measures $\nu$ for the so-called Collet-Eckmann rational
maps is finite and not exceeding the number of critical points. This is not
enough to conclude uniqueness of a  geometric measure.
The assertion of Theorem~\ref{theo:4} is valid only for conformal measures
with minimal exponent. In general, if $p>\dconf$
and $F$ has more than one  critical point of the maximal multiplicity, then
there exist  non-ergodic $p$-conformal measures.
Indeed, by Theorem~\ref{theo:poincare}, if $\mu(c)=\mmax$ then  $\dpoin(c)=\dconf$. 
Measures $\nu(c)$ defined by (\ref{equ:po1}) are $p$-conformal
and their convex sum  
has exactly $\#\{c: \mu(c)=\mmax\}$ ergodic components.
\paragraph{Induced hyperbolicity.}
Consider the  set $\Jlso$ of all points $x\in J$ which 
$\epsilon$-frequently  go to
the large scale of radius $r$, namely: 
$$\exists\, n_j\to\infty:~F^{n_j} 
\mbox{~is ~univalent~on ~} F^{-n_j}\br{B\br{F^{n_j}x,r}},
~\abs{\br{F^{n_{j+1}}}'(x)}\,<\,\abs{\br{F^{n_{j}}}'(x)}^{1+\epsilon}~,$$
where $F^{-n_j}\br{B\br{F^{n_j}x,r}}$ stands for its 
connected component, containing $x$.

\begin{theo}[Strong induced hyperbolicity]\label{theo:largescale}
Suppose that a rational function $F$ satisfies the summability condition
with an exponent
$$\alpha~<~\fr{\dpoin(J)}{\dpoin(J)+\mmax}~.$$
Then there exists $r>0$ so that for every $\epsilon>0$ almost every
point with respect to a  unique $\dconf$-conformal measure $\nu$ 
goes $\epsilon$-frequently to the large scale of diameter $r$.
\end{theo}

\paragraph{Invariant measures.}
The summability condition was introduced in real  dynamics in
\cite{nost} to study absolutely continuous invariant measures (shortly acim)
with respect to the Lebesgue measure.
In the conformal setting when conservative dynamics is often
concentrated on fractal sets with zero area, a concept of an acim
encounters some hurdles. 
In this situation, it is natural to study absolutely continuous invariant
measures  with respect to conformal measures.
This approach was already adopted by Przytycki, who proved
that  a rational Collet-Eckmann map has an acim with respect
to a $p$-conformal measure $\nu$, provided that $\nu$ is regular
enough along critical orbits \cite{przytyk2}. 
This regularity was verified only in particular cases 
(cf. Tsujii's condition in \cite{przytyk2})
and can be expressed by  an  integral condition 
(in \cite{przytycki-ce}  a slightly weaker condition was considered):
there exists $C>0$ so that for all $i>0$
\begin{equation}\label{def:cond}
\sup_{c\in \Crit} \int
|z-F^{i}(c)|^{-p\br{1-\frac{1}{\mmax}}}\;d\nu< C\; .
\end{equation}
The scope of validity of this condition is not known even
in the Collet-Eckmann setting.
We will call this condition the {\em integrability condition}
with an exponent $\eta:= p(1-\frac{1}{\mmax})$.

Surprisingly, we do not need to assume this condition to obtain an absolutely
continuous invariant measure with respect to a non-atomic conformal
measure, if $F$ satisfies the polynomial summability condition.

\begin{theo}\label{theo:inv}
Suppose that a rational function $F$ satisfies the polynomial summability
condition with an exponent
\[ \alpha< \frac{\dpoin(J)}{\dpoin(J)+\mmax}~.\]
Then  $F$ has a unique absolutely continuous invariant measure $\sigma$
with respect to a unique  $\dpoin(J)$-conformal measure $\nu$.
Moreover, $\sigma$ is ergodic, mixing, exact, has positive entropy and Lyapunov exponent.
\end{theo}
If the integrability condition is satisfied, 
then we have the following theorem.
\begin{theo}\label{theo:integ}
If $F$ satisfies the summability condition with an exponent
\[\alpha < \frac{\dpoin(J)}{\dpoin(J)+\mmax}~,\]
and a unique  $\dpoin(J)$-conformal measure $\nu$ is 
$\dpoin(J)(1-\frac{1}{\mmax})$-integrable then $F$
has a unique and ergodic absolutely continuous invariant measure.
\end{theo} 
Observe that $2$-dimensional Lebesgue measure becomes a 
geometric (conformal) measure, when $J=\CC$. 
We conclude, that in this case $F$ has an absolutely
continuous invariant measure with respect to 2-dimensional
Lebesgue measure if it satisfies
the summability condition with an exponent $\alpha< \frac{2}{2+\mmax}$.
We also study ergodic and regularity properties of the absolutely
continuous invariant measures.
\paragraph{Multimodal maps.}
Another important application of our techniques lies in
the dynamics of  analytic {\em  multimodal maps}
of a compact interval with non-empty interior. 
Contrary to the unimodal case (maps with exactly one local extremum), 
there are few results about the existence of absolutely continuous
invariant measures, \cite{BruinLuzzattoStrien}.
The only available results are either of perturbative nature 
(analogues of Jakobson's result in the quadratic family, i.e. they
require some transversality assumptions in one-parameter families)
or  impose some semi-unimodal  conditions.  
One of  the most general results  \cite{tsuji} states that
for a generic $C^{2}$ families of multimodal maps
there exist a positive set of parameters which correspond to Collet-Eckmann
maps with acims. In \cite{nost} it was proved that if  
S-unimodal map $f$ (i.e. unimodal with negative Schwarzian derivative),
and the critical point of  multiplicity $d$ 
satisfies the summability condition with the exponent $1/d$ then 
it has an absolutely continuous invariant measure 
with respect to the Lebesgue measure (recently a stronger result was
proved in \cite{BruinShenStrien}). 
The Schwarzian derivative is defined for a $C^3$ function $f$ by
\begin{equation}\label{eq:schdef}
S(f)(x) := f'''(x)/f'(x) - \frac{3}{2}\left(f''(x)/f'(x)\right)^2\;,
\end{equation}
provided $f'(x) \neq 0 $.
A prototype $S$-unimodal map is given by 
$z^{2d}+c$, with $c\in \R$ and $d$ a positive integer.

An absolutely continuous invariant measure
provides a useful information about statistical behavior of orbits.
We prove the following result for multimodal maps.
\begin{theo}\label{coro:uni}
Let $f$ be an analytic function of the unit interval with
all periodic points repelling and negative Schwarzian derivative.
If $f$ satisfies the summability condition with an exponent 
\[\alpha< \frac{1}{1+\mmax}~, \]
then $f$ has an absolutely continuous invariant measure
$\sigma$ with respect to $1$-dimensional Lebesgue
measure. 
\end{theo}

\paragraph{Fractal dimensions.}
We also show that for the Julia sets under considerations
are ``regular'' fractals,
in the sense that all possible dimensions coincide.

The Hausdorff dimension of a measure $\nu$ is defined as
the infimum of Hausdorff dimensions of its Borel supports:
\[\HD(\nu) := \inf\{\HD(A): \nu(A)=1\}~.\]
For the convenience of the reader we define and discuss 
basic properties of  Hausdorff, Minkowski, and Whitney dimensions
$\HD$, $\MD$, $\dwhit$
in Section 8.


\begin{theo}\label{theo:dims}
Suppose that a rational function $F$ 
satisfies the {\em summability condition}
with an exponent
$$\alpha~<~\fr{p}{\mmax+p}~,$$
where p is any (e.g., maximal) of the quantities in the formula below. Then
$$\dpoin(J)=\dwhit(J)=\MD(J)=\HD(J)=\HH(J)=\HD(\nu)~,$$
where $\nu$ is a unique $\dconf$-conformal measure.
\end{theo} 

\begin{coro}
Under the hypothesis of Theorem~\ref{theo:4},
the unique $\dconf$-conformal measure $\nu$ is a geometric measure.
\end{coro}
A natural question arises: 
does the claim of Theorem~\ref{theo:dims} remain valid  under any summability condition? 
This is an interesting question with possible application towards establishing
non-atomicity of conformal measures with minimal exponent.

\begin{coro}[Strong Ahlfors dichotomy]\label{theo:dim}
Suppose that a rational function $F$ satisfies the {\em summability condition}
with an exponent
$$\alpha~<~\fr{2}{\mmax+2}~.$$
Then the Julia set is either the whole sphere,
or its Minkowski dimension
is strictly less than $2$.
\end{coro}

\paragraph{Unstable rational maps.}
It is known 
(by an application of $\lambda$-lemma, see \cite{MSS}), 
that the affirmative answer to the dynamical Ahlfors conjecture 
(Julia set is either the whole sphere or of zero area)
in the class of rational functions with $J\neq\CC$ implies the Fatou conjecture.
If $F$ satisfies the summability condition with an exponent
$\alpha\leq \frac{2}{2+\mmax}$ then, by
Theorem~\ref{theo:poincare}, the Julia set has zero area,
and cannot carry a non-trivial  Beltrami differential. 

\begin{defi}
We say that a set $J$ is {\em conformally removable} if
every homeomorphism $\phi$ of $\CC$ which is holomorphic
off $J$, is in fact a M{\"o}bius transformation.
\end{defi}
For Julia sets, this is a very strong property, 
which generally does not hold even for hyperbolic rational maps. 
A counterexample, which is topologically a Cantor set of circles
is constructed in \S11.8 of the book \cite{beardon-book-iteration}.
Using a recent work \cite{josm},
we can establish conformal removability
(also called holomorphic removability)
of Julia sets for polynomials and 
Blaschke products satisfying the summability condition. 
\begin{theo}\label{theo:rem}
If a polynomial 
$F$ satisfies the {\em summability condition}
with an exponent
$$\alpha~<~\fr{2}{\mmax+2}~,$$
then the Julia set is conformally removable.
\end{theo}
More generally, the theorem above holds
not only for polynomials, but for rational functions
such that the Julia set is a boundary of one of the Fatou components.

The assumption above, that the Julia set coincides with the 
boundary of one of the Fatou components is essential
for conformal removability.
A more flexible concept of dynamical removability might hold
for all rational Julia sets.
\begin{defi}
We say that a Julia set $J_F$ is dynamically removable if
every homeomorphism $\phi$ of $\CC$ 
which conjugates $(\CC,F)$ with another rational dynamical system  $(\CC,H)$
and is quasiconformal off $J_F$ is globally quasiconformal.
\end{defi}
\begin{theo}\label{theo:rem1}
If a rational function $F$ satisfies the summability condition with an exponent
$$\alpha~<~\fr{1}{\mmax+1}~,$$
then the Julia set is dynamically removable.
\end{theo}

\begin{theo}\label{theo:rem2}
Let a rational function $F$ satisfy the 
summability condition with an exponent
$$\alpha~<~\fr{2}{\mmax+2}~.$$
Suppose that there is a quasiconformal homeomorphism $\phi$ of $\CC$
which conjugates rational dynamical systems $(\CC,F)$ and $(\CC,H)$.
\begin{itemize}
\item
If $J\neq\CC$, then the Beltrami coefficient
$\phi_{\mu}$ has to be supported on the Fatou set.
\item
If $J=\CC$, then either $\phi$ is a M\"obius transformation,
or $F$ is double covered by an integral torus endomorphism
(i.e. it is a Latt\'es example).
In the latter case the Beltrami coefficient $\mu_\phi$
lifts to a constant Beltrami coefficient
on the covering torus.
\end{itemize}
\end{theo}

\begin{coro}[No invariant line fields]
If a rational function $F$ satisfies the summability condition with an exponent
$$\alpha~<~\fr{2}{\mmax+2}~,$$
then $J$ carries no invariant line field, except when $F$ is double
covered by 
an integral torus endomorphism.
\end{coro}
Compare Corollary~3.18 in \cite{mcm}.

\paragraph{Geometry, local connectivity, 
and non-wandering continua.}
If  $z^{\ell}+c$ is a unicritical polynomial with locally
connected Julia set $J$ then the dynamics on $J$ has a particularly
simple representation: it is semi-conjugate to 
the multiplication by  $\ell$ modulo  $1$ on $[0,1)$. 
The quest for local connectivity of polynomial Julia sets 
dates back to the mid-Eighties. Recently, local connectivity
was obtained for all real unicritical polynomials $z^{l}+c$ with
connected Julia sets, \cite{levin-vanstrien},
and for Collet-Eckmann polynomials and Blaschke products with connected
Julia sets, \cite{ceh}. 

The quasihyperbolic distance $\qhdist{y,z}$
between points $y,z\in\Cal F$ is defined as the infimum of
$$\qhdist{y,z}~:=~\inf_{\gamma}~\int_{\gamma}~
\fr{\abs{d\zeta}}{\dist{\zeta,\partial\Cal F}}~,$$
the infimum taken over all rectifiable curves $\gamma$
joining $y$ and $z$ in $\Cal F$.

\begin{defi}\label{def:integral}
We will call a (possibly non-simply-connected) domain $\Cal F$ {\em integrable},
if there exists $z_0\in\Cal F$ and an integrable function $H : \R_{+}\to \R_{+}$,
\[\int_{0}^{\infty}H(r)<\infty~,\]
such that $\Cal F$ satisfies the following  {\em quasihyperbolic boundary condition}:
$$\dist{z,\partial\Cal F}~\leq ~H(\qhdist{z,z_0})~,$$ 
for any $z\in\Cal F$.
The distance $\dist{z,\partial \Cal F}$ is computed in the spherical metric.
\end{defi}
H\"{o}lder domains correspond to ``exponentially fast integrable'' domains with
$H(r)=\exp(C-\epsilon r)$. We will show
that the Fatou components of rational maps satisfying the
summability condition are integrable domains.

A concept of wandering subcompacta of connected Julia sets
is directly related to the local connectivity of components of $J$.
We say that a compact set $K$ is wandering if for every $m,n>0$,
$F^{n}(K)\cap F^{m}(K)=\emptyset$ whenever $m\neq n$.

\begin{theo}\label{prop:integrable}
Let $F$ be a rational function which satisfies the summability condition
with an exponent 
\[\alpha\leq \frac{1}{\mmax+1}~.~\]
Then every periodic Fatou component $\Cal F$ is an integrable domain.
If $F$ has a fully invariant Fatou component
then every component of $J$ is locally connected and $F$ does not have
wandering continua. 
\end{theo}


\paragraph{Uniform summability and  continuity of dimensions.}
We also study continuity properties of the Hausdorff dimension
of the Julia set as a function of $F$: $F\mapsto \HD(J_{F})$.
To this aim we consider  a certain class of  perturbations
of a rational map $F$ which satisfy the  summability condition in a  uniform way.
Since perturbations usually let critical
points escape from the Julia set, we need to take into 
account critical points of $F$ which do not belong to $J_{F}$.

Given a rational function $F$ and an  $\epsilon$-neighborhood 
$B_{\epsilon}(J)$ of its Julia set $J$ we define 
for every $c\in Crit\cap B_{\epsilon}(J)$
an {\em escape time} 
\[ E(\epsilon)=\inf\{j\in \N:\,F^{j}(c) \not \in B_{\epsilon}(J)\}\;.\]
If $F^{j}(c)\in B_{\epsilon}(J)$ for all $j$, we set $E(\epsilon):=\infty$.

\begin{defi}\label{def:unif}
Let $F$ be a rational function satisfying the summability
condition with an exponent $\alpha$. We say that a sequence 
$(F_i)$ of rational functions converges  $S(\alpha)$- uniformly to $F$ if:
\begin{enumerate}
\item  $F_i$ do not have parabolic periodic orbits and tend to $F$ uniformly on the Riemann sphere, 
\item  there exist $M,\epsilon >0$ so that for every $i$ large
enough  and  every $c\in \Crit_{F_i}\cap B_{\epsilon}(J_{F_i})$,
\begin{equation}\label{unisum}
\sum_{j=1}^{E(\epsilon)} |(F_i^{j})'(c)|^{-\alpha} < M\;.
\end{equation}

\item $\#Crit_{F}=\#Crit_{F_{i}}$ for $i$ large enough
(the critical points are counted without their multiplicities).
\end{enumerate}
\end{defi}
The notion of the $S(\alpha)$-uniform convergence  
involves only these critical points $c$  of $F_i$ which
``asymptotically'' belong to the Julia set  $J_{F}$. The 
$S(\alpha)$-uniform convergence demands  $1-1$ correspondence between 
the critical points of $F$ and $F_{i}$ for large $i$.

Theorem~\ref{theo:contin} establishes  continuity of the Hausdorff (Minkowski)
dimension for the $S(\alpha)$-uniform convergence.  

\begin{theo}[Continuity of Hausdorff Dimension]\label{theo:contin}
Suppose that a rational function $F$ satisfies the  summability condition with an exponent 
\[\alpha<\frac{\dpoin(J)}{\mmax+\dpoin(J)}\;,\]
and $(F_i)$ is a sequence of rational functions tending $S(\alpha)$-uniformly to $F$.
Then,  
\[\lim_{i \rightarrow \infty}\HD(J_{F_i})=\HD(J_{F})\;.\]
\end{theo}

\paragraph{Unicritical  polynomials.}
Let ${\cal M}_{d}$ be the connectedness locus of
unicritical  polynomials $f_{c}=z^{d}+c$,
\[{\cal M}_{d}= \{c:\, |f_{c}^{n}(c)|< \infty\}\; .\]
When $d=2$, ${\cal M}_{2}$ is better known as the Mandelbrot set.
By Shishikura's theorem \cite{shishikura} it is known that the Hausdorff dimension
as a function of $c\in \C\setminus {\cal M}_2$ does not extend continuously
to $\partial M$.
Yet typically with respect to the harmonic
measure of $\partial {\cal M}_2$ a continuous extension of $\HD(\cdot)$
along hyperbolic geodesics is possible. 

\begin{defi}\label{defi:radius}
A closure $\Gamma(c)$ of a hyperbolic geodesic in 
$\C\setminus {\cal M}_{d}$ which
contains $\infty$ and a point  $c\in \partial M_{d}$
is called an external ray. If $\Gamma(c)\cap
\partial {\cal M}_{d}=\{c\}$ then we say that $\Gamma(c)$
terminates at $c$. 
\end{defi}

We use  properties of the $S(\alpha)$-convergence with 
 $\alpha<1/(1+d)$,  and
the results of \cite{graczyk-swiatek-ce, smirnov-symbce} 
to deduce the following theorem. 
\begin{theo}\label{theo:cont} 
For almost all $c$ from ${\partial {\cal M}}_{d}$ with respect to the harmonic measure, we have
\[\lim_{\Gamma(c)\ni c' \rightarrow c}\HD(J_{c'}) = \HD(J_{c})~.\]
\end{theo}
Radial continuity of Hausdorff dimension for postcritically finite 
quadratic polynomials was established in \cite{mcmullen-radial}. 
The set of postcritically finite polynomials is of zero harmonic measure,
\cite{graczyk-swiatek-ce, smirnov-symbce}.
 
Another  consequence of our estimates and \cite{graczyk-swiatek-ce, smirnov-symbce}
is a conformal analogue of Jakobson and 
Benedicks-Carleson's theorem~\cite{jak,Beca,Becaa}.
Let $f_{c}(z)=z^{d}+c$ and suppose that $f_{c}$ has a geometric measure.
We call a probabilistic measure $\mu$, supported on the Julia set of $f_{c}$,
a  {\em Sinai-Ruelle-Bowen},
or SRB for short, measure if it is a weak accumulation point of 
measures $\mu_{n}$ equally distributed along the orbits
$x, f_{c}(x),\dots, f_{c}^{n}(x)$ for $x$ in a positive geometric
measure set. 
\begin{theo}\label{theo:end}
For almost all $c\in \partial {\cal M}_{d}$ with respect to
the harmonic measure, 
\begin{enumerate}
\item there exists a unique geometric measure $\nu_{c}$
of $z^{d}+c$ which is a weak limit of the normalized Hausdorff measures of 
$J_{c'}$, $c'\in \Gamma(c)$.
\item $\nu_c$ is ergodic and non-atomic,
\item  $\HD(\nu_c)=\HD(J)$,
\item $z^{d}+c$ has an  invariant SRB measure with a positive
Lyapunov exponent  which is equivalent to the  geometric measure $\nu_c$.
\end{enumerate}
\end{theo}

\paragraph{Acknowledgments.} 
The authors are  grateful to the referee
for carefully reading our manu\-script
and contributing numerous suggestions,
which greatly improved the exposition.

Stanislav Smirnov would like to thank
Zoltan Balogh, 
Chris Bishop, 
and Pekka Koskela 
for many useful discussions.

\paragraph{Notation and Conventions.}
We will write $A\lesssim B$ whenever $A\le C B$ with some
absolute (but depending on the equation) constant $C$. If 
$A\le CB$ and $B\le CA$ then we write $A\asymp B$.
We will denote a ball of radius $R$ around $z$ by $B_R(z)$.
We adopt the convention that $\sum_n(\omega_n)^{-\infty}<\infty$ 
means that the sequence $\omega_n$ tends to $\infty$ as $n\to\infty$. 

For simplicity and readers convenience we will write all the distortion estimates 
for the planar metric, when  Koebe distortion theorem has a more familiar formulation
(Lemma~\ref{lem:koeb}).
The estimates remain valid in the case of spherical metric,
with an appropriate version of Koebe distortion theorem
(which differs only by a multiplicative constant,
since we work with the scales smaller than some very small $R$).

Another general convention is following: we call $F^{-n}(z), \dots, z$
a sequence of preimages of $z$ by $F$ if for every $1\leq j\leq n$,
$$F(F^{-j})(z)=F^{-j+1}(z)\;.$$ 
\clearpage

\part{Poincar\'e series, induced hyperbolicity, invariant measures}

\section{Expansion along orbits}\label{sec:backexpan}

Our  goal is to estimate  the derivative of  $F^{n}$ at $z$ 
in terms of the summability condition and the injectivity radius 
of the corresponding inverse branch $F^{-n}$ at $F^{n}(z)$. 
This is attained by decomposing backward orbits into pieces
which either closely follow the critical orbits
or stay away from the critical points at a definite distance.
We provide also a technical introduction to  the theory
of the Poincar\'e series for rational maps.

Proposition~\ref{prop:fatouexpan}  can be regarded as a conditional
version of induced hyperbolicity. 
In applications, we will use a stronger statement  (the Main Lemma),
which contains more technically detailed assertions. 
The proof of the Main Lemma  will supply a  procedure of decomposing any orbit into
blocks of three different types, defined rigorously in Subsection~\ref{sec:local}. 
We will show that the derivative along each block of a given type is expansive 
up to an error term which is a function of a few dynamically defined 
parameters. The main difficulty
in proving Proposition~\ref{prop:fatouexpan} is a possibility
of  accumulation of error terms. We will prove that due to
cancellations, the expansion prevails. 

After initial  preparations in Section~\ref{sec:backexpan},
the proof of  Proposition~\ref{prop:fatouexpan}  
will be concluded in Section~\ref{sec:specif}.

\begin{prop}\label{prop:fatouexpan}
Let a  rational function $F$ satisfy the {\em summability condition} 
with an exponent $\alpha~\le~1~.$
There exist $\epsilon>0$
and  a positive sequence $\brs{\omega_n}$
summable with an exponent $-\beta\,:=\,-\fr{\mmax\alpha}{1-\alpha}\,$,
meaning $\sum_{n}\,\br{\omega_n}^{-\beta}~<~\infty$~,
so that for every point $z$ from $\epsilon$-neighborhood of the Julia set and
every univalent branch of 
$F^{-n}$ defined on the ball $B_{\Delta}(z)$ with $\Delta<\epsilon$,
$$\begin{array}{ll}
\abs{\br{F^{n}}'(F^{-n}z)}~>~\Delta^{1-\m{c}/\mmax}~\omega_n~&
{\it ~if~a~critical~point~}c\in B_{\Delta}(z)~,\\
\abs{\br{F^{n}}'(F^{-n}z)}~>~\Delta^{1-1/\mmax}~\omega_n~ &
{\it~otherwise~.}
\end{array}$$ 
\end{prop}

\begin{rem}
The statement above is well-formulated
even when when $\alpha=1$, if we recall
the convention that $\sum_n(\omega_n)^{-\infty}<\infty$ 
means that the sequence $\omega_n$ tends to $\infty$ as $n\to\infty$. 
\end{rem}

\begin{rem}
In the proof we actually obtain a specific form of $\omega_n$ 
in terms of $\{\sigma_n\}$ and Ma\~n\'e's lemma.
\end{rem}

\begin{coro}
Suppose that the assumptions of Proposition~\ref{prop:fatouexpan}
are satisfied.
Then there is $\epsilon>0$ such that for any point $z$ at a distance
$\Delta <\epsilon$ from the Julia set we have
$$\abs{\br{F^{n}}'(F^{-n}z)}~\geq ~\Delta^{1-1/\mmax}~\omega_n~,$$
where $\omega_{n}$ is given by Proposition~\ref{prop:fatouexpan}.
\end{coro}

\begin{coro}\label{cor:nosiegel}
If $F$ satisfies the summability condition with an exponent $\alpha\leq 1$
then $F$ has neither Siegel disks nor Herman rings.
\end{coro}
\begin{proof}
If $z$ belongs to an elliptic Fatou component ${\cal F}$
(Siegel disk or Herman ring) then for every preimage $F^{-n}z\in {\cal F}$, 
$\abs{\br{F^n}'(F^{-n}z)}\asymp1$. This contradicts 
Proposition~\ref{prop:fatouexpan} since $\lim_{n\to\infty}\omega_n=\infty$.
\end{proof}

\subsection{Preliminaries}

\paragraph{Shrinking neighborhoods.}
To control the distortion, we will use the method of shrinking neighborhoods, 
introduced in \cite{przytycki-ce} (see also \cite{ceh}). 
Suppose that $\sum_{n=1}^{\infty}\delta_{n}<1/2$ and
$\delta_{n}>0$ for every positive integer $n$.
Set $\del_n\,:=\,\prod_{k\le n}\br{1-\delta_k}$.
Let $B_r$ be a ball of radius $r$ around a point $z$  and $\{F^{-n}z\}$ 
be a sequence of
preimages of $z$. We define
$U_{n}$ and $U'_{n}$ as the connected
components of $~F^{-n}\,B_{r\del_n}~$ and 
$F^{-n}\,B_{r\del_{n+1}}$, respectively, which contain 
$ F^{-n}z$.
Clearly,
$$FU_{n+1}~=~U_n'~\subset~U_n~.$$
If $U_k$, for $1\le k\le n$, do not
contain critical points then distortion of
$F^n\,:~U_n'\to B_{r\del_{n+1}}$ is
bounded (Koebe distortion lemma below) 
by a power of $\fr1{\delta_{n+1}}$, multiplied
by an absolute constant.

Since $\sum_n\delta_n\,<\,\fr12$, one also has 
$\prod_n\br{1-\delta_n}\,>\,\fr12$,
and hence always $B_{r/2} \subset B_{r\del_{n}}$.

\paragraph{The Koebe distortion lemma.}
We will use the following version of the Koebe distortion lemma.
\begin{lem}\label{lem:koeb}
Let  $f:B\rightarrow \C$ be  a univalent map from the unit  disk
into the complex plane. 
Then the image $f(B)$ contains a ball of radius $\frac14|f'(0)|$
around $f(0)$.
Moreover, for every $z\in B$  we have that
\[\frac{(1-|z|)}{(1+|z|)^{3}} \leq
 \frac{|f'(z)|}{|f'(0)|}\leq \frac{(1+|z|)}{(1-|z|)^{3}}~,\]
and
\[ |f(z)-f(0)|\leq |f'(z)|\frac{|z|(1+|z|)}{1-|z|}\;.\]
\end{lem}
\begin{proof}
The first statement is Corollary 1.4, 
and the next inequality is Theorem 1.3 in \cite{pomme}.

Let $M$ be a M\"obius automorphism of the unit disk which maps $0$
onto $z$ and $-z$ onto $0$. By Theorem 1.3 in \cite{pomme} applied
to $f\circ M$, we have that
\[|f\circ M(-z)-f\circ M(0)|\leq |(f\circ
M)'(0)|\frac{|z|}{(1-|z|)^{2}}\;.\]
Since $M'(0)=1-|z|^{2}$, the last claim of the lemma follows.
\end{proof}

\subsection{Technical sequences}\label{sec:techseq}

Suppose that $F$ is a rational function which satisfies the summability
condition with an exponent $\alpha$.
To simplify calculations we will introduce 
three positive sequences 
$\{\alpha_n\},\,\{\gamma_n>1\},\,\{\delta_n\}$ so that
the growth of the derivative of  $F^{n}$ will be expressed
in terms of $\gamma_{n}$, the corresponding distortion will be majorized
by $\delta_{n}$, and the constants will be controlled through $\alpha_{n}$.

\begin{lem}\label{lem:techseq} 
If a sequence $\{1/\sigma_n\}$ is summable with an exponent
$\alpha\le1$: $\sum(\sigma_n)^{-\alpha}<\infty$,
then there exist three positive sequences 
$\{\alpha_n\},\,\{\gamma_n\},\,\{\delta_n\}$,
such that
$$ \begin{array}{rlr}
\lim_{n\to\infty}\alpha_n&=~\infty~,&\\
\sum_{n}(\gamma_n)^{-\beta} &<~1/(16\deg F\cdot\mmax)~,
&\beta\,:=\,\fr{\mmax\alpha}{1-\alpha}~~,\\
\sum_{n}\delta_n~~&<~1/2~,&
\end{array}$$
and 
$$\sigma_n~\ge~(\alpha_n)^2\,(\gamma_n)^{\mmax}\,/\,\delta_n~.$$
\end{lem}
\begin{proof}
Suppose that $\alpha <1$. 
We set $$\delta_n'\,:=\,\br{\sigma_n}^{-\alpha}~,
~~\gamma_n''\,:=\,\br{\sigma_n}^{(1-\alpha)/\mmax}~,$$
and observe that 
$\sigma_n=(\gamma_n'')^{\mmax}/\delta_n'$,
$\sum_n{\delta_n'}<\infty$, $\sum_n\br{\gamma_n''}^{-\beta}<\infty$.
It is an easy exercise, that there is a sequence $\brs{\alpha_n'}$,
 $\lim_{n\to\infty}\alpha_n'=\infty$, such that $\brs{\gamma_n'}$
defined by
$$\gamma_n'\,:=\,\br{\sigma_n}^{(1-\alpha)/\mmax}
\,/\,(\alpha_n')^{-2/\mmax}~$$
is still summable with an exponent $-\beta$.
Evidently, $\sigma_n=(\alpha_n')^2(\gamma_n')^{\mmax}/\delta_n'$.
Now, choose suitable constants $C_\gamma, C_\delta$ so that 
$\{\gamma_n\}:=\{C_{\gamma}\gamma_n'\}$ and
$\{\delta_n\}:=\{C_{\delta}\delta_n'\}$  satisfy
$$\sum_n{\delta_n}~<~1/2~,~~ \sum_n{\gamma_n}^{-\beta}~<~m~.$$
Set $C_\alpha:=\sqrt{C_\delta/C_\gamma^{\mmax}}$
and let $\alpha_n :=C_\alpha\alpha_n'$. Then
$$ \lim_{n\to\infty}\alpha_n~=~\infty~,
~~\sigma_n~\ge~(\alpha_n)^2\,(\gamma_n)^{\mmax}/\delta_{n} ~. $$
The case of $\alpha=1$ can be treated similarly.
\end{proof}

\subsection{Constants and scales}\label{sec:const}

A scale around critical points
is given in terms of fixed  numbers $R'\ll R\ll1$.
We will refer to objects which stay away from the critical
points and are comparable 
with $R'$ as the objects of the {\em large scale}. 
The proper choice of $R$'s is one of the
most important elements in the analysis of expansion along  pieces
of orbits. 

We impose the following  conditions on $R$ and $R'$
(note that $\supF<\infty$, since we use the spherical metric):

\begin{description} 
\itemm{{\rm (i)}} Any two critical points
are at least $100R$ apart
and there exists a constant $M>1$ which depends only on $R$ and so that
if $\dist{Fy,F(\Crit)}<(1+\supF)R<1$ 
or $\dist{y,\Crit}<5R^{1/\mmax}$ then
$$1/M~<~\abs{F'(y)}/\dist{y,c}^{\m{c}-1}~<~M~,$$
$$1/M~<~\dist{Fy,Fc}/\dist{y,c}^{\m{c}}~<~M~.$$
 
\itemm{\rm {(ii)}} $R$ is so small that the first return
time of the critical points in the Julia set
to $\bigcup F^{-1}B_{R}(Fc_{i})$ is
greater than a constant $\tau$ chosen so  that
$\alpha_k~>~16^{\mmax}~M^2~>~1~,$ for $k\ge\tau$.

\item{\rm{(iii)}} $R'$ is so small that
${16~R'}~{\supF}~/~\inf_{n}\br{\alpha_n}^2~\le~R~M^{-1}~.$

\item{\rm{(iv)}} $R'$ is so small that there are no critical points
in the $2R'$-neighborhood of the Julia set inside the Fatou set.
\end{description}

\subparagraph{Dictionary of  constants.}
For the sake of clarity we  list here other constants and 
indicate their mutual dependence and places of introduction.

\begin{description} 
\item{}
${\l2t}=\const(R',q)$, $K,R_{2t}=\const({\l2t},R')$, 
and $C(p)=\const({\l2t},R',p)$ 
in Lemma~\ref{lem:2t},
\item{}
$C_{3t}=\const(R_{2t})$  in Lemma~\ref{lem:3t},
\item{}
$L''=\const(C_{3t},R_{2t}),\,L'=\const(L'')$ 
in Subsection~\ref{subsec:estimates}.
\end{description}

\subsection{Types of orbits}\label{sec:local}
The general scheme of decomposing backward orbits into ``expansive'' blocks
was introduced in \cite{ceh} for Collet-Eckmann dynamics.
Despite many similarities,
our setting is substantially less hyperbolic than
that given by the Collet-Eckmann condition.
Though it is not possible to use directly the estimates of \cite{ceh}
and new results are needed,
 the strategy to capture expansion is similar:  
we classify  pieces of orbits, depending on whether they are close to
critical points or not, and derive expansion estimates.
We will obtain three different types of  estimates, which when 
combined will yield ``expansion''  of the derivative along any orbit.

\paragraph{First type.}
Our objective is to estimate expansion along pieces of the backward
orbit which ``join'' two critical points, i.e. there is a disk
in a vicinity of the first critical point which can be pulled  back 
conformally along the orbit until its boundary hits the second critical point.

The formulation of Lemma~\ref{lem:1t} has
to encompass the possibility of critical points with different
multiplicities and hence it does not guarantee immediate expansion.

\begin{defi}\label{defi:1t} 
A sequence $F^{-n}(z),\cdots, F^{-1}(z),z$ of preimages of $z$
is of the {\em first type} with respect to  critical points 
$c_1$ and $c_2$  
if there exists a radius $r<2R'$ such that
\begin{description}
\itemm{\em{1)}}~Shrinking neighborhoods~$U_k$ for $B_r(z)$, $1\le k< n$, avoid
 critical points,
\itemm{\em{2)}}~The critical point~$c_{2}\,\in\,\dd U_n$,
\itemm{\em{3)}}~The critical point~$c_{1}\,\in\,F^{-1}B_{R}(Fz)$.
\end{description}
\end{defi}
Note that given the sequence of preimages, the radius $r$
is uniquely determined by the condition 2). 
To simplify notation set  $\mu_i\,:=\,\m{c_i}$, 
$d_2\,:=\,\dist{F^{-n}z,c_2}$,
and $d_1\,:=\,\dist{z,c_1}$.
Let $r_2'$  be the maximal radius
so that
$B_{r_2'}(F^{-n}z)\,\subset\,F^{-n}(B_{r/2}(z))$.
For consistency, put $r_{1}:=r$.

\begin{lem}\label{lem:1t}
For any sequence $y=F^{-n}(z),\cdots,F^{-1}(z),z$ 
of preimages of the first type and
any $\mum\le\mmax$ we have 
\begin{description}
\itemm{\em 1)}$\abs{\br{F^n}'(y)}^{\mum}~>
~{\alpha_n}~(\gamma_n)^{\mmax}~2^{\mmax}
~~\fr{(d_2)^{\mu_2-1}}{(r_2')^{\mum-1}}
~\fr{(r_1)^{\mum-1}}{(r_1+d_1)^{\mu_1-1}}~$,
\itemm{\em 2)} $(d_{2})^{\mu_{2}}< 
(\max(r_1,d_1))^{\mu_{1}}(\gamma_{n})^{-\mmax}$,
\itemm{\em 3)}
$\dist{Fy,Fc_2}~\leq ~R~.$
\end{description}
\end{lem}
\begin{proof}
By the construction the map $F^{-(n-1)}:\,B_{\Delta_{n-1}\,r_1}(z)\,\to\, 
U_{n-1}$ is conformal.
The Koebe distortion Lemma~\ref{lem:koeb} implies that
\begin{eqnarray}
\dist{Fy,Fc_{2}}&\le&\frac{(1-\delta_n)(2-\delta_n)}{\delta_{n}}
~\Delta_{n-1}\,r_1~\abs{\br{F^{n-1}}'(Fc_2)}^{-1}\nonumber\\
&\le&\frac{2\,R'}{\delta_{n}}
~{\supF}~\abs{\br{F^{n}}'(Fc_2)}^{-1}\nonumber\\
&\le&\frac{2\,R'}{\delta_{n}}
~{\supF}~{\delta_{n}}~/~(\alpha_n)^2~\le~R~,\label{eq:1t,1}
\end{eqnarray}
and the third inequality follows by the choice of $R'$ (see {\rm (iii)}).

We prove the first inequality. 
The condition 2) of Definition~\ref{defi:1t} implies that
$F^nc_2\in\dd B_{r\del_n}$. Hence
$$\dist{F^nc_2,c_1}~\le~\dist{F^nc_2,z}\,+\,\dist{z,c_1}
~\le~r_1\,+\,d_1~.$$
Since $\dist{F^{n+1}c_2,Fc_1}$ is small 
(less than $(1+\supF)R$), by the choice of $R$
we have 
$\abs{F'(F^{n}c_2)}\le M\dist{F^nc_2,c_1}^{\mu_1-1}$. 
Thus
\begin{equation}\label{eq:1t,2}
\abs{\br{F^{n-1}}'\br{Fc_2}}~\ge
~\abs{\br{F^{n}}'\br{Fc_2}}~\fr1{M\dist{F^nc_2,c_1}^{\mu_1-1}} 
~\ge~\fr{\sigma_n}{M(r_1+d_1)^{\mu_1-1}}~.\end{equation}

Considering the conformal map
$F^{-(n-1)}:~B_{r\del_{n-1}}(z)\,\to\,\,U_{n-1}$~, 
by the Koebe distortion lemma~\ref{lem:koeb}
we obtain that
$$\abs{\br{F^{-(n-1)}}'\br{F^nc_2}}~\ge
~\fr{\delta_n}{\br{2-\delta_n}^3}~\abs{\br{F^{-(n-1)}}'\br{z}}~,$$
and therefore
\begin{equation}\label{eq:8oc,1}
\abs{\br{F^n}'(y)}~\ge~\fr{\delta_n}{8M}~\abs{\br{F^{n-1}}'\br{Fc_2}}
~\dist{y,c_2}^{\mu_2-1}~.
\end{equation}
Together with the estimate (\ref{eq:1t,2}) this yields
$$\abs{\br{F^n}'(y)}~\ge~\fr{\delta_n}{8M^2}~\sigma_n
~\fr{(d_2)^{\mu_2-1}}{(r_1+d_1)^{\mu_1-1}}~.$$

By the Koebe lemma,
the image of the map $F^{-n}:\,B_{r_1/2}(z)\,\to\, U_{n}$ contains
a ball of radius $\fr18\,r_1\,\abs{\br{F^n}'(y)}^{-1}$
and the center $y$. Hence
\begin{equation}\label{eq:8oc,2}
\abs{\br{F^n}'(y)}~\ge~\fr{r_1}{8r_2'}~.
\end{equation}

Combining the above estimate raised to the
power $(\mum-1)$ with the previous one
we obtain
\begin{eqnarray*}
\abs{\br{F^n}'(y)}^{\mum}&\ge&
\fr{\delta_n}{8^{\mum}M^2}~\sigma_n
~~\fr{(d_2)^{\mu_2-1}}{(r_2')^{\mum-1}}
~\fr{(r_1)^{\mum-1}}{(r_1+d_1)^{\mu_1-1}}\\
&=&
\fr{\delta_n}{8^{\mum}M^2}~(\delta_n)^{-1}
~(\alpha_n)^2~(\gamma_n)^{\mmax}
~~\fr{(d_2)^{\mu_2-1}}{(r_2')^{\mum-1}}
~\fr{(r_1)^{\mum-1}}{(r_1+d_1)^{\mu_1-1}}\\
&>&
{\alpha_n}~(\gamma_n)^{\mmax}~2^{\mmax}
~~\fr{(d_2)^{\mu_2-1}}{(r_2')^{\mum-1}}
~\fr{(r_1)^{\mum-1}}{(r_1+d_1)^{\mu_1-1}}~.
\end{eqnarray*}
The last inequality holds since  by the choice of $R$,
$n\ge\tau$.

Note that using the inequality (\ref{eq:1t,2})
we can modify the estimate (\ref{eq:1t,1}) by writing
\begin{eqnarray*}
(d_2)^{\mu_2}&\le&M\,\dist{Fy,Fc}\\
&~\le~&
\frac{(1-\delta_n)(2-\delta_n)}{\delta_{n}}
~\Delta_{n-1}\,r_1~\abs{\br{F^{n-1}}'(Fc_2)}^{-1}\nonumber\\
&\le&M\,\frac{2\,r_1}{\delta_{n}}\cdot\fr{M(r_1+d_1)^{\mu_1-1}}{\sigma_n}\\
&\le&2^{\mmax}\,M^2\,\br{\max(r_1,d_1)}^{\mu_1}\,
(\alpha_n)^{-2}\,(\gamma_n)^{-\mmax}\\
&=&2^{\mmax}\,M^2\,(\alpha_n)^{-2}~\br{\max(r_1,d_1)}^{\mu_1}\,(\gamma_n)^{-\mmax}
~<~\br{\max(r_1,d_1)}^{\mu_1}\,(\gamma_n)^{-\mmax}~
\end{eqnarray*}
which completes the proof of the second inequality.
\end{proof}

\paragraph{Second type.}
A piece of a  backward orbit is of the second type
if there exists a neighborhood of size $R'$
which can be pulled back univalently  along the backward orbit.
Type two preimages yield expansion along pieces of orbits
of a uniformly bounded length. In this ``uniformly bounded'' setting, type $2$ corresponds
to pieces of backward orbits which   stay
at a definite distance  from the critical points.

\begin{defi}\label{defi:2t}
Let  $\dist{z,\J}\,\leq \,~R'/2$.
A sequence $F^{-n}(z),\cdots, F^{-1}(z),z$ of preimages of $z$
is of the {\em second type}
if the ball $B_{R'}(z)$ can be pulled back univalently along it. 
\end{defi}

\begin{lem}\label{lem:2t}
Let $\typeii{(z)}$ be the set of all preimages 
of the second type of some point $z$.
Denote $n(y):=n$ if $F^ny=z$.
\begin{enumerate}
\item 
There exists a constant ${\l2t}>0$ so that the following holds:
\begin{equation}
\inf_{y\in \typeii{(z)},\,n(y)\ge {\l2t}}~\abs{\br{F^n}'(y)}~>~6~.
\end{equation}
\item 
If the Poincar\'e series  $\Sigma_q(v)$ converges
for some point $v$, then ${\l2t}$
can be chosen so that
$$\sum_{y\in \typeii{(z)},\,n(y)\ge {\l2t}}\abs{\br{F^n}'(y)}^{-q}
~<~\fr1{36}~.$$
\item Once ${\l2t}$ is chosen 
there exist positive constants $K$ and $C(q)$ 
such that
$$\begin{array}{ll}
\sum_{y\in \typeii{(z)},n(y)~\le~
{\l2t}}\abs{\br{F^n}'(y)}^{-q}~<~C(q)~,
&{\rm~and}\\
\abs{\br{F^n}'(y)}~>~K~
&{\rm~for~every~}y\in \typeii{(z)},~n(y)\leq {\l2t}~.
\end{array}$$
\item
Once ${\l2t}$ is chosen
there is a  positive $R_{2t}$ such that
for any point $z$ and
its second type preimage $y\,=\,F^{-{\l2t}}z$ of order ${\l2t}$ we have
$$B_{2R_{2t}}(y)~\subset~F^{-{\l2t}}\br{B_{R'}(z)}~.$$
\end{enumerate}
\end{lem}
\begin{proof}
The proof of the first part is standard and follows from
the compactness argument.
Suppose that the claim  does not hold.
Then there is an infinite collection of 
sequences of the second type
$$F^{-n_{i}}(z_{i}), \dots \,F^{-1}(z_{i}), z_i$$
such that  $n_{i} \rightarrow \infty$ and 
$\abs{\br{F^{n_i}}'(F^{-n_i}(z_{i}))}\,\leq\,6$.
Consider the preimages $F^{-n_i}\br{B_{R'/2}(z'_{i})}\ni F^{-n_i}(z_{i})$, 
where $z'_i$ is the closest point  to $z_i$  in $\J$. 
Without loss of generality we can assume that $R'\ll\diam\J$.
By the Koebe distortion lemma~\ref{lem:koeb}, 
any of these preimages contains a 
round ball around $F^{-n_i}(z'_{i})$ of the radius larger than 
$\eta\,:=\,R'/(8\cdot6)$.
Let $y$ be an accumulation point of the sequence
$F^{-n_i}(z'_{i})\in\J$.
By the construction,
there is an increasing subsequence $\brs{k_j}$ of the sequence
$\brs{n_j}$ such that the images of $B_{\eta/2}(y)$ under $F^{k_j}$
are contained in $B_{R'}(z)\,\not\supset\,\J$ 
and we arrived at a contradiction, since $y\in\J$ and the Julia set has 
{\em the eventually onto} property (see Theorem 1 in  \cite{CG}). 

To prove the second part, we recall again 
that if the Poincar\'e series for $v$ converges
then $v$ must be a non-exceptional point, 
i.e. with preimages dense in the Julia set. 
We can fix finitely many of them, say
$v_1,\dots,v_n$, so that they are $R'/4$-dense in $J$
and their Poincar\'e series will also converge.
Then for any point $z$ with $\dist{z,\J}\,<\,~R'/2$
there is a point $v_j\in B_{3R'/4}(z)$. By 
the Koebe distortion lemma~\ref{lem:koeb},
we can write
$$\sum_{y\in \typeii{(z)}}\abs{\br{F^n}'(y)}^{-q}~\lesssim~
\Sigma_q(v_j)\,\le\,\max_{j}\Sigma_q(v_j)~\lesssim~\Sigma_q(v)
\,<\,\infty~,$$
and after a proper choice of ${\l2t}$
the second inequality of the lemma follows.

For the third part, we use the Koebe distortion lemma~\ref{lem:koeb},
$$1~\grtsim~\diam(F^{-n}B_{R'}(z))~\ge~\fr14~R'~\abs{\br{F^{-n}}'(z)}
~=~\fr{1}{4}~R'~\abs{\br{F^{n}}'\br{F^{-n}(z)}}~.$$
Both statements easily follow.

The fourth part is an immediate consequence of the Koebe distortion lemma. 
\end{proof}

\paragraph{Third type.}
The third type gives more leeway in choosing blocks than the types
$1$ and $2$. In \cite{ceh} a more restrictive approach  was used. 
Blocks of type $3$  connect the large scale with critical points.  

\begin{defi}\label{defi:3t}
A sequence $F^{-n}(z),\cdots,F^{-1}(z),z$ of preimages of $z$
is of the {\em third type} with respect to the critical point 
$c_2$ if there exists a radius $r<2R'$ such that
\begin{description}
\itemm{\em{1)}}~Shrinking neighborhoods~$U_k$ for $B_r(z)$, $1\le k< n$, avoid 
critical points,
\itemm{\em{2)}}~The critical point~$c_{2}\,\in\,\dd U_n$~.
\end{description}
\end{defi}
Note that given the sequence of preimages, the radius $r$
is uniquely determined by the condition 2).

The next  lemma estimates the expansion along the third 
type preimages. 
To simplify notation set  $\mu_2\,:=\,\m{c_2}$, 
$d_2\,:=\,\dist{F^{-n}z,c_2}$.
Let $r_2'$  be the maximal radius
so that
$B_{r_2'}(F^{-n}z)\,\subset\,F^{-n}(B_{r/2}(z))$.
For consistency, put $r_{1}:=r$.

\begin{lem}\label{lem:3t}
There exists $C_{3t}>0$ such that for any
sequence of preimages of the third type
$F^{-n}(z),\cdots,F^{-1}(z),z$ and any $\mum\le\mmax$ we have
\begin{description}
\itemm{\em 1)}
$ \abs{\br{F^n}'(y)}^{\mum}~>~C_{3t}
~{\alpha_n}~(\gamma_n)^{\mmax}
~~\fr{(d_2)^{\mu_2-1}}{(r_2')^{\mum-1}}
~{(r_1)^{\mum-1}}~$,
\itemm{\em 2)} $\dist{Fy,Fc_2}~\leq ~R~$,
\end{description}
If the sequence of the third type is preceded 
by a sequence of the second type of length ${\l2t}$
then we can substitute
$R_{2t}$ for $r_1$ in the estimate {\em 1)}.
\end{lem}
\begin{proof}
The proof of the second inequality follows from ($\ref{eq:1t,1}$).
Indeed, we did not use there the existence of a critical point close to $z$,
so the proof works for the third type  of preimages.
Equation (\ref{eq:1t,1}) implies the following estimate
$$(d_2)^{\mu_2}~\le~M\,\dist{Fy,Fc}~\le
~M\,2R'\sup\abs{F'}/(\alpha_n)^2~,$$
and therefore
\begin{equation}\label{eq:3tdecayofd}
\dist{y,c}~\rightarrow~0~,~~~\mbox{as}~~ n\to\infty~.
\end{equation}

The inequalities~(\ref{eq:8oc,1}) and (\ref{eq:8oc,2}) from 
the proof of Lemma~\ref{lem:1t}  are valid also for the third type preimages.
So using the same notation, we can write
\begin{eqnarray*}
\abs{\br{F^n}'(y)}&\ge&
\fr{\delta_n}{8M}~\abs{\br{F^{n-1}}'(Fc_2)}~{(d_2)^{\mu_2-1}}\\
&\ge&\fr{\delta_n}{8M}~\fr{\sigma_{n}}{\supF}~{(d_2)^{\mu_2-1}}
~,\end{eqnarray*}
and
$$\abs{\br{F^n}'(y)}~\ge~\fr{r_1}{8r_2'}~.$$
Combining these estimates we conclude that
\begin{eqnarray*}
\abs{\br{F^n}'(y)}^{\mum}&\ge&
\fr{\delta_n}{8^{\mum}M}~\fr{\sigma_{n}}{\supF}
~~\fr{(d_2)^{\mu_2-1}}{(r_2')^{\mum-1}}
~{(r_1)^{\mum-1}}\\
&>&C_{3t}~{\alpha_n}~(\gamma_n)^{\mmax}
~~\fr{(d_2)^{\mu_2-1}}{(r_2')^{\mum-1}}
~{(r_1)^{\mum-1}}~,
\end{eqnarray*}
where
$$C_{3t}~:=~\inf_n\brs{\alpha_n~\br{8^{\mmax} M\supF}^{-1}\,}~>~0~.$$
It remains to observe that the last assertion of the lemma is true
since if a sequence of the third type is preceded 
by a sequence of the second type of length ${\l2t}$ 
then $r_1>R_{2t}$ by  Lemma~\ref{lem:2t}.
\end{proof}

\section{Specification of orbits}\label{sec:specif}

We will estimate expansion along the backward orbits 
by decomposing them into blocks of different types described
in Section~\ref{sec:backexpan}. 

 Lemma~\ref{lem:2t}, which governs the expansion in the large scale, 
was stated in the proximity of the Julia set,
and to apply it we will need the following Lemma, which  holds
in the absence of parabolic points
(see Lemma~5 in \cite{ceh}).

\begin{lem}\label{lem:comp}
There exists $\epsilon>0$ such that the backward orbit of 
any $z$ in the $\epsilon$-neighborhood of the Julia set
stays in the $R'/2$-neighborhood of the Julia set. 
\end{lem}
This means that the assertions of Lemma~\ref{lem:2t} are valid
for type $2$ preimages $F^{-n}(z)$, $n>0$, 
provided $z$ belongs to an $\epsilon$-neighborhood of the Julia set.

\begin{defi}\label{defi:adhoc}
We say that a backward  orbit $y=F^{-n}(z),\dots, z$ is decomposed
into a sequence of blocks 
if there exists an increasing sequence
of integers $0=n_{0}< \dots < n_{k}=n$ so that for every
$i=0,\dots,k-1$  the orbit $F^{-n_{i+1}}(z),\dots, F^{-n_{i}}(z)$
is of type $1$, $2$, or $3$. Given a pair of integers
 $0\leq r<l\leq n$, we say that a subsequence 
$F^{-n_{l}}(y),\dots, F^{-n_{r}}(y)$
yields expansion $M$ if 
\[|(F^{n_{l}-n_{r-1}})'(y)|\geq M\;.\]
The point $F^{-n_{l}}(z)$ 
is called a {\em terminal} point of the subsequence.

\end{defi}
Recall that $\sigma_n\,:=\,\min_{c\in\Crit}\brs{\abs{\br{F^n}'(Fc)}}$
was represented as the product 
$(\alpha_n)^2\,(\gamma_n)^{\mmax}\,/\,\delta_n~$.
The sequence $\brs{\delta_n}$ will majorate the distortion in the
shrinking neighborhoods construction, $\brs{\alpha_n}$ will swallow
all remaining constants, and $\brs{\gamma_n}$ will provide
the desired expansion. 
\begin{lem}[Main Lemma]\label{lem:main}
Let $\epsilon$ be supplied  by Lemma~\ref{lem:comp}. Assume that a rational
function $F$ satisfies the summability condition with an exponent 
$\alpha\leq 1$ and set $\beta=\mmax\alpha/(1-\alpha)$.
Suppose that a point $z$ belongs to $\epsilon$-neighborhood of the Julia
set $J$ and a ball $B_{\Delta}(z)$ can be pulled
back univalently by a branch of $F^{-N}$.
We claim that there exist  positive constants 
$L'>L, K$ independent of $z, \Delta$, and $\epsilon$
such that the sequence $F^{-N}(z),\dots, z$ 
can be decomposed into blocks of types $1$, $2$, and $3$, and
\begin{itemize}
\item  every type $2$ block,
except possibly the leftmost one, 
 has the length contained in $[L,L')$ and yields expansion  $6$,
\item the leftmost type $2$ block has the length contained in $[0,L]$
and  yields expansion  $K>0$, 
\item  all subsequences of the form $1\dots13$,
except possibly the rightmost one,  yield expansion 
$$\gamma_{k_j}\dots\gamma_{k_1}\gamma_{k_0}~,$$
$k_i$ being the lengths of the corresponding blocks,
\item  the rightmost sequence of the form $1\dots 13$ yields
expansion
$$\begin{array}{ll}
 \gamma_{k_j}\dots\gamma_{k_1}\gamma_{k_0}~\Delta^{(1-\mu(c)/\mmax)}&
{\it~if~a~critical~point~}c\in B_{\Delta}(z)~,\\
\gamma_{k_j}\dots\gamma_{k_1}\gamma_{k_0}~\Delta^{(1-1/\mmax)}&
{\it~if~otherwise~.}
\end{array}$$
\end{itemize}
\end{lem}
\subsection{Inductive decomposition of backward orbits}\label{sec:induc}
Let $z$ be a point which satisfies the assumptions of
Lemma~\ref{lem:main}.
We fix $N$ and a sequence of the preimages
$F^{-N}(z),\dots,F^{-1}(z),z$. 
We will split this sequence in the subsequences of 
the first, second, and third types.

Namely we will define by induction sequences $\{n_j\}$
and $\{z_{j}:= F^{-n_{j}}(z)\}$
such that $n_0=0,\,n_{m-1}>N-{\l2t},\,n_m=N$,
and 
\begin{description}
\itemm{{\mbox{I)}}}
For every $j>0$
the sequence $F^{-n_{j}}z,\dots,F^{-n_{j-1}}z$ 
is of the first, second, or third type.
\itemm{{\mbox II)}} 
For $j>0$ either the sequence
$F^{-n_{j}}z,\dots,F^{-n_{j-1}}z$ is of the
second type
(case {IIa)}), or
some critical value $F(c_j)\in B_{R}(Fz_j)$
(case {IIb)}).
\end{description}
Some additional properties will be discussed
in the process of the construction.

\paragraph{Base of induction.}
If the shrinking neighborhoods
for $B_{2R'}(z_{0})$ do not contain critical points,
we set $n_1:={\l2t}$, $z_1:=F^{-{\l2t}}z$,
and the condition IIa) is satisfied for $z_1$. 
We start from $j=1$ and continue the inductive procedure as below.

By Lemma~\ref{lem:comp}, $\dist{z_{j},J}<R'/2$, and
hence sequences of the second type will yield desired
expansion. 

Otherwise we take $r:=\del$.
By the choice of $\del$, the shrinking neighborhoods for $B_\del(z)$
omit the critical points. We increase $r$ continuously  until
certain shrinking neighborhood  $U_k$ hits some critical point $c$,
i.e. $c\in\dd U_k$.
It must happen for some $r=r_0$ with $\del<r_{0}<2R'$. Set $n_{1}:=k$.
Then $z_1$ is the third type preimage of $z_0$,
and the condition IIb) for $z_1$ is satisfied by 
 Lemma~\ref{lem:1t}.

\paragraph{Inductive procedure.}
Suppose we have already constructed $z_j$. 

\subparagraph{Case IIa.}
We enlarge the ball $B_r(z_{j})$ continuously
increasing the radius $r$ from $0$ until one 
of the following conditions occurs:
\begin{description}
\itemm{1)} for some $k$ the shrinking neighborhood
$U_k$ for $B_{r}(z_{j})$ hits some critical point $c'$, $c'\in\dd U_k$,
\itemm{2)} the radius $r$ reaches the value of $2R'$.
\end{description}

In the case 1) we put $n_{j+1}:=n_j+k$.
The condition I) is satisfied: $z_{j+1}$ is the third type preimage of $z_j$. 
The condition IIb)\,is satisfied by Lemma~\ref{lem:3t}.

In the case 2) set $n_{j+1}:=n_j+{\l2t}$.
Then $z_{j+1}$, $\dist{z_{j+1},\J}<R'/2$ is the second type preimage of $z_j$ 
of the length ${\l2t}$. Clearly, $z_{j+1}$ satisfies conditions I) and IIa). 

\subparagraph{Case IIb.} Suppose that we have IIb), but not IIa).
Set $r=0$.
The shrinking neighborhoods 
$U_{l}$ for $B_r(z_j)$, $l\le N-n_j$, do not contain 
critical points.
We increase $r$ continuously until some
domain $U_k$ hits a critical point $c'$, $c'\in\dd U_k$.
This must occur for some $r<2R'$, since IIa) is not satisfied for $z_{j}$.

Let $n_{j+1}:=n_j+k$. Then the condition I) is satisfied:
$z_{j+1}$ is the first type preimage of $z_j$.
Lemma~\ref{lem:1t} implies the condition IIb).

\paragraph{Coding.}
As a result of the inductive procedure, we have  decomposed the backward orbit
of the point $z$ into pieces of  type $1$, $2$ and $3$. This
gives a coding of backward orbits by sequences of 
$1$'s, $2$'s and $3$'s which are always read {\em from right to left}. 
Not all combinations of the entries are admissible here. By the construction,
type $3$ is always preceded by type $2$ except the coding sequence
starts with $3$.
For example we could have a sequence of the form
$$\dots11111323222211113221113~,$$
$F$ acts from the left to the right and our inductive procedure
has started from the right end. All pieces of the second type
except maybe for the very last one
have the length ${\l2t}$. 

This decomposition of backward orbits into pieces of different types
is by no means the only one satisfying the desired properties.
On the contrary, in the next Subsection we will have to reshuffle
the coding slightly to obtain the claim of Lemma~\ref{lem:main}.

\subsection{Estimates of expansion}\label{subsec:estimates}

\paragraph{Growth of the derivative along sequences containing $1\dots 13$.}
Consider a sequence $1\dots 132$ obtained in the inductive construction. 
We will estimate  expansion along a part of the orbit corresponding
to its shorter sequence $1\dots 13$. 
Suppose that in the sequence $1\dots132$
the consecutive pieces of type $1$  
have the lengths $k_{i},~ i=1,\dots, j$,
and the piece of type $3$ has the length $k_0$.
Let  $k=k_{0}+\dots+ k_{j}$.
By Lemma~\ref{lem:1t} and Lemma~\ref{lem:3t} with $\mum:=\mmax$
we have that
\begin{eqnarray}
\abs{\br{F^k}'(y)}^{\mmax}&>&
\prod_{i=1}^{j}
~{\alpha_{k_i}}~(\gamma_{k_i})^{\mmax}~2^{\mmax}
~\fr{d_{i+1}^{\mu_{i+1}-1}}{(r_{i+1}')^{\mmax-1}}
~\fr{r_i^{\mmax-1}}{(r_i+d_i)^{\mu_i-1}} \nonumber\\
&&~~~\cdot~~{}
C_{{3t}}~{\alpha_{k_1}}~(\gamma_{k_1})^{\mmax}
~~\fr{(d_2)^{\mu_2-1}} {(r_2')^{\mmax-1}}
~{(R_{2t})^{\mmax-1}}\nonumber\\
&>&\left(~C_{{3t}}~{\prod_{i=0}^{j}\alpha_{k_i}}\,(\gamma_{k_i})^{\mmax}\right)
~~\cdot~(R_{2t})^{\mmax-1}~\cdot
~\fr{d_{j+1}^{\mu_{j+1}-1}}{(r_{j+1}')^{\mmax-1}}\nonumber \\
&&~~~\cdot~~{}\left(
~\prod_{i=1}^{j}~2^{\mmax}
~\fr{d_{i}^{\mu_{i}-1}}{(r_{i}')^{\mmax-1}}
~\fr{r_i^{\mmax-1}}{(r_i+d_i)^{\mu_i-1}}\right). \label{equ:xia}
\end{eqnarray}
Since
$r_i'\,<\,\min(r_i,d_i)$ and $\mu_i\,\le\,\mmax$,  
we obtain that
\begin{equation}\label{eq:oc11,2}2^{\mmax}
~~\fr{d_{i}^{\mu_{i}-1}}{(r_{i}')^{\mmax-1}}
~\fr{r_i^{\mmax-1}}{(r_i+d_i)^{\mu_i-1}}~~>~~1~.
\end{equation}
Also  $r'_{j+1} < d_{j+1}$ and
$$\fr{d_{j+1}^{\mu_{j+1}-1}}{(r_{j+1}')^{\mmax-1}}~>~1~.$$
Combining (\ref{equ:xia}) and the estimates above, we obtain that
\begin{equation}~\label{equ:xia'}
\abs{\br{F^k}'(y)}^{\mmax} >
(R_{2t})^{\mmax-1}~\cdot~
C_{{3t}}~{\prod_{i=0}^{j}\alpha_{k_j}}\,(\gamma_{k_i})^{\mmax}~~.
\end{equation}

If  the rightmost sequence  $1\dots 13$ is not pre\-ce\-ded by $2$
then similarly as above, using Lemma~\ref{lem:1t}, we obtain that
\begin{equation}\label{eq:oc10,2}
\abs{\br{F^k}'(y)}^{\mmax} >  
~C_{{3t}}~{\prod_{i=0}^{j}\alpha_{k_j}}\,(\gamma_{k_i})^{\mmax}
~\cdot~\Delta^{\mmax-\mu(c)}~ 
\end{equation}
if there is a critical point $c$ inside  $B_{\Delta}(z)$.
Otherwise we use Lemma~\ref{lem:3t} instead to infer that  
\begin{equation}\label{eq:oc10,3}
\abs{\br{F^k}'(y)}^{\mmax} >  
~C_{{3t}}~{\prod_{i=0}^{j}\alpha_{k_j}}\,(\gamma_{k_i})^{\mmax}
~\cdot~\Delta^{\mmax-1}~.
\end{equation} 
\paragraph{Derivation of the Main Lemma.}
Consider a sequence of the preimages $F^{-m}(z),\dots,z$ with
a coding sequence of the form
$$\dots11111323222211113221113~.$$
If the lenght $k_{1}+\dots +k_j$  of a piece of the backward orbit  of the form 
 $11\dots13$ 
is large enough then the expansion prevails over
accumulation of distortion and small scale constants 
in the estimate (\ref{equ:xia'}).
Otherwise $k_{1}+\dots +k_j$ is  uniformly bounded  and the whole
piece of the backward orbit will be treated  as type $2$. Consequently,
we will convert the code $1\dots132$ into $2$.

We proceed with estimates along the above lines to complete the proof of
Lemma~\ref{lem:main}. Since $\lim_{i\rightarrow
\infty}\alpha_{i}=\infty$, there exists $\tau$ such that 
$\alpha_{i}\geq 8$ for $i\geq \tau$. Set
$$\alpha'_n~:=~\inf\brs{\prod_{j}\alpha_{i_j}:
~{i_0+i_1+i_2+\dots\ge n};~i_1,i_2,\dots\ge\tau}~,$$
and observe that $\lim_{n\to\infty}\alpha'_n\,=\,\infty$.
Now we  choose large $L''$ so that
for $n\ge L''$ one has $${\alpha'_n}\,C_{3t}\,(R_{2t})^{\mmax-1}\,\ge\,1\;.$$

A new coding of the preimages  $F^{-m}(z),\dots,z$ is designed
as follows:
take a piece of the backward orbit corresponding to 
a subsequence $1\dots132$ of the length $k$. The
 consecutive pieces (counted from the right to the left)
have the lengths $k_{i},~ i=0,\dots, j$. Consider two possible cases:
\begin{description} 
\itemm{{\rm 1)}} 
If $k<L''$ then $1\dots 132$ is replaced by $2$.
The corresponding  block of the preimages is indeed
of type $2$ and the  length $n\,:=\,{\l2t}+k$
with $n<L':=({\l2t}+L'')$.
\itemm{{\rm 2)}} 
If $k\ge L''$
then  $1\dots132$ remains unchanged and by the estimate (\ref{equ:xia'})
and the definition of $L''$, the derivative of $F^{k}(y)$
is greater than $\gamma_{k_{j}}\cdots\gamma_{k_{1}}$.
The last pair of estimates of the Main Lemma~\ref{lem:main} follows
immediately from $(\ref{eq:oc10,2})$ and  $(\ref{eq:oc10,3})$
and the definition of $L''$.
\end{description}
The proof of the Main Lemma~\ref{lem:main} is completed.

\paragraph{Strong expansion along some sequences $11\dots1$.}
Consider a subsequence  $11\dots1$ obtained in the inductive construction.
Suppose that $x$ is a terminal point of this subsequence and
the consecutive pieces of type $1$  
have lengths $k_{i},~ i=0,\dots,j$. 
Following the notation of Lemma~\ref{lem:main}, we
prove the lemma below, which will be later used in our investigation
of conformal and invariant densities.
\begin{lem}\label{lem:proxi}
Let $G$ be a set of indexes $j$ such that $d_j<r_j$ and
$\mmax'$  be the maximal multiplicity
which occurs in the sequence
$\brs{\mu_j:\,j\in G}$.
Choose an index $j$
from the  set
$G':=\brs{j\in G:\,\mu_j=\mmax'}$ and denote $k=\sum_{i=0}^{j-1}k_i$.
Then
$$\abs{\br{F^k}'(x)}~>~\prod_{i=0}^{j-1}
\gamma_{k_i}~.$$
\end{lem}
\begin{proof}
If $\mum:=\mmax'$ then  Lemma~\ref{lem:1t} and Lemma~\ref{lem:3t} 
imply a counterpart of the estimate (\ref{equ:xia}),
\begin{eqnarray}\label{equ:xia1}
\abs{\br{F^k}'(x)}^{\mmax'}&>&
\prod_{i=0}^{j-1}
~{\alpha_{k_i}}~(\gamma_{k_i})^{\mmax'}~2^{\mmax'}
~\fr{d_{i+1}^{\mu_{i+1}-1}}{(r_{i+1}')^{\mmax'-1}}
~\fr{r_i^{\mmax'-1}}{(r_i+d_i)^{\mu_i-1}} \nonumber\\
&>&\left(~{\prod_{i=0}^{j-1}\alpha_{k_i}}\,(\gamma_{k_i})^{\mmax'}\right)
~~\cdot~~\fr{d_{0}^{\mu_{0}-1}}{(r_{0}')^{\mmax'-1}}\\
&&~\cdot~{}\left(
~\prod_{i=0}^{j-1}~2^{\mmax'}
~\fr{d_{i}^{\mu_{i}-1}}{(r_{i}')^{\mmax'-1}}
~\fr{r_i^{\mmax'-1}}{(r_i+d_i)^{\mu_i-1}}\right)~
\cdot~\,\fr{r_j^{\mmax'-1}}{(r_j+d_j)^{\mu_j-1}}\;.\nonumber
\end{eqnarray}
We cannot proceed exactly as in (\ref{equ:xia}),
since it was essential  that $\mmax$ was 
the maximal multiplicity. 
Instead, we use the properties of $\mmax'$ and 
the set $G$ as follows:
\begin{description}
\item{\rm (i)} If $i\in G$ then 
$r_i'\,<\,\min(r_i,d_i)$ and $\mu_i\,\le\,\mmax'$. We see that the estimate
(\ref{eq:oc11,2}) holds with $\mmax$
replaced by $\mmax'$,
$$2^{\mmax'}~~\fr{d_{i}^{\mu_{i}-1}}{(r_{i}')^{\mmax'-1}}
~\fr{r_i^{\mmax'-1}}{(r_i+d_i)^{\mu_i-1}}~~>~~1~.$$
\item{\rm (ii)} If $i\notin G$ then $d_i\ge r_i>r_i'$ and 
the same estimate is still valid,
$$2^{\mmax'}~\fr{d_{i}^{\mu_{i}-1}}{(r_{i}')^{\mmax'-1}}
~\fr{r_i^{\mmax'-1}}{(r_i+d_i)^{\mu_i-1}}~\ge~
2^{\mmax}~\fr{d_{i}^{\mu_{i}-1}}{(2d_i)^{\mu_i-1}}
~\fr{r_i^{\mmax'-1}}{(r_{i}')^{\mmax'-1}}~>~1~.$$
\item{\rm(iii)} By our choice, $j\in G'$. This means that
 $d_j<r_j$ and $\mu_j=\mmax'$. Hence,
$$2^{\mmax'}\,\fr{r_j^{\mmax'-1}}{(r_j+d_j)^{\mu_j-1}}~\ge~
2^{\mmax}~\fr{r_j^{\mmax'-1}}{(2r_j)^{\mmax'-1}}~>~1~.$$
\item{\rm(iv)} By the definition, $r'_{0} < d_{0}$ and
$$\fr{d_{0}^{\mu_{0}-1}}{(r_{0}')^{\mmax'-1}}~>~1~.$$
\end{description}
Inserting the estimates ${\rm(i)-(iv)}$ into $(\ref{equ:xia1})$
we obtain the claim of the lemma.
\end{proof}

\paragraph{Proof of Proposition~\ref{prop:fatouexpan}.}
Fix a point $z$ sufficiently close to the  Julia set
and a branch of $F^{-n}$ at $z$.
By Lemma~\ref{lem:techseq},
$$\sum_{n}(\gamma_n)^{-\beta}~<~1/4~,~~~\beta\,
:=\,\fr{\mmax\alpha}{1-\alpha}~.$$
Let
$\gamma_n'\,:=\,\inf\brs{\prod_{j}\gamma_{i_j}:\,i_0+i_1+i_2+\dots= n}$.
By simple algebra, 
\begin{eqnarray*}
\sum_{n}(\gamma_n')^{-\beta}
&<&\sum_{n}(\gamma_n)^{-\beta}+\br{\sum_{n}(\gamma_n)^{-\beta}}^2
+\br{\sum_{n}(\gamma_n)^{-\beta}}^3+\dots\\
&<&\fr14+\br{\fr14}^2+\br{\fr14}^3+\dots~=~\fr13~.
\end{eqnarray*}

Lemma~\ref{lem:main} gives a decomposition
of the orbit  $F^{-n}(z),\dots,F^{-1}(z),z$
into pieces of type $1$, $2$, and $3$,
with the following properties
(we restate them using new notation):
\begin{description}
\item{(i)}
each piece of the form $1\dots13$ of length $k$,
except the rightmost one,
yields expansion  $\gamma_k'$~,
\item{(ii)}
the rightmost piece of the form $1\dots13$  
of length $k$ yields expansion $\gamma_k'\Delta^{1/\mmax-1}$~,
\item{(iii)}
each piece of the form $2$, except possibly the leftmost one,
has the length $l\in[L,L')$ and yields expansion  $6\ge\lambda^l$,
where $\lambda\,:=\,6^{1/L'}\,>\,1$~,
\item{(iv)}
the leftmost piece of the form $2$,
has the  length $l\in[0,L)$ and yields expansion $K$~.
\end{description}
If we set
$$\omega_n~:=~\inf
\brs{\,K\,\lambda^{k_0}\,\prod_{j\ge1}\gamma_{k_j}'\,:
~k_0+k_1+k_2+\dots\in{[n-{\l2t},n)}}~,$$
then properties (i)--(iv) above clearly imply
$$\abs{\br{F^{n}}'(F^{-n}z)}~>~\Delta^{1-1/\mmax}~\omega_n~.$$
On the other hand, 
\begin{eqnarray*}
\sum_{n}(\omega_n)^{-\beta}
&<& K^{-\beta}\,
\br{1+\lambda^{-\beta}+\lambda^{-2\beta}+\dots}\\
&&~~\cdot\br{\sum_{n}(\gamma_n')^{-\beta}+\br{\sum_{n}(\gamma_n')^{-\beta}}^2
+\br{\sum_{n}(\gamma_n')^{-\beta}}^3+\dots}\\
&<&K^{-\beta}\,\br{1-\lambda^{-\beta}
}^{-1}\,\br{\fr13+\br{\fr13}^2+\br{\fr13}^3+\dots}\\
&<&K^{\beta}\,\br{1-\lambda^{-\beta}}^{-1}\,{\fr12}~<~\infty~,
\end{eqnarray*}
which completes the proof of the first inequality
of Proposition~\ref{prop:fatouexpan}. The proof of the second,
when a critical point $c\in B_{\Delta}(z)$, is very much the same.

\subsection{Summability along backward orbits}
Fix a point $z$ and a positive number $\Delta$.
Let $\xx{}{z,\Delta}$ stand for 
the set of all preimages of $z$ 
such that a ball $B_{\Delta}(z)$ can be pulled
back univalently along the corresponding branch.
By Lemma~\ref{lem:main}, every backward orbit of  $z$
which terminates at $y \in \xx{}{z,\Delta}$ 
can be decomposed into blocks of type $1$, $2$, or $3$.

\begin{defi}\label{defi:symbol}
In the decomposition of the Main Lemma, let $x\in\xx{}{z,\Delta}$
be a point which starts  a type $3$ block. 
Denote by $\typei(x|z)=\typei^{\Delta}(x|z)$
the set of all $y\in\xx{}{z,\Delta}$ 
which are the endpoints of  type $1$ blocks preceded by exactly 
one type $3$ block beginning at $x$.
For example,  preimages of $x$ which are the endpoints of blocks
 $13$, $113, \dots$ belong to $\typei(x|z)$.
 
Let $L'>L$ be the constants supplied by the Main Lemma.
In the decomposition of the Main Lemma, let $x\in\xx{}{z,\Delta}$
be a point which starts a type $2$ block.
Denote by $\typeii_{l}(x|z)$ and $\typeii_{s}(x|z)$, respectively, the sets of all 
``long'' (of order $L'>n(y)\ge L$) and ``short'' (of order $n(y)<L$)
type $2$ preimages $y$ of $x$.
\end{defi}

Note that the definition of $\typei(x|z)=\typei^{\Delta}(x|z)$ depends on the choice of $\Delta$.
Also, the definitions of $\typeii_{l}(x|z)$ and $\typeii_{s}(x|z)$  depend on the choice of $\Delta$,
however all estimates  from Lemma~\ref{lem:2t} are independent of $\Delta$,
so we simplify the notation by omitting $\Delta$.

Note also that $z$ is its own preimage of order zero
(since formally $F^{0}(z)=z$),
so we write, e.g. $y\in\typei(z|z)$
if $y\in\xx{}{z,\Delta}$ 
is the endpoint of a type $1$ block preceded by exactly 
one type $3$ block beginning at $z$.

We will drop $z$ from the notation
of  $\typei(x|z), \typeii_{s}(x|z),\typeii_{l}(x|z)$
whenever no confusion can arise.

\begin{lem}\label{lem:main2}
Let $\beta=\mmax\alpha/(1-\alpha)$.
If a rational function $F$ satisfies the summability condition with an exponent 
$\alpha\leq 1$ 
then there exists $\epsilon>0$ so that for every point $z$ from
$\epsilon$-neighborhood of the Julia set $J$ and every set  
$\typei(x|z)=\typei^{\Delta}(x|z)$,
$$\begin{array}{ll}
\sum_{y\in\typei(x|z)}\abs{\br{F^{n(y)}}'(y)}^{-\beta}~<~\fr1{3}
&{\it ~if~}x\neq z~,\\
\sum_{y\in\typei(x|z)}\abs{\br{F^{n(y)}}'(y)}^{-\beta}
~<~\fr13~\Delta^{\beta(1/\mmax-1)}~
&{\it~if~}x=z~, \\
\sum_{y\in\typei(x|z)}\abs{\br{F^{n(y)}}'(y)}^{-\beta}
~<~\fr13~\Delta^{\beta(\mu(c)/\mmax-1)}~
&{\it~if~}x=z{\it~and~a~critical~point~}c\in B_{\Delta}(z).
\end{array}$$
\end{lem}
\begin{proof}
We will work with sequences $\alpha_{n}$, $\gamma_{n}$,
and $\delta_{n}$ supplied by  Lemma~\ref{lem:techseq}.
Recall that $\sum_{n}(\gamma_{n})^{-\beta}\,<\,1/(16\deg F)$.

Observe that any point $y\in F^{-k}(z)$ has 
at most $4\deg F$ preimages of a given length which are
of the first or the third type. In fact, since pull-backs 
to the critical values are univalent, there is only one
way to hit a specific critical value after particular number
of steps, and thus only $\mu(c)$ ways to hit a critical point $c$,
but 
\begin{equation}\label{equ:deg}
\sum_{c}\mu(c)\,=\,\#\{\Crit\}+\sum_{c}(\mu(c)-1)
\,\le\,2(\deg F-1)+2(\deg F-1)\,<\,4\deg F~.
\end{equation}
Therefore, for every sequence $k_0,k_1,\dots,k_m$ of positive integers
there are at most $(2\deg F)^{m+1}$ sequences $1\dots13$
with the corresponding lengths of the pieces of type $1$ and $3$.
By the Main Lemma~\ref{lem:main}, 
if $x\not= z$ then  for every $y\in \typei(x)$ 
\[\abs{\br{F^{n(y)}}'(y)} \geq  \gamma_{k_0}\gamma_{k_1}\dots\gamma_{k_m}\]
and
\begin{eqnarray}\label{eq:1tpoincare}
\sum_{y\in\typei(x)}\abs{\br{F^{n(y)}}'(y)}^{-\beta}
&<&\sum_{m,k_0,k_1,\dots,k_m}~(4 \deg F)^{m+1}
~\br{\gamma_{k_0}\gamma_{k_1}\dots\gamma_{k_m}}^{-\beta}\nonumber\\
&<&{\textstyle 4 \deg F\,\sum_{k}\gamma_{k}^{-\beta}
\,+\,(4 \deg F\,\sum_{k}\gamma_{k}^{-\beta})^2
\,+\,(4 \deg F\,\sum_{k}\gamma_{k}^{-\beta})^3
\,+\,\dots}\nonumber\\
&<&\fr14~+~\br{\fr14}^2~+~\br{\fr14}^3~+~\dots
~=~\fr1{3}~.
\end{eqnarray}

If $x=z$ then  the rightmost sequence begins
with $3$.  Similarly as before, using the estimates of 
the Main Lemma, we obtain that
\begin{eqnarray}
\sum_{y\in\typei(z)}\abs{\br{F^{n(y)}}'(y)}^{-\beta}
&<~\fr13~\Delta^{\beta(1/\mmax-1)}~,
&\quad{\rm or~}\label{typeisum}\\
\sum_{y\in\typei(z)}\abs{\br{F^{n(y)}}'(y)}^{-\beta}
&<~\fr13~\Delta^{\beta(\mu(c)/\mmax-1)}~,
&\quad{\rm if~ a ~critical~point~}c\in B_{\Delta}(z)~.\nonumber
\end{eqnarray}
This completes the proof of Lemma~\ref{lem:main2}.
\end{proof}

\begin{lem}\label{lem:bbb}
Assume that the Poincar\'e series  with exponent $q$ is summable
for some point $v\in \CC$.
Then there exists $\epsilon>0$ so that for every point $z$ from
$\epsilon$-neighborhood of the Julia set $J$ and every set  
$\typeii_{l}(x|z)$ and $\typeii_{s}(x|z)$,
$$\begin{array}{ll}
\sum_{y\in\typeii_{{l}}(x|z)}\abs{\br{F^{n(y)}}'(y)}^{-q}
~<~\fr1{36}~&,\\
\sum_{y\in\typeii_{{s}}(x|z)}\abs{\br{F^{n(y)}}'(y)}^{-p}
~<~C(p)&{\rm~for~any~}p~.
\end{array}$$
\end{lem}
\begin{proof}
This is  a reformulation of Lemma~\ref{lem:2t} in the new notation.
\end{proof}

\section{{Poincar\'e series}}\label{sec:Fatou}
In this Section we analyze Poincar\'e series,
particularly proving a self-improving property of the Poincar\'e exponent. 
Theorem~\ref{theo:poincare} is a direct consequence of this property.
We recall that $\xx{}{z,\Delta}$
stands for the set of all preimages $F^{-n}z$, $n\in\N$, such that
the ball $B_\Delta(z)$ can
be pulled back univalently along the corresponding branch of $F^{-n}$.

\begin{prop}[Self-improving property  of the Poincar\'e exponent]\label{prop:poincare}
Suppose that a rational function $F$ satisfies the summability condition 
with an exponent $$\alpha~<~\fr{q}{\mmax+q}~,$$
$q>0$, and the Poincar\'e series with exponent $q$ 
converges for some point $v$,
$\Sigma_q(v)<\infty$.
Then there exist $p<q$, $\epsilon>0$, and $C(\epsilon,p)$  
so that for every point 
$z$ in the $\epsilon$-neighborhood of the Julia set
$$\begin{array}{ll}
\sum_{y\in\xx{}{z,\Delta}}\abs{\br{F^n}'(y)}^{-p}~<~C~
\Delta^{p(\fr{\m{c}}{\mmax}-1)}~&
{~if~a~critical~point~}c\in B_{\Delta}(z)~,\\
\sum_{y\in\xx{}{z,\Delta}}\abs{\br{F^n}'(y)}^{-p}~<
~C~\Delta^{p(\fr1{\mmax}-1)}~&
{~otherwise~.}
\end{array}$$
\end{prop}
\begin{coro}\label{cor:critexpon}
Assume that $F$ satisfies the summability condition with an exponent
$$\alpha~<~\fr{q}{\mmax+q}~,$$
$q>0$, and there exists a point $z\in \CC$  so that
the Poincar\'e series $\Sigma_q(z)$ converges.
Then \begin{itemize}
\item $\dpoin(w)\,<\,q$ if $w$ is at a positive distance from 
the orbits
of the critical points,
\item $\dpoin(c)\,<\,q$ if $c$ is a  critical point of the maximal
 multiplicity.
\end{itemize}
\end{coro}
\begin{proof}
If the distance of  $w$ to the critical orbits in $J$ is positive
then all preimages of $w$ belong to $\xx{}{w,\Delta}$ with $\Delta$
sufficiently small. This yields  $\Sigma_p(w)<\infty$.

If $c$ is a critical point of the maximal multiplicity $\mu(c)=\mmax$
then 
$$\sum_{y\in\xx{}{c,\Delta}}\abs{\br{F^n}'(y)}^{-p}~<~C~
\Delta^{p(\fr{\m{c}}{\mmax}-1)}~=~C~,$$
and letting $\Delta$ go to zero  we obtain that
$$\sum_n\,\sum_{y\in F^{-n}c}\,\abs{\br{F^n}'(y)}^{-p}~<~C~.$$
\end{proof}
\begin{coro}
If a rational function $F$ satisfies the summability condition 
with an exponent $\alpha\,<\,{2}/({\mmax+2})$
and its Julia set is not the whole sphere, 
then there exists $p<2$ so that the conclusion 
of Proposition~\ref{prop:poincare} holds.
\end{coro} 
\begin{proof}
If the Julia set is not the whole sphere and the Fatou set does not
contain elliptic components
then there exists a point $v\in \CC\setminus J$
such that the Poincar\'e series
$\Sigma_2(v)$ converges. This a classical area argument, \cite{sullivan-rio}.
It is enough to notice that there exists 
a small ball $B_\delta(v)$ 
free from the critical orbits and  with preimages pairwise disjoint.
\end{proof}

\paragraph{Proof of Theorem~\ref{theo:poincare}.}
Suppose that the Poincar\'e series $\Sigma_q(v)$ converges
for a point $v\in \CC$. 
If $q\le\dpoin(J)$ then, by Corollary~\ref{cor:critexpon}, 
there exist $\epsilon>0$ so that
$\dpoin(J)\le q-\epsilon<\dpoin(J)$, a contradiction. 
This means that for $q=\dpoin(J)$, 
the Poincar\'e series $\Sigma_q(z)$ diverges for every point $z\in \CC$.
Hence, $\dpoin(z)\geq \dpoin(J)$ for every $z\in \CC$. 

By the definition of the Poincar\'e exponent $\dpoin(J)$,
for any $\epsilon>0$ there exist $q<\dpoin(J)+\epsilon$
and a point $v\in \CC$ so that the  Poincar\'e series $\Sigma_q(v)$
converges. By Corollary~\ref{cor:critexpon}, for all points which are at 
the positive distance to the critical orbits and for all critical points of
maximal multiplicity one has $\dpoin(z)<\dpoin(J)+\epsilon$
and Theorem~\ref{theo:poincare} follows.
\begin{flushright}
$\Box$
\end{flushright}

\paragraph{Proof of Proposition~\ref{prop:poincare}.}
We use the inductive decomposition of backward orbits  
described in Section~\ref{sec:induc}. 
Let $z$ be a point from an $\epsilon$-neighborhood of the Julia set. 
By  Lemma~\ref{lem:bbb}, 
$$\sum_{y\in\typeii_{l}(x)}
\abs{\br{F^{n(y)}}'(y)}^{-q}~<~\fr1{36}~.$$
But there are at most $(\deg F)^{L'}$
points in $\typeii_{l}(x)$,
since these preimages are all of the order at most $(L'-1)$.
Therefore by power means inequality 
(see e.g. Section 2.9 in \cite{halipo})
we have for $p<q$ sufficiently close to $q$
(namely $p>q-q\log2\br{L'\log\br{\deg F}}^{-1}\,$):
\begin{eqnarray*}
\sum_{y\in\typeii_{l}(x)}
\abs{\br{F^{n(y)}}'(y)}^{-p}
&<&\br{\sum_{y\in\typeii_{l}(x)}
\abs{\br{F^{n(y)}}'(y)}^{-q}}^{\fr pq}
~\cdot~\br{\deg F}^{L'\fr {q-p}q}\\
&<&\br{\fr1{36}}^{\fr pq}
~\cdot~\br{\deg F}^{L'\fr {q-p}q} \\
&<&\fr1{6}~\cdot~\br{\deg F}^{L'\fr {q-p}q}
~<~\fr13~.
\end{eqnarray*}
Also by  Lemma~\ref{lem:bbb}
$$\sum_{y\in\typeii_{s}(x)}\abs{\br{F^{n(y)}}'(y)}^{-p}~<
~C~=~C(p)~.$$

We expand $\sum_{y\in\xx{}{z,\Delta}}\abs{\br{F^n}'(y)}^{-p}$
by grouping preimages of the same kind into clusters. 
We begin with $z$ obtaining preimages
of three kinds: $\typei(z)=\typei^\Delta(z)$, $\typeii_{l}(z)$ and $\typeii_{s}(z)$.
Points in $\typeii_{s}(z)$ are terminal while
preimages $y$ of the points in
$\typei(z)$ and $\typeii_{l}(z)$
are divided further. 
We proceed in  this fashion down the tree of preimages of $z$.
If there is no critical point in $B_{\Delta}$ we obtain that
\begin{eqnarray*}
\sum_{y\in\xx{}{z,\Delta}}\abs{\br{F^n}'(y)}^{-p}
&=&
\sum_{z'\in\typeii_{s}(z)}\abs{\br{F^{n(z')}}'(z')}^{-p}
~+~\sum_{z'\in\typei,\typeii_{l}(z)}\abs{\br{F^{n(z')}}'(z')}^{-p}\\
&\cdot&
\Biggm(\,\sum_{z''\in\typeii_{s}(z')}\abs{\br{F^{n(z'')}}'(z'')}^{-p}
~+~\sum_{z''\in\typei,\typeii_{l}(z')}\abs{\br{F^{n(z'')}}'(z'')}^{-p}\\
&\cdot&
\Biggl(\,\sum_{z'''\in\typeii_{s}(z'')}\abs{\br{F^{n(z''')}}'(z''')}^{-p}
~+~\dots~~~\Biggl)\Biggm)\\
&\le&C\,+\,\br{\fr13+\fr13(\Delta)^{p(1/\mmax-1)}}
\br{C\,+\,\fr23\br{C\,+\,\dots}}\\
&=&C\,+\,\fr13\br{1+(\Delta)^{p(1/\mmax-1)}}C\br{1+\fr23+\br{\fr23}^2+\dots}\\
&=&\br{2+(\Delta)^{p(1/\mmax-1)}}C~<~3C\,(\Delta)^{p(1/\mmax-1)}~.  
\end{eqnarray*}
Otherwise, we have a stronger estimate
$$\sum_{y\in\xx{}{z,\Delta}}\abs{\br{F^n}'(y)}^{-p}
~<~3C\,(\Delta)^{p(\m{c}/\mmax-1)}~ .$$
This proves Proposition~\ref{prop:poincare}.
\begin{flushright}
$\Box$
\end{flushright}

\section{Induced hyperbolicity and conformal measures}\label{sec:induced}

\subsection{Inductive procedure with a stopping rule}\label{sec:stop}

We decompose a sequence  of preimages
 $F^{-N}(z),\dots,F^{-1}(z), z$ into blocks
of  types $2$ and $1\dots13$ using the 
usual inductive procedure 
with the following new stopping rule: at the first
occurrence of a type $2$ sequence we stop the induction.  
For the reader's convenience 
we will describe  the construction.

\paragraph{Construction.} We take shrinking neighborhoods $\{U_{k}\}$ for
$B_{2R'}(z)$. If they do not contain the critical points we form
one block of type $2$ of the length $N$. Otherwise, we set $r=0$
and increase it continuously until some shrinking neighborhood $U_{k}$
hits a critical point $c$, $c\in \partial U_{k}$. It must happen
for some $0<r< 2R'$. Set
$r_{0}:=\dist{F^{k}(c),z}$,  $n_{1}:=k$,  and $z_{1}:=F^{-n_{1}}(z)$.
Then $z_{1}$ is a third type preimage of $z$ and the ball $B_{r_{0}}$
can be pulled back univalently by $F^{N}$ along the backward orbit.

\subparagraph{Inductive procedure.}
Suppose we have already constructed $z_j=F^{-n_{j}}(z)$
 which is of type $1$ or $3$.
We enlarge  the ball $B_r(z_{j})$ continuously
increasing the radius $r$ from $0$ until one of the following
conditions is met:
\begin{description}
\itemm{1)} 
for some $k\leq N-n_{j}$ the shrinking neighborhood
$U_k$ for $B_{r}(z_{j})$ hits a critical point $c\in \Crit$, $c\in\dd U_k$,
\itemm{2)} radius $r$ reaches the value of $2R'$.
\end{description}
In the case 1) we put $n_{j+1}:=n_j+k$. Clearly,
 $z_{j+1}:=F^{-n_{j+1}}(z)$ is a type $1$ preimage of $z_{j}$.
If 2) holds, we set $z_{j+1}:= F^{-N}(z)$ 
which is a type $2$ preimage of $z_{j}$. This terminates the
construction in this case.

\subparagraph{Coding.}
As a result of the inductive procedure, we can  decompose the backward orbit
of a point $z$ into pieces of  type $1$, $2$ and $3$. This
gives a coding of backward orbits by sequences of 
$1$'s, $2$'s and $3$'s. By the construction, only the following
three types of  codings are allowable: $2, 1\dots 3, 21\dots1 3$.
We recall that according to our convention, during the inductive
procedure we  put symbols in the coding from the right to the left. 

We attach to every sequence of preimages of $z$  the sequence 
$k_l, \dots, k_0$ of the lengths  of the blocks of preimages of a
given type in its coding. Again our convention requires that
$k_{0}$ stands always for the length of the rightmost block of preimages in
the coding. Clearly, $k_{0}+\cdots + k_{l}=N$.  

\subsection{Most points go to large scale infinitely often}
We recall that a Jacobian
of a $\delta$-conformal  measure $\nu$  
is equal to $\abs{F'}^\delta$ (see Definition~\ref{def:mes}), 
$$d\nu(F(z))~=~\abs{F'(z)}^\delta\,d\nu(z)~.$$
Consider a subset of points in $J$ which infinitely often go to
the large scale of size $R'$ with bounded distortion: 
$$\Jls \,:=\,\brs{z\in J:\, \exists~n_j\to\infty, \mbox{~with~}F^{n_j} 
\mbox{~univalent~on ~} F^{-n_j}\br{B\br{F^{n_j}z,R'}}}~.$$
Note that the value of $R'$ is already fixed and does not depend on the point.

\begin{prop}\label{prop:largescale}
Suppose that a rational function $F$ satisfies the {\em summability condition}
with an exponent
$$\alpha~<~\fr{p}{\mmax+p}~.$$
Then for any $p$-conformal measure $\nu$ with no atoms at critical points
$\nu(J\setminus \Jls )=0$.
\end{prop}
\begin{proof}
For every $x\in J$ and every $k\in \N$, we 
use the inductive procedure of  Section~\ref{sec:stop}
to decompose the sequence  $x,\dots, F^k(x)$ of preimages of $F^{k}(x)$
into blocks of either $211\dots113$ or $2$.
The procedure is stopped at the first occurrence of type $2$ block,
which might be of arbitrary length.
In particular, it might be of length zero which means that a block
of type $2$ does not occur and the sequence ends with type $1$.

Denote by ${\cal E}_{x}$ the set of all codes obtained for $x$.
Points in $\Jls $ are precisely those for which we get infinitely
many different type $2$ sequences.
Hence, if  $x\in J\setminus \Jls $, then $x$ is  
a terminal point of an infinite number of
sequences $2111\dots113$ with only a finite choice of type $2$ blocks.
Let $k(x)$ be the minimal number for which  infinitely many 
sequences from ${\cal E}_{x}$ have the same type $2$ block of  
length $k(x)$. Denote $X_k:=\brs{x:k(x)=k}$ and observe that 
the sets $\{X_{k}:k=0,1,\dots\}$ form  
a countable partition of $J\setminus \Jls$. 

If $\nu$ has no atoms at critical points and
$F^{k}(X)$ is measurable then
$$\nu(F^kX)=0\iff \nu(X)=0\iff\nu(F^{-k}X)=0.$$
Since $F^k(X_k)\subset X_0$ and consequently  
$J\setminus \Jls\subset\cup_k F^{-k}(X_0)$, 
it is sufficient to prove that $\nu(X_0)=0$. 
Without loss of generality we can exclude from $X_0$
all preimages of the critical points since they are of zero $\nu$ measure.
Every  point  $x\in X_0$  must  be terminal 
for infinitely many different subsequences $1\dots1$ . 
Otherwise the orbit of $x\in X_0$ would pass near the critical points
only finitely many times and hence its distance to the set  $\Crit$
would be positive. Another consequence of the finitness of $1\dots 1$
subsequences would be an unbounded length of type $3$ blocks
in ${\cal E}_x$.  The estimate (\ref{eq:3tdecayofd})  yields
$\dist{x,\Crit}=0$, a contradiction.

By very much the same argument, using that the
distance from $x\in X_0$ to $\Crit$ is positive, we obtain that
the length of  the leftmost blocks of type $1$
in  ${\cal E}_{x}$ must be bounded and therefore 
we can choose infinitely many sequences from ${\cal E}_{x}$
with the same leftmost block. 
Next, we consider the second block of type $1$  from the left and repeat the
above argument  to  produce  infinitely many sequences in  
${\cal E}_{x}$ with the same two leftmost blocks. We continue
in this fashion until we build by induction an infinite sequence $1111\dots $
terminating at $x$.
Denote corresponding parameters by 
$d_j,\,r_j,\,r_j',\,c_j,\,\mu_j,\,n_j$ 
with $j=0,-1,-2,\dots$
(we use negative integers to preserve convention of enumerating
from the right to the  left).

Let $G$ be the set of indexes $j$ such that $d_j<r_j$.
The second inequality of Lemma~\ref{lem:1t} implies  that if $j\notin G$ then
$(d_{j+1})^{\mu_{j+1}}<(d_j)^{\mu_j}(\gamma_{n_j})^{-\mmax}$. This means
that  $G$ is infinite since  otherwise $\lim_{j\to-\infty}d_j=\infty$.
Now set $\mmax'$ to be the maximal multiplicity
which occurs infinitely often in the sequence
$\brs{\mu_j:\,j\in G}$.
Let $X_0'(k)$ stand for the set of all points $x\in X_{0}$ such
that there are no points of larger than $k=\mmax'(x)$ multiplicity in $G$.
We see that 
$$X_0\subset\bigcup_{k=2}^{\mmax}\bigcup_{i=0}^{\infty} 
F^{-i}(X_0'(k))\;.$$
Therefore, it is sufficient to show that $\nu(X_0'(k))=0$ for $k=2,\dots,\mmax$.
We fix $k=\mmax'$ and drop it from the notation of $X_0'(k)$.

Fix a point $x\in X_0'$ and
take an index $j$
in the infinite set $G':=\brs{j\in G:\,\mu_j=\mmax'}$.
Set $k=\sum_{i=-1}^{j}k_i$. Then,  by
 Lemma~\ref{lem:proxi},
$$\abs{\br{F^k}'(x)}^{p}~>~\prod_{i=j+1}^{-1}
(\gamma_{k_i})^{p}~,$$
and hence
$$d\nu(x)~=~\abs{\br{F^k}'(x)}^{-p}d\nu(F^kx)
~<~\prod_{i=j+1}^{-1}(\gamma_{k_i})^{-p}d\nu(F^kx)~.$$
Assuming that $\nu(X_0')$ is positive, we proceed 
similarly as in  the proof of Proposition~\ref{prop:poincare}
(note that parameters $k_i$, $j$  depend on $x$),
\begin{eqnarray*}
+\infty&=&\int_{X_0'}\#G'(x)d\nu(x)
~=~\int_{X_0'}\sum_{j\in G'(x)}1\, d\nu(x)\\
&<&\int_{X_0'}~~\sum_{j\in G'(x)}~~
\prod_{i=j+1}^{-1}(\gamma_{k_i})^{-p}~d\nu(F^kx)\\
&\le&\int_{J}~~\sum_{F^kx=z,j\in G'(x)}~~
\prod_{i=j+1}^{-1}(\gamma_{k_i})^{-p}~d\nu(z)\\
&\le&\int_{J}~\sum_{x\in\typei(z)}~
\prod_{i=j(x)+1}^{-1}(\gamma_{k_i(x)})^{-p}~d\nu(z)\\
&\le&\int_{J}~~\sum_{j,k_{-1},k_{-2},\dots,k_{j+1}}~(4 \deg F)^{\abs{j-1}}
~\br{\gamma_{k_{-1}}\gamma_{k_{-2}}\dots\gamma_{k_{j+1}}}^{-p}~d\nu(z)\\
&<&\int_{J}~~\left({\textstyle 4 \deg F\,\sum_{k}\gamma_{k}^{-\beta}
\,+\,(4 \deg F\,\sum_{k}\gamma_{k}^{-\beta})^2
\,+\,(4 \deg F\,\sum_{k}\gamma_{k}^{-\beta})^3
\,+\,\dots}\nonumber\right)d\nu(z)\\
&~<~&\int_{J}\br{\fr1{4}+\br{\fr14}^2+\dots}d\nu(z)~=~\int_{J}\fr1{3}d\nu(z)~<~+\infty~.
\end{eqnarray*}
This yields  a contradiction and proves the proposition.
\end{proof}

\subsection{Conformal measures}

The  notion of conformal measures 
was introduced to rational dynamics by D.~Sullivan
following an analogy with Kleinian groups, see Definition~\ref{def:mes}.
Loosely speaking,  a probabilistic measure $\nu$, supported on the Julia set,
 is {\em conformal with exponent} $\delta$, if its Jacobian
is equal to $\abs{F'}^\delta$, i.e.
$$d\nu(F(z))~=~\abs{F'(z)}^\delta\,d\nu(z)~.$$
D.~Sullivan proved in \cite{sullivan-rio} that for every Julia set
there exists a conformal measure with an exponent
$\delta\in(0,2]$. For hyperbolic Julia sets, there exists
only one conformal measure which coincides with a normalized 
$\HD(J)$-dimensional Hausdorff measure.
In  general, it is more difficult to describe analytical properties 
of conformal measures.
For example, it is  an open problem  whether there exists
a non-atomic conformal measure for a given rational function.

We recall that a conformal dimension $\dconf(J)$ of $J$ is 
defined as 
$$\dconf(J)~=~\inf\brs{\delta:~\exists~\delta-{\mathrm conformal~measure}}~.$$

A simple compactness argument (see \cite{sullivan-rio})
 shows that infimum is attained in the definition above.
The following lemma is a version of standard
Patterson-Sullivan construction of conformal measures
(cf. \cite{sullivan-rio}):

\begin{lem}\label{lem:existconf}
Let $z$ be either a critical point of the maximal
multiplicity in the Julia set, or a point at a positive distance
from the orbits of the critical points. Then there exists
a $\dpoin(z)$-conformal measure.
\end{lem}
\begin{proof}
If $z$ is a critical point, then for any $q>\dpoin(z)$ there is 
an atomic conformal measure  supported on the preimages of $z$.
To see this,  normalize
$$\sum_n~\sum_{y\in F^{-n}z}~\abs{\br{F^n}'(y)}^{-q}~1_y~,$$
where $1_{y}$ is  a Dirac measure at $y$, to be a  probabilistic measure.

If $z$ is a point at a positive  distance from the critical orbits
then standard arguments
of \cite{sullivan-rio} apply.
\end{proof}

\begin{lem}\label{lem:erg}
Suppose that there exist a $p$-conformal measure $\eta$
and a $q$-conformal measure $\nu$ which have no
atoms at critical points. 
If $F$ satisfies the summability condition with an exponent
$$\alpha~<~\fr{\max\brs{p,q}}{\mmax+\max\brs{p,q}}
~=~\max\brs{\fr{p}{\mmax+p},\fr{q}{\mmax+q}}~,$$
then $p=q$ and $\eta=\nu$.
\end{lem}
\begin{proof}
If a  ball $B$ of radius $r(B)$ is mapped with a bounded distortion to the large
scale, i.e. $F^{n}(B)=A$, then
$$\nu(B) ~\asymp~\br{\fr{r(B)}{\diam (A)}}^{q}\,\nu(A)
~\asymp~r(B)^{q}~.$$

Assume first that $p$ and $q$ are different,
without loss of generality $p<q$.
Then, by  Proposition~\ref{prop:largescale},
$\nu$-almost every point goes infinitely often to the large scale
with bounded distortion.
This implies  that for $\nu$-almost 
every point $z$ there is a sequence of balls $B_{j}$
of radius $R_{j}\to 0$ centered at $z$ so that 
$$\eta(B_j)~\asymp~(R_j)^p~=~(R_j)^{p-q}~(R_j)^q~\asymp
~(R_j)^{p-q}~\nu(B_j)~.$$
Let ${\Cal B}$ be a collection of all  balls of  radius  less than $r$ which
are mapped with a uniformly bounded distortion to the large scale. By the
Besicovitch covering theorem
(see Section 2.7 in \cite{mattila}) there exists a subcollection
$\Cal B'$ of $\Cal B$ so that 
$\nu$-almost all points  of $J$ are contained in $\bigcup_{B\in {\Cal B}'}$ 
and  every point in $\C$ is covered by at most $P$ balls from $\cal B'$.
Then
\begin{eqnarray*}
\eta(J)&\ge&P^{-1}\sum_{B\in{\Cal B'}}\eta(B)
~\gtrsim~\sum_{B\in{\Cal B'}}r(B)^{p-q}\,\nu(B)\\
&\ge&r^{p-q}\sum_{B\in{\Cal B'}}\nu(B)~\ge~r^{p-q}\,\nu(J)~
\end{eqnarray*}
which (for sufficiently small $r$)
contradicts the fact that $\eta(J)=\nu(J)=1$.

Hence $p=q$. If $\nu$ and $\eta$ are different probabilistic measures
then their difference $\nu-\eta$ has a non-trivial
positive and negative part. After normalization,
$(\nu-\eta)_-$ and $(\nu-\eta)_+$ become $q$-conformal
measures which are mutually singular.
Therefore,  without loss of generality,  we can assume that $\nu$ and $\eta$
are mutually singular.

If $E\subset J$ is an open set then, by the  Besicovitch covering theorem,
we can choose a cover $\Cal B'$ of
$\nu$-almost all points of $E$ 
such that every point in $\C$ is covered by at most $P$ balls
and no points outside $E$ are covered.
Then
$$\eta(E)~\ge~P^{-1}\sum_{B\in{\Cal B'}}\eta(B)~\asymp~
\sum_{B\in{\Cal B'}}(r(B))^{q}~\asymp~
~\sum_{B\in{\Cal B}}\nu(B)~\ge~\nu(E)~$$
and consequently  $\eta(E)~\gtrsim~\nu(E)$ for every Borel set $E$.
This contradicts the mutual singularity of $\eta$ and $\nu$,
and completes the proof.
\end{proof}

\begin{coro}\label{cor:confnoatom}
If  $F$ satisfies the summability condition with an exponent
$$\alpha~<~\fr{\dpoin(J)}{\mmax+\dpoin(J)}~$$
then 
\begin{enumerate}
\item{}
There is a unique $\dpoin(J)$-conformal measure.
It is ergodic and  non-atomic.
This is the only conformal measure  with no atoms at the critical points.
In particular, there are no non-atomic measures
with exponents different from $\dpoin(J)$.
\item{}
There are no conformal measures
with exponents less than $\dpoin(J)$, i.e. $\dpoin(J)=\dconf(J)$.
\item{}
For every $q>\dpoin(J)$ there exists
an atomic  $q$-conformal measure 
supported on the backward orbits of the critical points.
Every conformal measure has  no atoms at other points.
\end{enumerate}
\end{coro}
\begin{proof}\begin{enumerate}
\item{}
By Lemma~\ref{lem:existconf}, there is
a $\dpoin(J)$-conformal measure. It cannot have  atoms 
since otherwise the corresponding Poincar\'e series  converges
and by Corollary~\ref{cor:critexpon}, $\dpoin(J)<\dpoin(J)$.
Now, uniqueness and ergodicity  follow from Lemma~\ref{lem:erg}.
\item{}
There are no atomic measures by Corollary~\ref{cor:critexpon}
and no non-atomic measures by  Lemma~\ref{lem:erg}.
\item{}
To obtain an atomic $q$-conformal measure, $q>\dpoin$, 
distribute atoms at all preimages 
of a critical point of the maximal multiplicity.
If there is a conformal measure  with an atom at a point
whose orbit omits the critical points then
we can easily produce a conformal measure  which has
no atoms at the critical points. 
By Lemma~\ref{lem:erg}, the latter  coincides with a unique
$\dpoin$-conformal measure which is non-atomic, a contradiction.
\end{enumerate}
\end{proof}

Corollary~\ref{cor:confnoatom} implies Theorem~\ref{theo:4}.

\subsection{Frequency of passages to the large scale}
In this Section we give a proof of Theorem~\ref{theo:largescale}.
Consider the  set $\Jlso$ of all points $x\in J$ which 
$\epsilon$-frequently  go to
the large scale of size $R'$, namely: 
$$\exists\, n_j\to\infty:~F^{n_j} 
\mbox{~univalent~on ~} F^{-n_j}\br{B\br{F^{n_j}x,R'}},
~\abs{\br{F^{n_{j+1}}}'(x)}\,<\,\abs{\br{F^{n_{j}}}'(x)}^{1+\epsilon}~.$$
Note that the value of $R'$ is already fixed and does not depend on a point.

\begin{prop}\label{prop:freq}
Suppose that a rational function $F$ satisfies the {\em summability condition}
with an exponent
$$\alpha~<~\fr{p}{\mmax+p}~$$
and
for every  $q>p$
there exists a point $v$ such that the Poincar\'e series $\Sigma_q(v)$
converges. 
If a  $p$-conformal measure $\nu$ has
no atoms at the critical points then 
$\nu(J\setminus\Jlso)=0$.
\end{prop}
\begin{proof}
We say that 
a point $x$ goes
to the large scale of size $R'$ univalently at time $k$  if
\begin{equation}\label{lscale}
F^{k}\mbox{~is~univalent~on~} F^{-k}\br{B\br{F^{k}x,R'}}~.
\end{equation}
Assume that a point $x$ goes to the large scale at a time $m$
and apply the inductive procedure of Section~\ref{sec:stop}
for the sequence  $x,\dots,F^{m-1}(x)$. As a result we obtain 
a sequence of blocks of the form $21\dots13$, $21\dots1$, or $2$
(blocks of type $2$ might be of zero length, meaning that we end with a type $1$ block).
Suppose that  $y:=F^n(x)$ is the first point which belongs
to a block of type $2$ or equivalently is a terminal point
of the longest sequence of the form $1\dots13$ in the decomposition
into blocks of the orbit $x,\dots, F^{m-1}(x)$.
By the definition of $m$, a ball of radius $R'$
around $F^{m}(x)$ can be pulled univalently back to $x$.
Hence, the same is true for a ball of radius $\Delta:=R'/(4M)$,
$M=\sup_{y\in J}|F'(y)|$, around $z:=F^{m-1}(x)$. 
Therefore, $y\in\typei^\Delta(z)=\typei^\Delta(z|z)$
-- recall that $\typei^\Delta(z)$ stands for the set 
of the first type preimages $y\in\xx{}{z,\Delta}$ of $z$,
 obtained in the course of the inductive procedure for $z$.

By  Proposition~\ref{prop:largescale},  
we already know that $\nu$-almost all
points in $J$ go to the 
large scale infinitely often  so it is sufficient to show
that $\nu(X)=0$ for $X:=\Jls\setminus\Jlso$.

Suppose that for every  $x\in X$ there are two increasing sequences 
$\{n_j\}$ and  $\{m_j:=m(n_j)\}$ such that 
$$\abs{\br{F^{m_j}}'(x)}\,\ge\,\abs{\br{F^{n_{j}}}'(x)}^{1+\epsilon}~.$$ 
Therefore,
$\abs{\br{F^{m_j-n_j}}'(F^{n_j}(x))}\,\ge
\,\abs{\br{F^{n_{j}}}'(x)}^{\epsilon}$.
Denote $y=y_j(x):=F^{n_j}x$ and
$z=z_j(x):=F^{m_j}x$.  Then 
\begin{eqnarray*}
\abs{\br{F^{m_j}}'(x)}^{-p}
&=&\abs{\br{F^{n_j}}'(x)}^{-p}\,\abs{\br{F^{m_j-n_j}}'(y)}^{-p}\\
&\le&\abs{\br{F^{n_j}}'(x)}^{-(p+\delta)}
\,\abs{\br{F^{m_j-n_j}}'(y)}^{-(p-\delta/\epsilon)}~.
\end{eqnarray*}
Choose $\delta$ so small that
$\alpha<\beta/(\mmax+\beta)$
for $\beta:=p-\delta/\epsilon$.
Then,  by Lemma~\ref{lem:main2},
\begin{equation}\label{equ:obs1} 
\sum_{y\in\typei^\Delta(z)}
\abs{\br{F^{n(y)}}'(y)}^{-(p-\delta/\epsilon)}~<~\const(\Delta)~.
\end{equation}

By the assumptions, the Poincar\'e series for $q:=p+\delta$
converges for a point $v$ whose
preimages are dense in the Julia set.
We can choose finitely many of them, say
$v_1,\dots,v_n$, so that they are $R'/4$-dense in $J$ and
their Poincar\'e series are  also convergent.
Now, for every point $z$ with $\dist{z,\J}\,<\,~R'/2$,
there is a point $v_j\in B_{3R'/4}(z)$. By 
the Koebe distortion lemma~\ref{lem:koeb},
we have that
$$\sum_{y\in \typeii{(z)}}\abs{\br{F^n}'(y)}^{-q}~\lesssim~
\Sigma_q(v_j)\, \le\,\max_{j}\Sigma_q(v_j)~\lesssim~\Sigma_q(v)
\,<\,\infty~,$$
and 
\begin{equation}\label{equ:obs2}
\sup_{y}\,\sum_{x\in\typeii(y)}
\abs{\br{F^{n(x)}}'(x)}^{-(p+\delta)}~<~\const~<~\infty~.
\end{equation}
Combining the estimates (\ref{equ:obs1}) and (\ref{equ:obs2}),
we obtain that
\begin{eqnarray*}
\infty\cdot\nu(X)&=&
\int_{X}\sum_{j}1\;d\nu(x)~=~
\int_{X}~\sum_{j}~~\abs{\br{F^{m_j}}'(x)}^{-p}\,d\nu\br{F^{m_j}x}\\
&=&\int_{J}~\sum_{j,x:\,z=z_j(x)}~~\abs{\br{F^{m_j(x)}}'(x)}^{-p}\,d\nu\br{z}\\
&\le&\int_{J}~\sum_{j,x:\,z=z_j(x)}~
~\abs{\br{F^{n_j}}'(x)}^{-(p+\delta)}\,
\abs{\br{F^{m_j-n_j}}'(y_j(x))}^{-(p-\delta/\epsilon)}~d\nu\br{z}\\
&\le&\int_{J}~\sum_{x,j:z=z_j(x)}\abs{\br{F^{n_j}}'(x)}^{-(p+\delta)}~
\sum_{y:\exists x,\,y=y_j(x),\,z=z_j(x)}
\,\abs{\br{F^{m_j-n_j}}'(y)}^{-(p-\delta/\epsilon)}d\nu\br{z}\\
&\le&\br{\sup_{y}\,\sum_{x\in\typeii(y)}
\abs{\br{F^{n(x)}}'(x)}^{-(p+\delta)}}
~\cdot~\int_{J}~
\sum_{y\in\typei(z)}\abs{\br{F^{n(y)}}'(y)}^{-(p-\delta/\epsilon)}
d\nu\br{z}\\
&\le&\const\,\int_{J}~\const(\Delta)~d\nu(z)~~<~~\infty~.
\end{eqnarray*}
Therefore $\nu(\Jls\setminus\Jlso)=0$ and the proposition follows.
\end{proof}
\paragraph{Proof of Theorem~\ref{theo:largescale}.}
Theorem~\ref{theo:largescale} is a consequence of Theorem~\ref{theo:4}
and Proposition~\ref{prop:freq}.

\section{Invariant measures}\label{sec:inv}

\subsection{Polynomial summability condition}

In this Subsection we begin the proof of Theorem~\ref{theo:inv},
establishing existence of an absolutely continuous invariant measure,
provided the {\em polynomial summability condition} holds.
We start with the geometric measure $\nu$
with exponent $\delta:=\dpoin(J)$, 
which exists by Theorem~\ref{theo:4}.
It is sufficient to find $Z\in L^1(\nu)$
such that for all $n$
\begin{equation}\label{eq:invmeas}
\frac{d\nu\circ F^{-n}}{d\nu}(z)~\lesssim~Z(z)~.
\end{equation}
In fact, any 
weak subsequential limit of
$$\frac1{n}\sum_{k=1}^n d\nu\circ F^{-k}~,$$
is an invariant measure,
and (\ref{eq:invmeas}) implies that its
density is majorated by $Z(z)$, and hence
it is absolutely continuous with respect to $\nu$.

To find $Z$ and establish (\ref{eq:invmeas}), we proceed as follows.
Given two points $y,\,z$ with $z=F^n(y)$
we define points $v=v(y,z),\,w=w(y,z)$ by the following construction,
which is feasible for $\nu$-almost every $z$.

Since we are interested only in $\nu$-generic points, we can assume,
by Proposition~\ref{prop:largescale},
that $y$ goes to the large scale infinitely often.
Let $n'$ be the first time $n'>n$ when $y$ goes to the large scale,
and denote $w:=F^{n'-1}(y)$. By the choice of $n'$, 
the ball of radius $R'$ around $F(w)$ can be pulled univalently back to $y$.
The same is of course true for the ball of radius $\Delta:=R'/(4M)$ around $w$,
$M=\sup_{y\in J}|F'(y)|$. Now we carry out 
the inductive procedure from the Main Lemma~\ref{lem:main}
for the preimages of $w$ of order $<n'$ till we get a block of type $2$.
By the definition of $n'$, a code of the sequence $y,\,F(y),\dots,\,z,\dots,w$
is  of the form $21\dots13$. Let $v=v(y,z):=F^l(y)$ be the point which
starts the block  of type $2$ (in other words, $v$ ends the blocks $1\dots13$).
Note that $y\in\typeii(v)$ and  $v\in\typei(w)=\typei^\Delta(w|w)$.
We recall that $\typei^\Delta(w|w)$ stands for the set 
of first type preimages of $w$ belonging to $\xx{}{w,\Delta}$ 
 obtained in the course of the inductive procedure for $w$.

Below, we assume that $l=l(v)$ is chosen so that $v=F^l(y)$,
and $j=j(v)=n-l$.
If $F^n(y)=z$, we denote $n(y,z):=n$.
Recalling that $\delta=\dpoin(J)$ is the exponent of
the conformal measure $\nu$,
we can write
\begin{eqnarray*}
\frac{d\nu\circ F^{-n}}{d\nu}(z)
&=&\sum_{y\in F^{-n}(z)}~\abs{(F^n)'(y)}^{-\delta}\\
&=&\sum_{v:\,\exists y\in F^{-n}(z),\,v=v(y,z)}
~\sum_{y\in F^{-l}(v)}~\abs{(F^l)'(y)}^{-\delta}~\abs{(F^{n-l})'(v)}^{-\delta}\\
&\le&\sum_{v:\,\exists y\in F^{-n}(z),\,v=v(y,z)}
\br{\sup_{x}\sum_{y\in\typeii(x),\,F^{l}(y)=x}
~\abs{(F^l)'(y)}^{-\delta}}~\abs{(F^{n-l})'(v)}^{-\delta}\\
&\lesssim&\sum_{v:\,\exists y,\,v=v(y,z)}
~\abs{(F^{n(v,z)})'(v)}^{-\delta}~=:~Z(z)~.
\end{eqnarray*}
The estimate above is possible since
for a fixed $n$ and $z$ every point $v\in F^{-j}z$
is counted only if it is $v(y,z)$ for some $y\in F^{-n}z$,
and in this case $l=n-j$ is fixed (and independent of $y$).
However, once the summation is done, $n$ disappears from
the estimate and does not figure in the definition of $Z$.
Note also that summation set satisfies
$$\brs{v:\,\exists n,y\in F^{-n}(z),\,v=v(y,z)}
\subset \brs{v:\,\exists w,\,y\in\typei(w),\,
y\in F^{-n}(z)\cap F^{-m}(w),\,m\ge n}~.$$
Thus it suffices to prove that $Z\in L^1(\nu)$,
which we can do
by writing
\begin{eqnarray*}
\int Z(z)d\nu(z)
&=&\int\sum_{v:\,\exists y,\,v=v(y,z)}
~\abs{(F^{n(v,z)})'(v)}^{-\delta}d\nu(z)\\
&\le&\int \sum_{v,z:\,\exists y,v=v(y,z),w=w(y,z)}~\abs{(F^{n(v,w)})'(v)}^{-\delta}d\nu(w)\\
&\le&\int \sum_{v\in\typei(w)}~n(v,w)~
\abs{(F^{n(v,w)})'(v)}^{-\delta}d\nu(w)\\
&\lesssim&\int \sum_{n}~n~
\sum_{i,k_1,\dots,k_i:\,k_1+\dots+k_i=n}~\br{\gamma_{k_1}\dots\gamma_{k_i}}^{-\delta}\\
&\lesssim&\int \sum_{n}~n~{\gamma_{n}}^{-\delta}~<~\infty~,
\end{eqnarray*}
the last inequality being true
 since $F$ satisfies the polynomial summability condition
with an exponent $\alpha<\delta/(\delta+\mmax)$.
We also use above that for given $v$ and $w$ there are at most $n(v,w)$
possible choices of $z$,
namely $v,F(v),\dots,F^{n(v,w)-1}(v)=F^{-1}(w)$.
This concludes the proof of the existence
of an absolutely continuous invariant measure.

\subsection{Ergodic properties}\label{sec:lyap}
 
In this Section we complete the proof of Theorem~\ref{theo:inv},
establishing that an absolutely continuous invariant measure
is unique, ergodic, mixing, exact, has positive entropy and Lyapunov exponent.
We do not require the polynomial summability condition of 
Theorem~\ref{theo:inv}:
it is sufficient to assume the corresponding summability condition 
and existence of an absolutely continuous invariant measure.

If an absolutely continuous invariant measure exists,
its ergodicity and uniqueness follow immediately
from the ergodicity of the geometric measure
asserted by Theorem~\ref{theo:4}.

\paragraph{Lyapunov exponents.}
A {\it Lyapunov exponent} of $F$ at $z$ is defined as
\[\chi(z):=\lim_{n\rightarrow \infty} \frac{1}{n}\log|(F^{n})'(z)| ~ ,\]
provided that the limit exists.
A {\it Lyapunov exponent} of an invariant measure $\sigma$ 
is defined as $\chi_{\sigma} =\int \log |F'|\,d\sigma$. 
Birkhoff's ergodic theorem implies that if $\sigma$ is ergodic then for almost every
point $z$ with respect to $\sigma$ the Lyapunov  exponent
$\chi(z)$ exists and is equal to $\chi_{\sigma}$. 
The next lemma is based on standard reasoning (see e.g. \cite{demelo-vanstrien}).

\begin{lem}
Let $\nu$ be a geometric measure of a rational function $F$ which
satisfies the summability condition with an exponent 
\[\alpha<\frac{\dpoin(J)}{\dpoin(J)+\mmax}~.\]
Suppose that $\sigma$ is an absolutely continuous invariant measure with
respect to $\nu$. Then $\sigma$ has positive entropy and Lyapunov exponent
and for almost every point $z$ with respect to $\nu$,
\[\chi(z)=\int \log|F'|\;d\sigma >0.\]
\end{lem}
\begin{proof}
The entropy is given by the formula
$h_{\sigma}=\int \log \Jac_{\sigma}\, d\sigma$
where the Jacobian is defined as the Radon-Nikodym
derivative: $\Jac_{\sigma}\,:=\,d\sigma\circ F/d\sigma$.
The latter is always $\ge1$  since
$\sigma$ is invariant. 
In our case for sufficiently small
sets $A$ which do not contain  the critical points of $F$  we can write
$$\Jac_{\sigma}|_A\asymp\frac{\sigma(F(A))}{\sigma(A)} 
\asymp\frac{\nu(F(A))}{\nu(A)}>0~,$$
and hence
$$1~=~\sum_{y\in F^{-1}F(y)}\frac1{\Jac_\sigma(y)}~>~
\frac1{\Jac_\sigma(y)}~$$
for $\sigma$-almost every $y$.
We conclude that $\sigma$-almost everywhere
$\Jac_\sigma>1$ and hence entropy of $\sigma$ is positive.
Since $\sigma$ is invariant and ergodic,
the remaining statements follow from \cite{mane}. 
\end{proof}  

\paragraph{Exactness.}
Recall that a measure preserving endomorphism $F$ is called {\em mixing} if
for every two measurable sets $A$ and $B$
\[\lim_{n\rightarrow \infty}\sigma(A\cap T^{-n}(B))=\sigma(A)\sigma(B)~.\]
A measure preserving  endomorphism $F$ is {\em exact} if for every measurable $A$, $0<\nu(A)<1$,
there is no sequence of sets $A_{n}$ so that $A=F^{-n}(A_{n})$.
\begin{lem}
Suppose that $F$ satisfies the summability condition 
with an exponent
$$\alpha<\frac{\dpoin(J)}{\dpoin(J)+\mmax}~, $$
and has an absolutely continuous invariant measure $\sigma$.
Then $$\limsup_{n\rightarrow \infty}\sigma(F^{n}(A))=1$$ for
every measurable set $A$ of positive $\sigma$-measure, 
and hence $F$ is exact and mixing.
\end{lem}
\begin{proof}
The proof that exactness implies mixing can be found in~\cite{walters}.
Also it is clear that ($\sigma$ is absolutely continuous
with respect to an ergodic $\nu$)
it is sufficient to prove the same statement for $\nu$:
 $\limsup_{n\rightarrow \infty}\nu(F^{n}(A))=1$.

By Proposition~\ref{prop:largescale}, 
there exists $R'>0 $  so that
for almost every point $z\in J$ with respect to
$\sigma$ there is a sequence of integers $n_{j}$
and sequences of balls $B_{r_{j}}(z)$ and 
topological disks $D_{j}(F^{n_{j}}(z))\supset B_{R'}(F^{n_{j}}(z))$
so that $F^{n_{j}}:B_{r_{j}}(z)\mapsto D_{j}(F^{n_{j}}(z))$
is a univalent function with bounded distortion. 
Let $z$ be a density point of $A$ with respect to $\nu$. 
The bounded distortion of $F^{n_{j}}$ implies that for every $\epsilon>0$ 
there exist $j$ so that
\[\frac{\nu(A\cap  D_{j}(F^{n_{j}}(z)) )}{\nu(  D_{j}(F^{n_{j}}(z)) )}~
\geq~ (1-\epsilon)\frac{\nu(A\cap B_{r_{j}}(z))}{\nu(B_{r_{j}}(z))}~\geq1-2\epsilon~.\]
By compactness, there exists $N=N(R')$ such that every disk $B_{R'}(y)$,
$y\in J$, is mapped onto $J$ by $F^{N}$. Hence, 
$$\lim_{j\rightarrow \infty}\nu(F^{n_{j}+N}(A))~\ge~
\lim_{j\rightarrow \infty}
\frac{\nu(A\cap  D_{j}(F^{n_{j}}(z)) )}{\nu(  D_{j}(F^{n_{j}}(z)) )}
\,\nu(J)~=~1~,$$
and the lemma follows.
\end{proof}

This concludes the proof of Theorem~\ref{theo:inv}.

\clearpage

\clearpage
\part{Geometry, rigidity, perturbations}

\section{Fractal structure}\label{sec:geometry}

In this Section we will prove that the geometry of the 
Julia sets satisfying appropriate summability conditions
is effectively fractal and self-similar.
Namely, every sufficiently small ball shrinks under the pull-backs and hence
its geometry is infinitely many times reproduced at different scales.
Moreover, it is ``usually'' (i.e. around most points and for most scales)
reproduced with bounded distortion.

\subsection{Average contraction of preimages}

\begin{prop}\label{prop:goodcover}
Suppose that a rational function $F$ satisfies the summability condition
with an exponent
$$\alpha~<~\fr{2}{\mmax+2}~,$$
and the Julia set is not the whole sphere.
Then there is $p<2$ such that
for every sufficiently small ball $B$ with center on the Julia set
$$\sum_{n}\,\sum_{F^{-n}}~\br{\diam\br{F^{-n}B}}^{p}
~\degg{F^n}{F^{-n}B}~<~\infty~,$$
where $\degg{F^n}{F^{-n}B}$ denotes the degree of $F^n$ restricted to 
the connected component ${F^{-n}B}$
of the preimage of $B$ under $F^{-n}$.
\end{prop}


We continue to work  with sequences 
$\{\alpha_n\}$, $\{\gamma_n\}$, $\{\delta_n\}$ of   Lemma~\ref{lem:techseq}.
To control the diameters we will need a new decomposition procedure.

\paragraph{Local Analysis.}

First we prove the analogues of Lemma~\ref{lem:1t} and Lemma~\ref{lem:3t}.

\begin{lem}\label{lem:1again}
Suppose that
\begin{description}
\itemm{\em{1)}}~Shrinking neighborhoods~$U_k$ for $B_{4r}(z)$, $1\le k< n$, avoid 
critical points and $(r)^{\mu_1}\,<\,R$\,,
\itemm{\em{2)}}~a critical point~$c_{2}\,\in\,U_n$~,
\itemm{\em{3)}}~a critical point~$c_{1}\,\in\,B_{r}(z)$~.
\end{description}
To simplify notation set  $\mu_i\,:=\,\m{c_i}$, 
$r_2:=\br{\diam\br{U_n}}$ and, for consistency, $r_{1}:=r$.

Then
$$ (r_2)^{\mu_2}~<~(r_1)^{\mu_1}~(\gamma_n)^{-\mmax}~,$$
in particular, $~(r_2)^{\mu_2}~<~R~.$
\end{lem}
\begin{proof}
First note, that $F^nc_2\in B_{4r\del_n}$, hence
$$\dist{F^nc_2,c_1}~\le~\dist{F^nc_2,z}\,+\,\dist{z,c_1}
~\le~5\,r_1~<~5 R^{1/\mmax}~,$$
and by the choice of $R$
we have 
$\abs{F'(F^{n}c_2)} \stackrel{M}{\asymp} \dist{F^nc_2,c_1}^{\mu_1-1}$. 
Therefore,
$$\abs{\br{F^{n-1}}'\br{Fc_2}}^{-1}~\le
~{M\dist{F^nc_2,c_1}^{\mu_1-1}}~\abs{\br{F^{n}}'\br{Fc_2}}^{-1} 
~\le~\fr{M(5r_1)^{\mu_1-1}}{\sigma_n}~.$$ 

We recall that $U_{n}\supset U_{n-1}'=F(U_{n})$.
By the Koebe distortion theorem
(see Lemma~\ref{lem:koeb}) applied 
to the conformal map 
$F^{-(n-1)}:\,B_{4r_1\Delta_{n-1}}(z)\,\to\, U_{n-1}$
we obtain that
\begin{eqnarray*}
\diam\br{U_{n-1}'}&\le&2\,\frac{(1-\delta_n)(2-\delta_n)}{\delta_{n}}
~\Delta_{n-1}\,4r_1~\abs{\br{F^{n-1}}'(Fc_2)}^{-1}\\
&\le&\frac{16\,r_1}{\delta_{n}}
~\fr{M(5r_1)^{\mu_1-1}}{\sigma_n}\\
&\le& 16^{\mmax}~M~(\alpha_n)^{-2}~(r_1)^{\mu_1}
~(\gamma_n)^{-\mmax}\\
&\le& (r_1)^{\mu_1}~(\gamma_n)^{-\mmax}
~(\alpha_n)^{-1}~.
\end{eqnarray*}
The last inequality is true by our choice of $\alpha_{n}$
and $R$, see  condition {\rm (ii)} in Section~\ref{sec:const}. 
In particular,
$\diam\br{U_{n-1}'}\,<\,(r_1)^{\mu_1}\,<\,R$ and  again
by condition {\rm (i)} of Section~\ref{sec:const} we have that
\begin{eqnarray*}
(r_2)^{\mu_2}&\le&M\,\diam\br{U_{n-1}'}\\
&\le& M~(r_1)^{\mu_1}~(\gamma_n)^{-\mmax}
~(\alpha_n)^{-1}\\
&\le&(r_1)^{\mu_1}~(\gamma_n)^{-\mmax}~ <~R,
\end{eqnarray*}
which completes the proof.
\end{proof}

\begin{lem}\label{lem:3again}
Suppose that
\begin{description}
\itemm{\em{1)}}~shrinking neighborhoods~$U_k$ for $B_{4r}(z)$, $1\le k< n$, avoid 
critical points and $r\,<\,R'$~,
\itemm{\em{2)}}~a critical point~$c_{2}\,\in\,U_n$~.
\end{description}
Set $\mu_2:=\m{c_2}$ and $r_2:=\br{\diam\br{U_n}}$.
For consistency, put $r_{1}:=r$.
Then
$$(r_2)^{\mu_2}~<~(\gamma_n)^{-\mmax}~,$$
and $(r_2)^{\mu_2}~<~R~.$
\end{lem}

\begin{proof}
Applying the Koebe distortion 
Lemma~\ref{lem:koeb}
we obtain that
\begin{eqnarray*}
\diam\br{U_{n-1}'}
&\le&2\,\frac{(1-\delta_n)(2-\delta_n)}{\delta_{n}}
~\Delta_{n-1}\,4r_1~\abs{\br{F^{n-1}}'(Fc_2)}^{-1}\\
&\le&\frac{16\,r_1}{\delta_{n}}
~\fr{\supF}{\sigma_n}\\
&\le&16~R'~\supF~\br{\alpha_n}^{-2}~\br{\gamma_n}^{-\mmax}\\
&\le&R~{M}^{-1}~\br{\gamma_n}^{-\mmax}~\le~R~.
\end{eqnarray*}
The last inequality is true by the choice of $R'$.
Particularly, $U_{n-1}'$ is close to $Fc_2$ and
\begin{eqnarray*}
(r_2)^{\mu_2}&\le&M\,\diam\br{U_{n-1}'}\\
&\le& M~{M}^{-1}~R~\br{\gamma_n}^{-\mmax}\\
&=&R~\br{\gamma_n}^{-\mmax}~.
\end{eqnarray*}
\end{proof}

\paragraph{Proof of Proposition~\ref{prop:goodcover}.}
Let $z$ be  a point from the Julia set
and fix an inverse branch of $F^{-n}$ so that
$F^{-n}(z)\mapsto\dots\mapsto F^{-1}(z)\mapsto z$.
Next, take a ball $B=B_{r_1}(z)$ of sufficiently small radius $r_1<R'$
and consider the shrinking neighborhoods for the  $4$ times larger
ball
$B_{4r_1}(z)$.
Let $k_1$ be the first time when $ U_{k_1}$ catches a  critical point
$c_{2}$.
Then, by Lemma~\ref{lem:3again}, $r_2:=\diam\br{F^{-k_1}B_{r_1}}$,
we have that 
$$(r_2)^{\mu_2}~<~(\gamma_{k_1})^{-\mmax}~.$$

Consider now the shrinking neighborhoods for the ball $B_{4r_2}(z_2)$ with
$z_2:=F^{-k_1}z$.
Let $k_2$ be the first time when  $U_{k_2}$ hits a critical point $c_{3}$.
Again, by Lemma~\ref{lem:1again}, $r_3:=\diam\br{F^{-k_2}B_{r_2}}$,
we obtain that 
$$(r_3)^{\mu_3}~<~(r_2)^{\mu_2}~(\gamma_{k_2})^{-\mmax}~.$$

We continue in the same fashion, taking shrinking neighborhoods
for  $B_{4r_3}(z_3)$ with 
$z_3:=F^{-k_2}z_2$, and so on.
Observe, that during the construction we always  have
$$F^{-(k_1+k_2+\dots+k_j)}B~\subset
~F^{-(k_2+\dots+k_j)}B_{r_2}(z_2)~\subset
~\dots~\subset~B_{r_{j+1}}(z_{j+1})~,$$
and there is a bound for the degree:
$$\degg{F^{(k_1+k_2+\dots+k_j)}}{F^{-(k_1+k_2+\dots+k_j)}B}
~\le~\br{\mmax}^j~.$$

We can repeat the above construction until  we meet
 a  ball $B_{4r_l}(F^{-k_{l-1}}z)$ whose  
shrinking neighborhoods do not contain critical points. This  means that
the ball $B_{2r_l}(z_l)$ can be pulled back univalently along the 
considered branch. We will call $z_{l}$ a terminal point.
In more general notation, $y:=z_{l}$ with parameters
$r(y)=r_l$, $l(y)=l$, $c_y=c_l$.

Now, we look at the backward orbit of $z$ for all possible 
inverse branches of $F$ and
denote  by ${\Cal Y}(z)$ the set of all terminal points.

By the Koebe distortion theorem,
\begin{eqnarray*}
\diam\br{F^{-n}B}&<&
\diam\br{F^{-m}B_{r_l}(z_l)}\\
&<&16~\abs{\br{F^{-m}}'(z_l)}~r_l\\
&=&16~r_l~\abs{\br{F^{m}}'(x)}^{-1}~,
\end{eqnarray*}
where $m=(n-k_1-\dots-k_{l-1})$ and $x=F^{-m}z_l=F^{-n}z$.
Note, that $x\in\xx{}{z_l,r_l}$ in the terminology of 
Proposition~\ref{prop:poincare}, and
$\degg{F^{n}}{F^{-n}B}~\le~\br{\mmax}^{l-1}$.

Now, using the result of Proposition~\ref{prop:poincare}
we can expand (for $p<2$ close to $2$)
\begin{eqnarray*}
\sum_{n}&{\displaystyle\sum_{F^{-n}}}&
\br{\diam\br{F^{-n}B}}^{p}\,\degg{F^n}{F^{-n}B}\\
&<&
\sum_{y\in{\Cal Y}(z)}\,\sum_{x\in\xx{}{y,r(y)}}
\,16^p\,\br{r(y)}^p\,\abs{\br{F^{n(x)}}'(x)}^{-p}\,\degg{F^n}{F^{-n}B}\\
&<&
\sum_{y\in{\Cal Y}(z)}\,16^p\,\br{r(y)}^p\,
C\,\br{r(y)}^{p(\m{c_{y}}/\mmax-1)}\,\degg{F^n}{F^{-n}B}\\
&<&
16^p\,C\,\sum_{y\in{\Cal Y}(z)}
\,\br{r(y)}^{p\m{c_{y}}/\mmax}\,\br{\mmax}^{l(y)-1}\\
&<&
16^p\,C\,\sum_{y\in{\Cal Y}(z)}\,
\br{\gamma_{k_{l-1}}^{-\mmax}\dots\gamma_{k_1}^{-\mmax}}^{p/\mmax}
\,\br{\mmax}^{l-1}\\
&<&
16^p\,C~\sum_{l,k_1,\dots,k_{l-1}}~(2 \deg F)^{l}
~\br{\gamma_{k_1}\dots\gamma_{k_{l-1}}}^{-p}~\br{\mmax}^{l-1}\\
&\le&
16^p\,C~\sum_l\,\br{2 \deg F\,\mmax\sum_k\gamma_k^{-p}}^l\\
&<&C\,16^p\,\sum_l\br{\fr12}^l~=~C\,16^p~<~\infty
~ ,
\end{eqnarray*}
which  proves  Proposition~\ref{prop:goodcover}.
{\begin{flushright}$\Box$\end{flushright}}

Note, that  substuting into the last formula Lemma~\ref{lem:main2}
(instead of  Proposition~\ref{prop:poincare}), we 
can arrive at a better estimate
(where the sum is taken only over some preimages):
\begin{coro}\label{cor:goodcover}
Taking $\beta=\mmax\alpha/(1-\alpha)$ and using the notation above,
we get
$$
\sum_{y\in{\Cal Y}(z)}\,\sum_{x\in\typei^{r(y)}(y)}
\br{\diam\br{F^{-n(x,z)}B}}^{\beta}
~<~\infty~.$$
\end{coro}

\subsection{Contraction of preimages}

\begin{prop}\label{prop:shrink}
Suppose that a rational function $F$ satisfies 
the summability condition with an exponent
$$\alpha~\le~1~.$$
Then there exist a positive sequence $\brs{\tilde\omega_n}$,
summable with an exponent $-\beta\,:=\,-\fr{\mmax\alpha}{1-\alpha}$:
$$\sum_{n}\,\br{\tilde\omega_n}^{-\beta}
~<~\infty~,$$
such that for every sufficiently small 
(of radius less than $R'$)
ball $B$ centered on the Julia set, every $n$,
and every branch of  ${F^{-n}}$ we have
$$\diam\br{F^{-n}B}~<~(\tilde\omega_n)^{-1}~.$$
\end{prop}

\begin{rem}
The proof of Proposition~\ref{prop:shrink} given below will actually imply that 
$$\diam\br{F^{-n}B}~<~\const
~(\tilde\omega_n)^{-1}~\br{\diam\br{B}}^{1/\mmax}~.$$
\end{rem}

Also, for any periodic point $z:~F^{k}(z)=z$,
by the proposition above we can find such $n$ that for
the branch of $F^{-kn}$, fixing $z$, and a small ball $B(z,\rho)$ one has
$$F^{-nk}B(z,\rho)~\subset~B(z,\rho/2)~.$$
By a standard use of the Schwartz lemma, the latter implies
$\abs{\br{F^k}'(z)}>1$, and we arrive at the following 

\begin{coro}\label{cor:nocremer}
Under the assumptions as above, $F$ has no Cremer points.
\end{coro}
\begin{proof}
To prove Proposition~\ref{prop:shrink},
take a ball $B_{r}(z)$ of a small radius $r_1<R'$ 
and proceed as in the proof of
Proposition~\ref{prop:goodcover}
-- we preserve the notation.
Then, with the help 
of Proposition~\ref{prop:fatouexpan},
(the sequence $\brs{\omega_n}$ was constructed there)
we obtain that
\begin{eqnarray*}
\diam\br{F^{-n}B}&<&
\diam\br{F^{-m}B_{r_l}(z_l)}\\
&<&16~\abs{\br{F^{-m}}'(z_l)}~r_l\\
&<&16~r_l~(r_l)^{\m{c_l}/\mmax-1}~(\omega_m)^{-1}\\
&<&16~(r_l)^{\m{c_l}/\mmax}~(\omega_m)^{-1}\\
&<&16~\br{(\gamma_{k_1})^{-\mmax}\dots
(\gamma_{k_l})^{-\mmax}}^{1/\mmax}~(\omega_m)^{-1}\\
&<&16~(\gamma_{k_1})^{-1}\,\dots\,
(\gamma_{k_l})^{-1}~(\omega_m)^{-1}~,
\end{eqnarray*}
where $k_1+\dots+k_l+m=n$.

It means that setting
$$\tilde\omega_n~:=~\inf
\brs{\gamma_{k_1}\,\dots\,\gamma_{k_l}\,\omega_m\,/\,16~:
~~k_1+\dots +k_l+m=n}~,$$
we have that
$$\diam\br{F^{-n}B}~<~(\tilde\omega_n)^{-1}~.$$


On the other hand
(for $-\beta\,:=\,-\fr{\mmax\alpha}{1-\alpha}$) 
\begin{eqnarray*}
\sum_{n}(\tilde\omega_n)^{-\beta}
&<&16^{\beta}\,\br{\sum_m(\omega_m)^{-\beta}}
\cdot\sum_{l=0}^{\infty}
\br{\sum_k(\gamma_k)^{-\beta}}^l\\
&<&16^{\beta}\,{\sum_m(\omega_m)^{-\beta}}
\cdot\sum_{l=0}^{\infty}\br{\fr12}^l
~=~16^{\beta}\,2\,{\sum_m(\omega_m)^{-\beta}}~<~\infty~.
\end{eqnarray*}
which completes the proof of Proposition~\ref{prop:shrink}.
\end{proof}

\subsection{Most points go to large scale infinitely often}
We will prove that the Hausdorff dimension of points which do not
``go to a large scale infinitely often'' is small provided
$F$ satisfies the summability condition. This should be compared with
Proposition~\ref{prop:largescale} 
where it is shown that most points go to a large scale infinitely often 
with respect to conformal measure.

The definition of the subset of points in $J$ which infinitely 
often go to the large scale of size $R'/2$ 
is as follows: 
$$\Jls \,:=\,\brs{z\in J:\, \exists~n_j\to\infty, {\em~with~}F^{n_j} 
{\em~univalent~on ~} F^{-n_j}\br{B\br{F^{n_j}x,R'/2}}}~.$$
Note that the value of $R'$ is already fixed and does not depend on a
point.

\begin{prop}\label{prop:largescale2}
Suppose that a rational function $F$ satisfies the summability condition
with an exponent $\alpha<1$, then
$$\HD(J\setminus \Jls )~\le~\fr{\mmax\alpha}{1-\alpha}~.$$
\end{prop}
\begin{proof}
The proof is a modification of the proof of 
Proposition~\ref{prop:largescale}.
Denote $\beta:=\fr{\mmax\alpha}{1-\alpha}$.

Take a finite cover $\{B_j:j=1,\dots, K\}$
of the Julia set by balls of radii $R'/2$ centered at points $w_j\in J$.
For  every $x\in J$ and every $k\in \N$, we 
decompose the sequence $x,\dots,F^k(x)$ into blocks of new types 
$1^{*}, 3^{*}$, and blocks of old types $1,2,3$. An inductive
procedure ascribing a code to the sequence $x,\dots ,F^{k}(x)$ will be
defined  only for  preimages of  the center of the ball $B_j\ni F^{k}(x)$. 
By the definition, the sequence $x,\dots,F^k(x)$ inherits the code of
the corresponding sequence of preimages $F^{-k}(w_{j}),\dots ,w_{j}$.

We start by defining blocks of type $1^{*}$ and $3^{*}$ for the preimages 
of $w_{j}$. To this aim
we invoke the inductive procedure from Proposition~\ref{prop:goodcover}.
Namely, we start by picking a ball
$B_j\subset B_{R'}(F^k(x))$, denoting $z_1:=w_j$, $r_1:=R'/2$, 
and considering the shrinking neighborhoods for the $4$ 
times larger  ball $B_{4r_1}(z)$.
Let $k_1$ be the first time when $U_{k_1}$ hits a  critical point
$c_{2}$.
We set $r_2:=\diam\br{F^{-k_1}B_{r_1}}$, $z_2:=F^{-k_1}(z_1)$,
and proceed by induction.
The construction is repeated until we meet
a ball $B_{4r_l}(F^{-k_{l-1}}z)$ whose  
shrinking neighborhoods do not contain critical points. 
This  means that the ball $B_{2r_l}(z_l)$ can be pulled back 
univalently along the corresponding branch  of $F^{-k+(k_{1}+\dots+k_{l})}$.
Summarizing, our construction leads to a decomposition
of the sequence   $z_l,F(z_l),\dots,z_1$
into  blocks  $1^{*}\dots 1^{*}3^{*}$.
We see that the symbol $3^{*}$ 
stands for the initial  sequence
of preimages of type $3$ with $r=2R'$ (see Definition~\ref{defi:3t}). 
The terminal point of the type $3^{*}$ sequence is $z_{2}$. After, only
type $1^{*}$ sequences are allowed with terminal points $z_{3}, \dots ,z_{l}$,
respectively.

Having defined $z_{l}$, we apply to
$F^{-k+k_1+\dots+k_{l-1}}(z_l),\dots,z_l$
the inductive procedure 
with a stopping rule (see Section~\ref{sec:stop}).
This yields 
a decomposition of the sequence  into blocks of the form $21\dots3$ or $2$.
Finally,  we can represent the orbit $F^{-k}(z_1),\dots,z_1$,
as a sequence of blocks $21\dots111^{*}1^{*}\dots1^{*}3^{*}$.
By our convention, the orbit  $x,\dots,F^k(x)$ has the same decomposition.
Note that if we have no blocks of type $1^{*}$, i.e. $z_l$ coincides with 
$F^{-k}(w_{j})$, 
then the ball of radius $R'$ can be pulled back univalently 
along the corresponding branch of $F^{-k}$, and we have
no blocks $1$ either. This means that the sequence $x, \dots, F^{k}(x)$
is  encoded as  $2$.
Note also that the corresponding endpoints of blocks from the orbits
 $x,\dots,F^k(x)$ and $F^{-k}(z_1),\dots,z_1$ are $R$-close to each other.

Following the proof of  Proposition~\ref{prop:goodcover},
we denote  by ${\Cal Y}(w_j)$ the set of all possible terminal points
$z_l$ for all inverse branches of $F$ defined on the ball $B_j$. 
Introducing  more general notation, we set $y:=z_{l}$, 
$r(y)=r_l$, $l(y)=l$, and $c_y=c_l$.

Denote by ${\cal C}_{x}$ the set of all codes obtained for $x$.
If for some point $x$ we get infinitely many different type  $2$ sequences
then $x$ must belong to $\Jls$.
Indeed, a type $2$ sequence means that an $R'$-ball
around a point $R'/2$-close to some image of $x$ can be  pulled back
univalently. Hence, the same is true for $R'/2$-ball around the image of $x$.

Therefore, if  $x\in J\setminus \Jls $ then $x$ is  
a terminal point of an infinite number of
sequences $211\dots111^{*}1^{*}\dots1^{*}1^{*}3^{*}$ 
with only a finite number of choices for type $2$ blocks.
Let $k(x)$ be a minimal number for which  infinitely many 
sequences from ${\cal C}_{x}$ have the same type $2$ block of  
length $k(x)$. Denote $X_k:=\brs{x:k(x)=k}$ and observe that 
the sets $\{X_{k}:k=0,1,\dots\}$ form  
a countable partition of $J\setminus \Jls$. 

Obviously, for any Borel set
$X\subset J$ we have
$\HD(F^kX)=\HD(X)=\HD(F^{-k}X)$.
Since $F^k(X_k)\subset X_0$ and consequently  
$J\setminus \Jls\subset\cup_k F^{-k}(X_0)$, 
it is sufficient to prove that $\HD(X_0)\le \mmax\alpha/(1-\alpha)$. 

Every  point  $x\in X_0$  must  be a  terminal point 
for infinitely many different subsequences of the form
 $1\dots11^{*}\dots1^{*}3^{*}$,
containing  at least one block $1^{*}$.
Thus every  point  $x\in X_0$ is covered by infinitely many preimages
$$F^{-n(v,w_j)}(B_j)~,~~~j= 1,\dots, K,~~~y\in{\cal
Y}_j,~~~v\in\typei^{r(y)}(y)~.$$
Corollary~\ref{cor:goodcover} implies that for every $j=1,\dots, K$ and
$\beta=\mmax\alpha/(1-\alpha)$,
$$\sum_{y\in{\Cal Y}(w_j)}\,\sum_{v\in\typei^{r(y)}(y)}
\br{\diam\br{F^{-n(v,w_j)}B_{j}}}^{\beta}~<~\infty~.$$
We conclude that 
$\HD{(X_0)}\le\fr{\mmax\alpha}{1-\alpha}$ which proves the proposition.
\end{proof}

\section{Dimensions and conformal measures}\label{sec:dim}
\subsection{Fractal dimensions}

First we will remind the definitions of various dimensions, used in this paper.
For properties of the Hausdorff and Minkowski measures, contents, and dimensions
one can consult the monographs
\cite{mattila} and \cite{federer}.


Assume that we are given a compact subset $K$ of the complex plane
(or a complex sphere with the spherical metric).

\begin{defi}
For positive $\delta$ the {\em Hausdorff measure} $\ha_\delta$ is defined by
$$\ha_\delta(K)
~:=~\lim_{\rho\to0}~\inf_{\Cal B_\rho}
~\sum_{B\in\Cal B_\rho}~r(B)^{\delta}~,$$
the infimum taken over all covers $\Cal B_\rho\,=\,\{B\}$
of the set $K$ by balls $B$ of radii $r(B)\le\rho$.
\end{defi}
Usually the measure above is normalized by some factor, depending on $\delta$,
but this is not necessary for our purposes.

It is easy to show that there exists some number
$\delta'\in[0,2]$, such that  $\ha_\delta(K)$ is infinite for $\delta<\delta'$
and zero for $\delta>\delta'$.
The latter is called the Hausdorff dimension:

\begin{defi}
The {\em Hausdorff dimension} of a set $K$ is defined by
$$\HD(K)~:=~\inf~\brs{\delta\,:~\ha_\delta(K)=0}~.$$
The {\em Hausdorff dimension} of a Borel measure  $\nu$ is defined 
as the infimum of the dimensions of its Borel supports:
$$\HD(\nu)~:=~\inf~\brs{\HD(E)\,:~E{\mathrm~is~Borel~and~}\nu(E^c)=0}~.$$
\end{defi}

The upper and lower Minkowski dimensions can be defined similarly using the
corresponding Minkowski contents.
Equivalently, one can take a shortcut and
define them as follows:

\begin{defi}
Let $N(K,\rho)$ be the minimal number of the balls of radius $\rho$
needed to cover $K$.
The {\em upper} and {\em lower} {\em Minkowski dimensions}
are defined as
$$\begin{array}{ccc}
\MDsup(K)&:=&\limsup_{\rho\to0}~\frac{\log N(K,\rho)}{\log 1/\rho}~,\\
\MDinf(K)&:=&\liminf_{\rho\to0}~\frac{\log N(K,\rho)}{\log 1/\rho}~.
\end{array}$$
\end{defi}
If those dimensions coincide, their common value 
is called the {\em Minkowski dimension} $\MD(K)$.

\begin{rem}
Since we restricted ourselves to a smaller collection of coverings,
than in the definition of the Hausdorff measure, one clearly has
$$\HD(K)~\le~\MDinf(K)~\le~\MDsup(K) ,$$
for arbitrary compact set $K$.
\end{rem}

In the absence of the dynamics, the Whitney exponent
can be regarded as a substitute for the Poincar\'e exponent.
One can decompose domain $\Omega\,:=\,\C\setminus K$ in the complex plane
into a collection $\brs{Q_j}$ of non-overlapping dyadic squares 
so that $\dist{Q_j,K}\,\asymp\,\diam(Q_j)$ up to a constant of $4$
(consult \cite{Stein} for this classical fact, the Whitney decomposition).

\begin{defi}
{\em Whitney exponent} is defined as an exponent of convergence
$$\dwhit(K)~:=~\inf~\brs{\delta\,:
~\sum_{Q_j:\,\diam(Q_j)\le1}\,\diam(Q_j)^{\delta}\,<\,\infty}~.$$
\end{defi}
Note that in the Whitney decomposition the smaller the squares, the closer they are to the
set, so to describe its geometry it is enough to work with the small squares only.
Thus the large squares are dropped from the series above so that it becomes convergent.

Clearly, this definition admits the following integral reformulation
(and hence does not depend on the choice of Whitney decomposition):
$$\dwhit(K)~:=~\inf~\brs{\delta\,:
~\int_{\Omega}~\dist{z,K}^{\delta-2}\,dm(z)\,<\,\infty}~,$$
where $m$ denotes area, and we use the spherical metric.
One can restrict integration to some neighborhood of $K$
(and should do so if working with the  planar metric).
\begin{rem}
The definitions of Whitney and Poincar\'e exponents assume 
that the complement of $K=J$ is non-empty.
Should $K=J$ coincide with the whole sphere,
we set $\dwhit(K)\,=\,\dpoin(K)\,:=\,2$.
\end{rem}

\subsection{Multifractal analysis}
The following is Lemma~2.1 in \cite{bishop-poincare}, where it was used
in similar situation, involving Poincar\'e exponent of a Kleinian group
and Minkowski dimension of its limit set.
We thank Chris Bishop for bringing it to our attention.

\begin{fact}\label{lem:bishop}
For any compact set $K$, $\dwhit(K)\,\le\,\MDsup(K)$.
If, in addition, $K$ has zero area, then $\dwhit(K)\,=\,\MDsup(K)$.
\end{fact}
By taking a cover of $J$ by a finite number of small balls
and applying Proposition~\ref{prop:goodcover}
to each of them,
we easily obtain the following

\begin{lem}\label{lem:hdim}
Suppose that a rational function $F$ satisfies the summability condition
with an exponent
$$\alpha~<~\fr{2}{\mmax+2}~.$$
If the Julia set is not the whole sphere,
then its Hausdorff dimension
is strictly less than $2$.
\end{lem}

It seems to be folklore   that for rational maps without neutral
orbits $\dpoin(J)=\dwhit(J)$. We were unable to find a reference
to this fact  in the literature and thus we supply the proof below.
The following is an analogue of Lemma~3.1 in \cite{bishop-poincare}:

\begin{lem}\label{lem:poinwhit}
For any rational function $F$ 
without Siegel disks, Herman rings, or parabolic points one has
$$\dpoin(J)~=~\dwhit(J)~.$$
\end{lem}

\begin{rem}
The proof below can be modified to work for parabolic points as well.
However, in the presence of Siegel disks or Herman rings 
the introduced version of Poincar\'e series
does not work well.

Under an additional assumption that the Julia set has zero area, the
lemma 
together with Fact~\ref{lem:bishop} imply
that the Poincar\'e exponent coincides with the upper Minkowski dimension.
\end{rem}
\begin{proof}
Fix points $\{z_j\}$ used in the definition of $\dpoin(J)$
--  one inside each cycle of periodic Fatou components.
From Lemma~7 of \cite{ceh} (Lemma~\ref{lem:ourkoebe}) follows  
that for any $y\in F^{-n} z_j$
one has
\begin{equation}\label{eq:dynmetric}
\dist{y,J}~\asymp~\abs{\br{F^{n}}'(y)}^{-1}~,
\end{equation}
up to a constant depending on $z_j$ only.


Knowing that only (super) attractive  Fatou components are possible, 
we can choose ``fundamental'' domains $z_j\in U_j$, so that their preimages under
all possible branches of $F^{-n}$ are disjoint and cover almost all
of some neighborhood $U$ of $J$ inside the Fatou set.
Also $z_j$ and then $U_j$ can be chosen so that under iteration
critical points never enter some neighborhoods of $U_j$, and hence
by distortion theorems, up to a constant $\const(z_j,U_j)$
\begin{equation}\label{eq:dynmetrica}
\dist{x,J}~\asymp~\abs{\br{F^{n}}'(x)}^{-1}~\asymp
~\abs{\br{F^{n}}'(y)}^{-1}~,
\end{equation}
for any $x\in F^{-n}U_j$ and $y$ being the corresponding
preimage of $z_j$: $y\in F^{-n} z_j$.

Hence, for any $\delta\ge0$ we can write
(here $V\in F^{-n}U_j$ means that $V$ is one of the components of connectivity
of the latter)
\begin{eqnarray*}
\int_{U}~\dist{x,J}^{\delta-2}\,dm(x)&=&
\sum_j\,\sum_{n=1}^{\infty}\,\sum_{V\in F^{-n}U_j}
\int_{V}~\dist{x,J}^{\delta-2}\,dm(x)\\
&\asymp&\sum_j\,\sum_{n=1}^{\infty}\,\sum_{V\in F^{-n}U_j}\,
\int_{V}~\abs{\br{F^n}'(x)}^{2-\delta}\,\abs{\br{F^n}'(x)}^{-2}\,dm(F^n(x))\\
&\asymp&\sum_j\,\sum_{n=1}^{\infty}\,\sum_{V\in F^{-n}U_j}\,
\int_{V}~\abs{\br{F^n}'(F^{-n}z_j)}^{-\delta}\,dm(F^n(x))\\
&\asymp&\sum_j\,\sum_{n=1}^{\infty}\,\sum_{y\in F^{-n}z_j}\,\abs{\br{F^n}'(y)}^{-\delta}
\int_{U_j}\,dm(z)
~\asymp~\Sigma_\delta(J,\{z_j\})~,
\end{eqnarray*}
which clearly implies the desired equality.
\end{proof}

The definitions  and our discussion so far imply two chains
of inequalities: 
\[ \HD(J)\leq \MDinf(J)\leq \MDsup(J)\;\]
and
\[\dpoin(J)=\dwhit(J)\, \leq \,\MDsup(K)~.\]

\begin{prop}[Gauge function estimate for conformal measure]\label{prop:confgauge}
Suppose that a rational function $F$ satisfies the summability condition
with an exponent
$$\alpha~<~\fr{q}{\mmax+q}~,$$
and $\nu$ is $q$-conformal measure with no atoms at critical points.
Then for $\nu$-almost every $x\in J$ and any $\epsilon>0$
there are  constants $C_x>0$ and $C_{x,\epsilon}>0$,
such that for any ball $B(x,r)\,,~r<1,$ centered at $x$ one has
$$C_{x}\,r^q~<~\nu(B(x,r))~<~C_{x,\epsilon}\,r^{q-\epsilon}~.$$
\end{prop}
\begin{proof}
It is a straight-forward use of Proposition~\ref{prop:freq}.
\end{proof}

It is easy to see, that the proposition above implies that
$\HD(J)\,\ge\,\HD(\nu)\,=\,q$.
Combining this
with the Corollary~\ref{cor:confnoatom}, we obtain the following
\begin{coro}\label{cor:hdconf}
Assume that $F$ satisfies the summability condition with an exponent
$$\alpha~<~\fr{\dpoin(J)}{\mmax+\dpoin(J)}~,$$
then 
$$\HD(J)~\ge~\dpoin(J)~=~\HD(\nu)~.$$
\end{coro}

\paragraph{Proof of Theorem~\ref{theo:dims}.}
By now we have
$$\dpoin(J)~=~\dconf(J)~=~\dwhit(J)~=~\MD(J)~\ge~\HD(J)~\ge~\dpoin(J)~,$$
and hence all these dimensions coincide,
provided that a rational function $F$ satisfies the  summability condition
with an exponent
$$\alpha~<~\fr{\dpoin(J)}{\mmax+\dpoin(J)}~.$$
By the work of M.~Denker, F.~Przytycki, and M.~Urbanski,
the hyperbolic and dynamical dimensions will also 
be equal to the dimensions above.


\section{Removability and rigidity} 

In this Section, we prove Theorems~\ref{theo:rem}, \ref{theo:rem1}, and \ref{theo:rem2}.

\subsection{Conformal removability and strong rigidity}

The notion of {\em conformal removability} (also called holomorphic removability)
appears naturally in holomorphic dynamics:
often one can show that two dynamical systems are conjugated by a homeomorphism
which is (quasi)conformal outside the Julia set,
and conformal removability of the latter ensures  global (quasi)conformality
of the conjugation.

\begin{defi}
We say that a compact set $J$
is {\em conformally removable}
if any homeomorphism of the Riemann sphere $\hat{\C}$,
which is conformal outside $K$,
is globally conformal
and hence is a M\"{o}bius transformation.
\end{defi}
The quasiconformal removability is defined similarly.
An easy application of the measurable Riemann mapping theorem
shows that the two notions are equivalent.
The problem of geometric characterization of  removable sets is open,
see \cite{josm} for discussion and relevant references.
Sets of positive area are non-removable,
as are Cartesian products of intervals with Cantor sets 
of positive length.
On the other hand, quasicircles and sets of $\sigma$-finite length are removable.
Note that there are removable sets of Hausdorff dimension $2$
and non-removable of dimension $1$.

In \cite{josm} a few geometric criteria for removability are given,
some close to being optimal and well-adapted for dynamical applications.
We will use the following fact (which is Theorem~5 in \cite{josm}):
\begin{fact}\label{fact:js}
Suppose that $F$ is a polynomial,
and $\{B_j\}$ is a finite collection of domains whose closure covers $J_F$.
Denote by $\{P_i^n\}$ the collection of all components of connectivity of pullbacks
$F^{-n}B_j$, and by $N(P_i^n)$ the degree of $F^n$ restricted to $P_i^n$.
Then the geometric condition,
\begin{equation}\label{jscondition}
\sum_{i,n}\,N(P_i^n)\,\diam\br{P_i^n}^2\ <\ \infty\ ,
\end{equation}
is sufficient for  conformal removability of the Julia set.
\end{fact}

If a polynomial satisfies the summability condition
with an exponent $\alpha\,<\,\fr{2}{\mmax+2}$
then Proposition~\ref{prop:goodcover} 
implies condition (\ref{jscondition})
for a cover by  sufficiently small balls $B_j$.
By  Fact~\ref{fact:js}, the Julia set is conformally removable
and Theorem~\ref{theo:rem} follows.
Similar reasoning works for every Julia set, 
which is a boundary of one of the Fatou components.

\subsection{Dynamical removability and rigidity}

The assumption that the Julia set coincides with the 
 boundary of one of the Fatou components is essential
for conformal removability.
Indeed, there are hyperbolic rational functions
with non-removable Julia sets.
An example of a non-removable Julia set, 
which is topologically a Cantor set of circles,
is constructed in \S11.8 of the book \cite{beardon-book-iteration}.
It is a classical observation that these type of  sets 
are not conformally removable.
An exotic homeomorphism is given by rotating
annuli between circles by a devil's staircase of angles:
the resulting homeomorphism is conformal on each annulus,
globally continuous since the devil's staircase is,
but clearly is not M\"obius (i.e. not globally conformal).

Even though such Julia sets are not conformally removable,
they will be removable for all ``dynamical'' conjugacies.
To make this precise,
consider a homeomorphism $\phi$ which conjugates 
a rational dynamical system $(\CC,F)$ to
another dynamical system $(\CC,G)$
and assume that $\phi$ is quasiconformal outside the Julia set $J$.

Recall a metric definition of quasiconformality
(which states that images of circles look like circles themselves):
a homeomorphism $\phi$ is quasiconformal,
if there is a constant $H$ such that for every  point $x\in \CC$
\begin{equation}\label{eq:qcdef}
\limsup_{r\to0}\frac{L_\phi(x,r)}{l_\phi(x,r)}~\le~H~<~\infty~,
\end{equation}
where
\begin{eqnarray*}
L_\phi(x,r)&:=&\mbox{sup}\,\brs{\abs{\phi(x)-\phi(y)}\,:~|x-y|\leq r}~,\\ 
l_\phi(x,r)&:=&\mbox{inf}\,\brs{\abs{\phi(x)-\phi(y)}\,:~|x-y|\geq r}~. 
\end{eqnarray*}

If a rational function $F$ is hyperbolic,
then every sufficiently small ball with center at the Julia set 
is mapped univalently by some iterate of $F$ to a large scale
with bounded distortion,
and the inequality (\ref{eq:qcdef}) holds by a compactness argument
implying a (global) quasiconformality of $\phi$.

For non-hyperbolic maps the property of 
``going to large scale with bounded distortion'' fails for many
small balls. In these circumstances one has to resort
to more subtle tools in the theory of quasiconformal maps.
A  theorem of great use for complex dynamical systems
was proved recently by J.~Heinonen and P.~Koskela
 in \cite{heinonen-koskela-def}.
They have shown, that for Euclidean spaces 
the upper limit ``$\limsup$'' in the metric definition of
the quasiconformality can be replaced by ``$\liminf$.''
J.~Heinonen and P.~Koskela's result was immediately applied
by F.~Przytycki and S.~Rohde \cite{przytycki-rohde-rigidity}
to deduce rigidity of Julia sets
satisfying the topological Collet-Eckmann condition (shortly TCE).
The argument of ~\cite{przytycki-rohde-rigidity} goes as follows:
for every point $x\in J$ there is 
a sequence of radii $r_j\to0$
such that the balls $B_{r_j}(x)$
are mapped by  some iterates of $F$ to a large scale
with bounded distortion (though no longer univalently but with uniformly
bounded criticality), and the inequality (\ref{eq:qcdef}) 
for $\liminf$ holds again by a compactness argument
implying  a (global) quasiconformality of $\phi$.

Rational maps which satisfy the summability condition have weaker properties
than TCE maps (in the unicritical case the latter class
is strictly smaller), so we need an even stronger theorem than that
of J.~Heinonen and P.~Koskela.
It is a well-known fact, that the 
metric definition (with ``$\limsup$'') of quasiconformality allows
for an exceptional set.
Partially motivated by the perspective applications to our paper,
S.~Kallunki and P.~Koskela \cite{kallunki-koskela}
established very recently
that one can also have an exceptional set in the ``$\liminf$'' definition
of quasiconformality.
The following is  Theorem 1 of \cite{kallunki-koskela}:
\begin{fact}\label{fact:kalkos}
Let $\Omega\subset{\R}^n$ be a domain and suppose that
$\phi:\Omega \rightarrow \phi(\Omega)\subset{\R}^n$ is a homeomorphism. 
If there is a set $E$ of $\sigma$-finite $(n-1)$- dimensional Hausdorff measure so that
$$\liminf_{r\to0}\frac{L_\phi(x,r)}{l_\phi(x,r)}~\le~H~<~\infty~,$$
for each $x\in\Omega\setminus E,$
then $\phi$ is quasiconformal in $\Omega.$ 
\end{fact}
This theorem fits very well into our framework.
By Proposition~\ref{prop:largescale}, if $F$ satisfies the 
summability condition with  an exponent $\alpha<\fr1{\mmax+1}$,
then except for a set $E$ of Hausdorff dimension $<1$ 
every  point $x\in J$ ``go to a large scale'' 
infinitely often.
More precisely, for every $x\in J$ there exists (a point-dependent)
sequence of radii $r_j\to 0$
such that the balls $B_{r_j}(x)$
are mapped by iterates $F^{k_j}$
to the large scale of size  $\asymp R'$ univalently
and with a bounded distortion. 
Thus for every $x\in J\setminus E$
(cf. \cite{przytycki-rohde-rigidity})
one has
\begin{eqnarray*}
\liminf_{r\to0}\frac{L_\phi(x,r)}{l_\phi(x,r)}&\le&
\liminf_{j\to\infty}\frac{L_\phi(x,r_j)}{l_\phi(x,r_j)}~\lesssim~
\liminf_{j\to\infty}\frac{L_\phi(F^{n_j}(x),R')}{l_\phi(F^{n_j}(x),R')}\cr
&\lesssim&\sup_{x\in J}\frac{L_\phi(x,R')}{l_\phi(x,R')}
~=:~H~<~\infty~,\end{eqnarray*}
the latter quantity being finite by a compactness argument.
We infer that the homeomorphism $\phi$ is globally quasiconformal,
thus deducing Theorem~\ref{theo:rem1}.

Consider now a quasiconformal homeomorphism $\phi$ 
which conjugates a rational dynamical system $(\CC,F)$ 
to another dynamical system $(\CC,G)$ and assume that
$F$ satisfies the (weaker) summability condition 
with an exponent $\alpha<\fr2{\mmax+2}$.

If $J\neq\CC$ then the area of the Julia set is zero,
by the Corollary~\ref{theo:dim}, and
an  invariant  Beltrami coefficient $\phi_{\mu}$ 
has to be supported on the Fatou set.
There are two interesting special settings 
when $\phi$ is automatically  a M\"obius transformation.
If $\phi$ is conformal outside the Julia set,
then the Beltrami coefficient is identically zero,
and $\phi$ is  M\"obius.
Also, if there is only one simply-connected Fatou component
(e.g. this is the case for polynomials with
all critical points in the Julia set),
it has to be super-attracting,
and by a standard argument it does not  support non-zero
$F$-invariant Beltrami coefficients,
and $\phi$ is M\"obius again.

If $J=\CC$ then by Proposition~\ref{prop:largescale}
 except for a set $E$ of Hausdorff Dimension $<2$, 
all points ``go to a large scale'' infinitely often.
Then a standard technique (see the proof of 
either Theorem 3.9 or Theorem 3.17 in \cite{mcm})
implies that the Beltrami coefficient $\mu_\phi$ has to be holomorphic,
and by Lemma~3.16 from \cite{mcm}, we have that  
$\mu_\phi\equiv 0$  or $F$ is a  double
cover of an integral torus endomorphism, i.e. it is a Latt\'es example.
This concludes the proof of Theorem~\ref{theo:rem2}.

\section{Integrability condition and invariant measures.}
A natural question arises whether a rational map $F$ 
has invariant measures absolutely
continuous with respect to conformal measures. 
We will make two different types of assumptions. Firstly, we make
a general assumption about  regularity of 
conformal measures. This will guarantee that a
conformal measure is not too singular with respect to
the corresponding  Hausdorff
measure (integrability condition). 
Secondly, we demand that $F$ has some expansion property,
usually given in the form of a suitable summability condition.

Let $\nu$ be a conformal measure with an exponent $\delta$
defined on the Julia set $J$.
Assume also that $\nu$ satisfies the 
uniform integrability condition with an exponent $\eta$.
We recall that this means that
there exist positive $C$ and $\eta$  such that for 
all positive integers $n$ and every $c\in \Crit $,
\begin{equation}\label{equ:inte2}
\int\frac{d\nu}{\dist{x,F^{n}(c)}^{\eta}}~<~C~<~\infty~.
\end{equation}
\subsection{Ruelle-Perron-Frobenius transfer operator.}
We study the existence of an absolutely continuous invariant measure
$\sigma$ through analysis of the Ruelle-Perron-Frobenius operator ${\cal L}$
which ascribes to every measure $\nu$, the density
of $F_{*}(\nu)$ with respect to $\nu$. The $N$-th iterate of 
${\cal  L}(\nu)$ evaluated at $z$ is equal to 
\[{\cal L}^{N}(\nu)(z):=\frac{dF^{N}_{*}(\nu)}{d\nu}=
\sum_{y\in F^{-N}(z)}\frac{1}{|(F^{N})'(y)|^{\delta}}\;.\]

For simplicity, we will drop $\nu$ from the notation of the
Ruelle-Perron-Frobenius operator.

\begin{prop}\label{prop:in1}
Assume that a rational function $F$ satisfies the summability
condition with an exponent
\[\alpha < \frac{\delta}{\delta+\mmax}\]
and $\nu$ is a  $\delta$-conformal measure  on $J$. 
Let $\Delta_{k}:=\dist{f^{k}(\Crit), z}$ and 
\[{\mathrm g}(z):= \sum_{k=1}^{\infty}
\gamma_{k}^{-\delta}\Delta_{k}^{-(1-\frac{1}{\mmax})\delta}\;.\]
There exists a positive constant $K$ such that for every 
$z\not \in \bigcup_{n=1}^{\infty} F^{n}(\Crit)$ and  every positive
integer $N$, 
\[{\cal L}^{N}(z) < K\; {\mathrm g}(z)~.\]
The sequence $\gamma_{k}^{-1}$ (defined in \mbox{Lemma~\ref{lem:techseq}})
is  summable with an exponent $\beta=\frac{\mmax\alpha}{1-\alpha}<\delta$.
\end{prop}
\begin{proof}
We will use the estimates of the Main Lemma for a modified
decomposition procedure, as in Subsection~\ref{sec:stop}. 

\paragraph{Construction.} Let $z\in X$ and a sequence
\[ F^{-N}(z),\dots, F^{-1}(z), z\]
form a chain of preimages, that is $F(F^{-i}(z))=F^{-i+1}(z)$ for
$i=1,\dots, N$ and $F^{-N}(z)\in X$. 
We decompose the chain into blocks of preimages
of the types $2$ and $1\dots13$. 
This is done using the procedure
of the Main Lemma with the following stopping rule: at the first
occurrence of a type $2$ block we stop the procedure of decomposing
the chain, so this type $2$ block can be arbitrarily long. 
See Subsection~\ref{sec:stop} for the details.

The inductive procedure gives a coding of backward orbits by sequences of 
$1$'s, $2$'s and $3$'s, with  only the following
three types of  codings  allowable: $2, 1\dots 3, 21\dots1 3$.
We recall that according to our convention, during the inductive
procedure we  put symbols in the coding from the right to the left. 

We attach to every chain of preimages of $z$  the sequence 
$k_l, \dots, k_0$ of the lengths  of the blocks of preimages of a
given type in its coding. Again our convention requires that
$k_{0}$ always stands  for the length of the rightmost block of preimages in
a coding. Clearly, $k_{0}+\cdots + k_{l}=N$.  
\paragraph{Estimates.}
We recall that the Main Lemma implies that every sequence of the form
$11\dots 3$ with the length of the corresponding pieces $k_{l},\dots
k_{0}$ yields expansion
\[\gamma_{k_{l}}\cdot\dots\cdot \gamma_{k_{0}}\;
\Delta_{k_{0}}^{1-\frac{1}{\mmax}}\;,\]
where $\Delta_{k_{0}}=\dist{f^{k_{0}}(\Crit), z}$.

For singleton sequences $\{2\}$ of the length $n$   
we have the following estimate,
\begin{equation}\label{equ:only}
\sum_{y\in F^{-n}(x)}\frac{1}{|(F^{n})'(y)|^{\delta}}\;\lesssim\;
\sum_{ F^{-n}}
\frac{\nu(F^{-n}(B_{R'}(x)))}{\nu(B_{R'}(x))} \leq 
\frac{1}{\nu(B_{R'}(x))}\leq K_{R'}\;,
\end{equation}
which is independent from $n$.

Consider now a set of all preimages $y\in F^{-N}(z)$. 
Let $11\dots 13$ denote the set of all points  $x\in
\bigcup_{i=1}^{N}F^{-i}(z)$ 
which are coded by maximal sequences of $1$'s and $3$'s.
We define $n_{x}$ by the condition $F^{n_{x}}(x)=z$.

Hence,
\begin{eqnarray*}
{\cal{ L}}^{N}(z) & = & \sum_{ 
      {\mbox {all
      codings}}}\frac{1}{|(F^{N})'(y)|^{\delta}} \\
&\leq & \sum_{n_{x}=1}^{N}\left(\sum_{11\dots13}
\frac{1}{|(F^{n_{x}})'(x)|^{\delta}}\right)
\left(\sum_{y\in F^{-N+n_{x}}}
      \frac{1}{|(F^{n})'(y)|^{\delta}}\right)\\
&\leq&\;K_{R'}\; \sum_{11\dots13}
(\gamma_{k_{l}}\cdot\dots\cdot \gamma_{k_{0}})^{-\delta}\;
\Delta_{k_{0}}^{-(1-\frac{1}{\mmax})\delta}\;\\
\end{eqnarray*}
Since for every sequence of $k_{l},\dots, k_{0}$ positive integers
there is at most $(2\deg F)^{l+1}$ sequences $11\dots 13$ with
the corresponding lengths of the pieces of type $1$ and $3$ (see the
estimate (\ref{equ:deg})), we obtain that
\begin{eqnarray*}
{\cal L}^{N}(z)
&\lesssim&
\sum_{l,k_{l},\dots, k_{0}} (4\deg F)^{l+1}
(\gamma_{k_{l}}\cdot\dots\cdot \gamma_{k_{0}}^{-\delta}\;
\Delta_{k_{0}}^{-(1-\frac{1}{\mmax})\delta}\;\\
&<& \sum_{l=1}^{\infty}
\left(4\deg F \sum_{k_{l}} \gamma_{k_{l}}^{-\delta}\right)
\cdot \dots\cdot \left(4\deg F \sum_{k_{1}} \gamma_{k_{1}}^{-\delta}\right)\cdot
\left(4\deg F \sum_{k_{0}} 
\gamma_{k_{0}})^{-\delta}\Delta_{k_{0}}^{-(1-\frac{1}{\mmax})\delta}\right)\\
&<&\; \sum_{l=1}\left(\frac{1}{4}\right)^{l}
\left(4\deg F \sum_{k_{0}} 
\gamma_{k_{0}}^{-\delta}\Delta_{k_{0}}^{-(1-\frac{1}{\mmax})\delta}\right)\\
&<& K\;  \sum_{k_{0}} 
(\gamma_{k_{0}})^{-\delta}\Delta_{k_{0}}^{-(1-\frac{1}{\mmax})\delta}\;.
\end{eqnarray*}
By Lemma~\ref{lem:techseq}, $\gamma_{k}$ is summable
with any exponent bigger than 
$-\frac{\mmax\alpha}{1-\alpha}$.
\end{proof}
\paragraph{Proof of Theorem~\ref{theo:integ}.}
Now it is a standard reasoning. 
By Proposition~\ref{prop:in1},
${\cal L}^{n}(z)\leq  \gr(z)$ and $\gr(z)\in L^{1}(\nu)$.
Hence,
\[\nu_{N}=\frac{1}{N}\sum_{i=0}^{N-1}F^{i}_{*}(\nu)\;\]
form a weakly compact set of probabilistic  measures absolutely continuous
with respect to $\nu$ with densities bounded by $\gr$.
A weak limit of $\nu_{N}$ gives an absolutely continuous
invariant measure.

Uniqueness and ergodicity follow by an argument presented in 
Section~\ref{sec:lyap}.

\subsection{Regularity of invariant densities}
\begin{coro}\label{theo:inreg}
Suppose that $F$ satisfies the summability
condition with an exponent
\[\alpha < \frac{\delta}{\delta+\mmax}\; .\]
Let $\eta_{1} > \eta_{0}:=\delta(1-\frac{1}{\mmax})$
and assume $\eta$-integrability of the unique $\delta_{Point}$-conformal measure  $\nu$ for every $\eta\in (0,\eta_{1})$.
Then for every $\zeta<\eta_{1}/\eta_{0}$ 
\[\int \left(\frac{d\sigma}{d\nu}\right)
^{\zeta}\, d\nu< \infty \;,\]
and for every Borel set $A$,
\[ \sigma (A)\leq K_{\zeta}
\nu(A)^{1-1/\zeta}\;.\]
\end{coro}
\begin{proof}
To prove $\zeta$-integrability of the density of 
$d\sigma/d\nu$ we use Proposition~\ref{prop:in1}. Recall that
$\Delta_{k}= \dist{f^{k}(\Crit), z}$.
We use the H\"{o}lder inequality
with the exponents $1/\zeta+1/\zeta'=1$  in the estimates
below.
\begin{eqnarray}\label{equ:above}
\sum_{k}\gamma_{k}^{-\delta}\Delta_{k}^{-(1-1/\mmax)\delta}& =  &
\sum_{k}\gamma_{k}^{-\delta/\zeta'}\;\frac{\gamma_{k}^{-\delta/\zeta}}
{\Delta_{k}^{(1-1/\mmax)\delta}}\\
& \leq & 
\left(\sum_{k}\gamma_{k}^{-\delta}\right)^{1/\zeta'}
\;
\left(\sum_{k}
\frac{\gamma_{k}^{-\delta}}{\Delta_{k}^{\delta(1-1/\mmax)\zeta}}\right)
^{1/\zeta}\nonumber\\
&\leq & K\;  
\left(\sum_{k}
\frac{\gamma_{k}^{-\delta}}{\Delta_{k}^{\eta}}\right)
^{1/\zeta}\nonumber
\end{eqnarray}
with $\eta\in [0,\eta_{1})$. Next, we integrate the density $d\sigma/d\nu$ to the power $\zeta$.
Proposition~\ref{prop:in1} and the inequality
(\ref{equ:above}) imply that   
\begin{eqnarray*}
\int \left(\frac{d\sigma}{d\nu}\right)
^{\zeta}\, d\nu &\lesssim &
\int
\left(\sum_{k}
\gamma_{k}^{-\delta}\Delta_{k}^{-(1-1/\mmax)\delta}\right)^{\zeta}\, d\nu\\
& \lesssim & \sum_{k}
\gamma_{k}^{-\delta}\int \Delta_{k}^{-\eta} d\nu < \infty \;.
\end{eqnarray*}
To prove the last estimate of the Corollary, we apply once more
the H\"{o}lder inequality with the exponents $\zeta$ and $\zeta'$.
\begin{eqnarray*}
\sigma(A) &=& \int_{A} \frac{d\sigma}{d\nu}\,  d\nu 
\leq \left(\int_{A}d\nu\right)^{1/\zeta'}
\left(\int_{A}
  \left(\frac{d\sigma}{d\nu}\right)^{\zeta} d\nu\right)^{1/\zeta'}\\
&\leq & K_{\zeta} \nu(A) ^{1-1/\zeta}\;.
\end{eqnarray*}
\end{proof}
\begin{prop}
Under the assumptions of Theorem~\ref{theo:integ}, there
exists a positive constant $C$ so that for $\nu$ almost
every point $z$, 
\[\frac{d\sigma}{d\nu}(z) \geq C, \]
where
$d\sigma/d\nu$ stands for the density of the absolutely continuous
invariant measure $\sigma$. 
\end{prop}
\begin{proof}
Let ${\rm g}=\sum_{k=1}^{\infty}
\gamma_{k}^{-\delta}\Delta_{k}^{-(1-\frac{1}{\mmax})\delta}$
and  ${\rm g}_{n}(z):=\sum_{k=n}^{\infty}
\gamma_{k}^{-\delta}\Delta_{k}^{-(1-\frac{1}{\mmax})\delta}$.
By the hypothesis of  Theorem~\ref{theo:integ} (the intergrability condition),  
\[ \int {\mathrm g}(z)\, d\nu < \infty \]
and  thus for almost all $z\in \CC$ with respect to $\nu$,
$\lim_{n\rightarrow\infty}{\rm g}_{n}(z)=0$.
Consequently, there exist a point $w \not \in \bigcup_{n=1}^{\infty}F^{-n}(\Crit)$ and a positive integer $N$ so that $d\sigma/d\nu (w)\geq 1$ and 
${\rm g}_{n}(w)<1/2$ for every $n\geq N$. Observe that the choice of $N$
and $w$ does not depend on $R'$.

Choose  $R'$ so small that $B_{R'}(w)\cap F^{k}(\Crit)=\emptyset$
for all $k\leq N$ and decompose the backward
orbit of  $w$ into pieces of type $1$, $2$ or $3$ as described in the proof 
of Proposition~\ref{prop:in1}. By the construction, only the following
three types of codings are allowable: $2$, $1\dots 3$, $21\dots 3$.

Let ${\cal R}_{ R'}^{n}(z)=\sum_{\{2\}}|(F^{n})'(y)|^{-\delta}$ be
the sum over all type $2$ preimages of $z$ of length $n$ (regular part). 
The sum over all other preimages is denoted by
${\cal S}_{R'}^{n}(z)$ (singular part, compare \cite{przytycki-ce}). 
By the choice of $R'$, ${\cal S}_{R'}^{n}(w)<1/2$ 
for every $n$ positive.
Hence, by the choice of $w$ and for  $n$ large enough
\[\frac{3}{4}\leq \frac{1}{n}\sum_{i=0}^{n-1}{\cal L}^{i}(w)= 
\frac{1}{n}\sum_{i=0}^{n-1}{\cal R}_{R'}^{i}(w)+
\frac{1}{n}\sum_{i=0}^{n-1}{\cal S}_{R'}^{i}(w)
\leq \frac{1}{n}\sum_{i=0}^{n-1}{\cal R}_{R'}^{i}(w)
+\frac{1}{2}~.\]
We choose new $R_{\rm{new}}':=R'/2$. 
This will change the decomposition of backward
orbits of points,  generally allowing more type $2$ sequences
of preimages. If $z\in B_{R'/2}(w)$ then every type $2$ preimage of $w$ 
with the parameter $R'$ corresponds to exactly one type $2$ preimage of $z$
with $R_{\rm new}'$. By the bounded distortion,
${\cal R}_{R'/2}^{n}(z)~\grtsim ~{\cal R}_{R'}^{n}(w)$ and thus
for $n$ large enough,
\[\frac{1}{n}\sum_{i=0}^{n-1}{\cal L}^{i}(z)\geq \frac{1}{n}\sum_{i=0}^{n-1}
{\cal R}_{R'}^{i}(w) ~\grtsim ~1/4~.\]

By the eventually onto property, there exists a positive integer $m$
so that for every $z\in J$
there is a preimage $u=F^{-m}(z)\in B_{R'}(w)$. 
Let $M=\sup_{z\in J} |F'(z)|$. Then
for almost every $z\in \CC$ with respect to $\nu$,
\[d\sigma/d\nu(z)=\lim_{n\rightarrow \infty}\frac{1}{n}
\sum_{i=0}^{n-1}{\cal L}^{i}(z) \geq M^{-m} 
\lim_{n\rightarrow\infty}\frac{1}{n}\sum_{i=j}^{n-1}{\cal L}^{i-j}(u)~ 
\grtsim ~M^{-m}/2~.\]
\end{proof}

\subsection{Analytic maps of interval}
We are interested in analytic maps $f$ with negative Schwarzian
derivative (i.e. ${S}(f)<0$, recall the formula (\ref{eq:schdef})),
which map a compact
interval  $I$ with non-empty interior  into itself. 
We will  denote by $F$  
a complex extension of $f$ 
to a neighborhood $U_F\supset I $ in the complex plane which does not contain any critical points different then the real ones. We will study iterates of {\em real}$\;$ inverse
branches $\Fe$ of $F$ defined by the condition: 
if $U\subset U_F$ is a topological disk such that $U\cap \R$ is an interval then for every
 $x\in U\cap \R$,  $\Fe(x)\in I$. By the definition, $\Fe(U)$ is a topological disk and
$\Fe(U\cap \R)=  \Fe(U)\cap \R$~.   


We will show that  iterates of points by the real inverse branches
stay in a bounded distance from the interval $I$. 
To this aim we will need the disk property  based on 
 Proposition~3 and Proposition~4 of \cite{gss}. For every  $a,b\in \R$ define  $D_{a,b}$
 to be the open disk  centered at $(a+b)/2$ with the radius $|a-b|/2$. 
\begin{fact} \label{SAD1}
Let $h:I\to \R$ be an analytic diffeomorphism from a compact interval $I$
with non-empty interior into the real line. Suppose that
$f:I\to \R$ either concides with $h$ or is  of the form $h^{\ell}$, $\ell>1$
is an integer.  In the latter case assume that  there exists $\zeta\in I$ such that 
$h(\zeta)=0$.  Let $H$ be an 
extension of $h$ to a complex neighborhood of $I$ and $F(z)$
be equal either to $H(z)$ or $H(z)^{\ell}$.
If $S(f) < 0$ on $I$ then there exists
$\delta > 0$ such that for every two distinct points $a,b\in I$ with 
$\zeta \not \in  (a,b)$  and $\diam{D_{f(a),f(b)}}<\delta$ the connected
component of $F^{-1}(D_{f(a),f(b)})$ which contains $(a,b)$
is contained in $D_{a, b}$.
\end{fact}

\begin{lem}\label{coro:real}
There exists  a constant $L>0$ so that for every $n\in \N$, every $x\in I$,
and every  inverse branch $f^{-n}$ defined on $B_L(x)\cap \R$, 
the real inverse branch  $F_{\rm{rl}}^{-n}$  which coincides with $f^{-n}$ on
$B_L(x)\cap \R$ is well defined on $B_L(x)$ and 
 $F_{\rm{rl}}^{-n}(B_{L}(x))\subset U_{F}$.
\end{lem}
\begin{proof}
Observe that
$f$ has only finitely many critical points in $I$ and for every critical
point $c\in \Crit$ there exists a neighborhood $U_{c}\ni c$
such that $F$ is in the form $F(c)(1-H_{c}(z)^{\ell(c)})$ with 
$H_{c}$  a biholomorphic
function near $c$. The proof follows immediately from
Fact~\ref{SAD1}. 

\end{proof}

\paragraph{Proof of Theorem~\ref{coro:uni}.} 
The normalized Lebesgue measure $\nu$ on $I$ can be regarded as a $\delta$-conformal 
measure with an exponent $\delta=1$. Clearly, $\nu$ satisfies
the uniform integrability condition for any $\eta<1$. 
We will show that we can use the estimates of
Proposition~\ref{prop:in1} for the real inverse branches of $F$. 

From  Lemma~\ref{coro:real}  we infer that all preimages of disks
$B_{L}(x)$ by the  real inverse branches of $F$  are
well-defined and of uniformly bounded diameters. 
This means that the estimates of
Lemma~\ref{lem:main} (the Main Lemma)
for disks $B_{\Delta}(x)$, $x\in\R$, and the real inverse branches $F_{\rm{rl}}^{-N}$  
are still valid. 

We define a real part ${\cal
L}_{{\mathrm rl}}$ of the Ruelle-Perron-Frobenius operator $\cal
L$ by,
\[{\cal L}_{{\mathrm rl}}(\nu)(z):=
\sum_{y\in F^{-1}
(z)\cap \R}\frac{1}{|F'(y)|^{\delta}}\;.\]

In the proof of Proposition~\ref{prop:in1},
the equality $\Jac_{\nu}(x)=|F'(x)|^{\delta}$  was used  only for the estimates
involving type $2$ preimages. Explicitly, it is the estimate
(\ref{equ:only}). Let  $F_{\rm{rl}}^{-k}$ be a real inverse branch 
of type $2$ (we use  the inductive procedure from the
proof of Proposition~\ref{prop:in1}) and  $y= F_{\rm{rl}}^{-k}(x)$.   We have the following
uniform estimate,  
\[|(F^{k})'(y)|\;\sim \frac{\nu(B_{R'}(x))}{\nu(F^{-k}(B_{R'}(x)))}~.\]

The inductive construction of Proposition~\ref{prop:in1} yields a coding of backward orbits  by
sequences of $1$'s, $2$'s, and $3$'s. We recall that only three types of the codings are
allowable: $2$, $1\dots 13$, $21\dots13$.  
Lemma~\ref{coro:real} guarantees  that only the critical points of $f$ are used in the inductive
construction of the  backward codings.  

After these preparations
we are ready to invoke Proposition~\ref{prop:in1} with
$\cal L$ replaced by ${\cal L}_{{\mathrm rl}}$.
Hence, there exists a positive constant $K$ such that for every $z\in \R\setminus \bigcup_{n=1}^{\infty} f^{n}(\Crit)$
and every positive integer $N$, 
\[{\cal L}_{{\mathrm rl}}^{N}(\nu)(z) =\sum_{y\in F^{-N}
(z)\cap \R}\frac{1}{|(F^{N})'(y)|^{\delta}} < K\; \sum_{k=1}^{\infty}
\gamma_{k}^{-\delta}\Delta_{k}^{-(1-\frac{1}{\mmax})\delta}\;,\]
where $\Delta_k=\dist{f^k(\Crit),z}$ and the sequence $\gamma_{k}^{-1}$ (defined in \mbox{Lemma~\ref{lem:techseq}})
is summable with  an exponent smaller than $\delta$. To complete the proof, 
observe that ${\cal L}_{{\mathrm rl}}$ with $\delta=1$ is
the Ruelle-Perron-Frobenius operator  for $f$. Therefore, the densities of
$\nu_{N}=f^{N}_{*}(\nu)$ are bounded by   a constant multiple of an $L^1$ function $\sum_{k=1}^{\infty}\gamma_{k}^{-1}\Delta_{k}^{-(1-\frac{1}{\mmax})}$. Any
weak limit of $\nu_{N}$ gives an absolutely continuous
invariant measure for $f$. The additional claims
of Theorem~\ref{coro:uni} follow immediately from Corollary~\ref{theo:inreg}.

\section{Geometry of Fatou components}
\subsection{Integrable domains}
\begin{lem}\label{prop:sum}
For every periodic Fatou component  ${\cal F}$ 
the following two conditions are equivalent:
\begin{enumerate}
\item For any 
(some - by the 
Koebe distortion lemma~\ref{lem:koeb}
the statements are equivalent)  
point $z\in{\cal F}$ away from the critical orbits 
there exist a sequence $\omega_{n}$ so that
$\abs{\br{F^n}'(y)}~>~\omega_{n}~$
for points $y\in{\cal F}\cap F^{-n}z$,
and $\sum_{n=1}^{\infty}\omega_{n}^{-1}<\infty$.
\item $\cal F$ is an integrable domain.
\end{enumerate}
\end{lem}
Without loss of generality we may assume that
$F$ fixes a Fatou component ${\cal F}$. 
Throughout the rest of this Section we will always mean by $F^{-n}$
a branch mapping ${\cal F}$ to itself.

Take a subdomain $\om\,\subset\,\Cal F$ with smooth
boundary containing all critical points from $\Cal F$
such that $F\om~\subset~\om$.
Any point $z\in{\cal F}$ eventually gets to $\om$
under some iterate of $F$, so we can define
$n(z):=\min\brs{n:~F^n(z)\in\Omega}$.
Also fix a reference point $z_0\in\om$.

Lemma~\ref{prop:sum} follows immediately from
 Lemma~7 from \cite{ceh} which we state as Lemma~\ref{lem:ourkoebe}.
\begin{lem}\label{lem:ourkoebe}
Suppose that $z\notin\om$
and $n=n(z)$.
Then 
\begin{eqnarray*}
\dist{z,\partial\Cal F}&\asymp&\abs{\br{F^n}'(z)}^{-1}~,\\
\qhdist{z,z_0}&\asymp&n~,
\end{eqnarray*}
up to some constant depending on $\Cal F$ and our choice of $\Omega$ only.
\end{lem}

The proof of Theorem~\ref{prop:integrable} is preceded by a few
analytical observations, compare \cite{gehring-martio} and \cite{smith-stegenga}.
We will use that quasihyperbolic metric is a geodesic metric, see 
exposition \cite{koskela-quasihyp} for this and other properties of quasihyperbolic metric.
\begin{lem}\label{lem:integre}
Every  integrable domain (see Definition~\ref{def:integral})  is geodesically bounded, i.e.  
every  minimal quasihyperbolic geodesic  is of  uniformly bounded Euclidean arclength. 
\end{lem}
\begin{proof}
Let $\gamma$ be a (minimal) quasihyperbolic geodesic joining $z_{0}$
and $z$, i.e. $\qhdist{z_{0},z}=\qhlen\gamma$. 
We parameterize $\gamma$ by its Euclidean arclength 
starting from $z_{0}$. 
We define function $g$ on the interval $[0,\len\gamma]$
by $g(t):=\qhlen{\gamma[0,t]}=\qhdist{z_0,\gamma(t)}$,
the latter quantities are identical
since the geodesic is minimal. 
Then 
$$\frac{d}{dt} g(t)=\dist{\gamma(t),\partial{\cal F}}^{-1}~.$$
From the definition of integrable domains,
we obtain the following differential inequality
\begin{equation}\label{equ:equ}
\frac{d}{dt}g(t) \cdot H(g(t))\geq 1~.
\end{equation} 
Integrating leads to
\[s=\int_{0}^{s}dt \leq \int_{0}^{s}
 H(g(t))dg(t) \leq 
\int_{0}^{\infty}H(r)dr < \infty~. \]
\end{proof}

Denote by $\gamma(z,y)$ a minimal quasihyperbolic geodesic
joining $z, y \in \Omega$, that is 
 $\qhdist{z,y}=\qhlen{\gamma(z,y)}$.
\begin{lem}\label{lem:geo}
Let $z_{0}$ be a base point in $\Omega$ (see Definition~\ref{def:integral}).
There exists a continuous decreasing function $A: \R_{+}\to\R_{+}$ 
with $\lim_{r\rightarrow 0}A(r)=\infty$ so that for
every point $z_{1}\in \gamma$
\[\qhdist{z_{0},z}\leq A(\len{\gamma(z_{1},z)})~.\]
\end{lem}
\begin{proof}
We follow the notation from the proof of Lemma~\ref{lem:integre}.
Let $\qhlen{\gamma(z_{0},z_{1})}=L$ and $z_1=\gamma(t)$. Then 
$\qhdist{z_{1},z}=L-g(t)$.
By the inequality (\ref{equ:equ}),
\[\frac{d}{dt}\len{\gamma(z_{1},z)}= -1
\geq -\frac{d}{dt}g(t) H(g(t)) = \frac{d}{dt}
\int_{\qhdist{z,z_{0}}}^{\infty}H(r)dr~.\]
Integrating from  $t$ to $L$, we obtain that
\[-\len{\gamma(z_{1},z)} \geq 
-\int_{\qhdist{z,z_{0}}}^{\infty}H(r)dr+\int_{L}^{\infty}H(r)dr~\]
and 
\[\len{\gamma(z_{1},z)} \leq 
\int_{\qhdist{z,z_{0}}}^{\infty}H(r)dr~.\]
Set $A(r)$ to be the inverse of $\int_{r}^{\infty}H(r)dr$.
Then $A(r)$ is a non-decreasing continuous function
and $\lim_{r\rightarrow 0}A(r)=\infty$. The lemma follows. 
\end{proof}
\begin{coro}\label{coro:loc}
For every integrable domain $\Omega$
there exists a continuous function $A$ defined in Lemma~\ref{lem:geo} so that
for every subarc $\gamma_{1}$ of a minimal quasihyperbolic geodesic
$\gamma$ starting at the base point $z_{0}$ 
the following inequality holds
\[\frac{\len{\gamma_{1}}}{A(\len{\gamma_{1}})}~\lesssim ~\max_{z\in \gamma_{1}}
\dist{z,\partial\Omega}~.\] 
\end{coro}
\subsection{Local connectivity
and continua of convergence}   
A connected set $K$ is {\em locally connected} if for every $z\in K$
and each neighborhood $U$ of $z$ there exists a neighborhood $V$
of $z$ such that $K\cap V$ lies in a single component of $K\cap U$.

\begin{defi}
A  continuum $K_{\infty}\subset M$  is called a
continuum of convergence of  a set $M$ if there exists a sequence
of continua $K_{i}\subset M$ so that
\begin{enumerate}
\item $K_{i}$ are pairwise disjoint for $i=1,\dots,\infty$.
\item $\lim_{i\rightarrow \infty}K_{i}=K_{\infty}$ in the Haudorff  metric. 
\end{enumerate}

\end{defi}
A concept of continuum of convergence has a principal application
in study of local connectedness of continua.
We will say that a continuum of convergence is non-trivial if it
contains more than one point. We have the following fact (see
Theorems~12.1 and 12.3 in \cite{whyburn}) which follows almost immediately
from the definition of local connectivity.
\begin{fact}~\label{fact:why}
If a continuum $M$ is not locally connected then there exists a
non-trivial continuum of convergence of $M$.
\end{fact}
From Fact~\ref{fact:why}, it is clear that the local continuity of $M$
cannot fail just in one point. It fails for all points
from  a non-trivial continuum of convergence.

\begin{lem}\label{lem:conver}
The boundary of an integrable
domain $\Omega$ (which is not necessarily connected) does not have a non-trivial
continuum of convergence.
\end{lem}
\begin{proof}
Suppose that there exists a non-trivial continuum
of convergence  $K_{\infty}\subset \partial \Omega$. 
Let $K_{i}\rightarrow K_{\infty}$ be mutually disjoint continua of 
$\partial \Omega$.
We choose a point $w$ in $K_{\infty}$ and $\epsilon>0$ 
so that the circle $\{|z-w|= \epsilon\}$
intersects $K_{\infty}$ in at least two points and
$\Omega\setminus \overline{B_{\epsilon}(w)}$ is non-empty. 
Without loss of generality, every $K_{i}$ intersects $\{|z-w|=\epsilon\}$ 
in at least two points and hence $B_{\epsilon}(w)\setminus K_{i}$
has at least two components.
For every $n>1/\epsilon$ choose  a component  $\Omega_{n}$ of
$\Omega\cap B_{\epsilon}(w)$ and $z_{n}\in \Omega_{n}$    
with the following properties: 
${\rm(i)}$  $\dist{z_{n},w}=1/n$, ${\rm (ii)}$ 
there exists $j$ such that $w$ and
$z_{n}$ are in  different components of $B_{\epsilon}(w)\setminus K_{j}$.
Observe that there are infinitely many different $\Omega_{n}$. Indeed, 
let $\Gamma_{j}$ be a component of $B_{\epsilon}(w)\setminus K_{j}$ 
which contains $w$. By {\rm (ii)},
there exists $j$ so that $\Omega_{n}$
is contained in a component of $B_{\epsilon}(w)\setminus K_{j}$ different 
than $\Gamma_{j}$. By {\rm (i)}, there
exists $N>n$ so that $z_{N}$ and thus $\Omega_{N}$ are contained
in $\Gamma_{j}$. Consequently,
every $\Omega_{m}$ with $m>N$ and $\Omega_{n}$ are different.

By the connectivity of $\Omega$, $\partial\Omega_{n}\cap \{|z-w|=\epsilon \}\not = \emptyset$. 
We join every $z_{n}$ with a base point $z_{0}\in \Omega
\setminus \overline{B_{\epsilon}(w)}$ 
by a minimal quasihyperbolic geodesic $\gamma(z_{n},z_{0})$.
Let  $\gamma_{n}\ni z_{n}$ be a component of $\gamma(z_{n},z_{0})\cap
B_{\epsilon/2}(w)$.
The Euclidean length of $\gamma_{n}$ is at least $\epsilon/4$ for
large $n$.
Then, by Corollary~\ref{coro:loc}, 
\[~\max_{y\in \gamma_{n }}\dist{y,\partial\Omega_{n}}
\geq~\max_{y\in \gamma_{n }}\dist{y,\partial\Omega}
~\grtsim~\frac{\len{\gamma_{i}}}{A(\len{\gamma_{i}})}~\grtsim~
\frac{\epsilon}{4A(\epsilon/4)}=:\delta.\]
Consequently, every $\Omega_{n}$ contains a ball of radius $\delta$ 
and hence there  are infinitely many disjoint
balls of the same radius in  $B_{\epsilon}(w)$, a contradiction.

\end{proof}
The non-existence of a non-trivial continuum of convergence implies
immediately that every component of the boundary is locally connected.
We  see immediately that every subcontinuum of a continuum
without non-trivial continuum of convergence is locally connected.
In fact it gives even more precise topological description of every
component of $\partial \cal F$ (see \cite{whyburn} pp 82-84). 
\begin{defi}
We say that a continuum $K$ is a {\em rational} curve if each point $z\in K$
has a family of arbitrary small neighborhoods $U_{n}$ so that $\partial U_{n}
\cap K$ is {\em countable}.
\end{defi}
Theorem~3.3 and Remark 2.1 of Chapter V of \cite{whyburn} assert
that every continuum which has no non-trivial continuum of
convergence is a rational curve.  Another direct consequence
of Lemma~\ref{lem:conver} is that for every $\epsilon > 0$
there exists only finitely many components $\partial {\cal F}$ 
with diameters larger than $\epsilon$. This in turn will imply
the absence of wandering continua for polynomials which satisfy
the summability condition.
\subsection{Wandering continua}
\begin{defi}\label{def:wander}
A continuum  $K$ is called wandering if for every two non-negative integers
$m\not = n$
\[F^{m}(K)\cap F^{n}(K)=\emptyset\;.\]
\end{defi}

\begin{lem}\label{lem:wand}
Suppose that  $F$ satisfies the summability condition with
exponent $\alpha< \frac{1}{1+\mmax}$ and $\cal F$ is a periodic
Fatou component. Then $F$ has no wandering continua contained
in $\partial {\cal F}$ other than points.
In particular, if $F$ is a polynomial then
every non-point component of connectivity of $J$
is preperiodic.
\end{lem}
\begin{proof}
We can always assume that $K$ is connected. 
Let $R'$ be supplied by Proposition~\ref{prop:shrink}.
We look at the orbit  $K_{i}:=F^{i}(K)$. Since $\partial {\cal F}$
does not have a non-trivial continuum of convergence almost all
continua $K_{i}$ have the diameter smaller than $R'/2$. 
A ball $B_{i}$  of the radius $R'$ centered at any point of $K_{i}$
contains $K_{i}$.
By Proposition~\ref{prop:shrink}, we obtain that
\[\diam K \leq \diam  F^{-i}(B_{i}) < (\tilde{\omega_{n}})^{-1} \rightarrow 0\]
and $K$ must be a point.
\end{proof}

\section{Uniform summability condition}
\paragraph{Continuity of  scales and parameters.} 
Suppose that $F$ is a rational function which satisfies the summability
condition with an exponent $\alpha$ and a sequence of rational functions $F_i$ tends $S(\alpha)$-uniformly to $F$.
By the definition of $S(\alpha)$-uniform convergence and Lemma~\ref{lem:techseq}, we choose
the same sequences $\{\alpha_{n}\}$, $\{\gamma_{n}\}$, and $\{\delta_{n}\}$ 
for all $F_i$ and $F$.

All constants defined in $\rm{(i)-(iii)}$ in Section~\ref{sec:const} can be chosen uniformly that is
independently from $F_i$ for $i$ large enough.  The condition (iv) from Section~\ref{sec:const} 
needs to be replaced by the condition $\rm {(iv')}$: 
\begin{description}
\item{$\rm{(iv')}$} Let $\Crit'_F$ be a set of critical points of $F$ outside the Julia set $J_F$.
We chose $R'$ is so small that for all $i$ large enough  there are no critical points $\Crit_F'$
in the $2R'$-neighborhood of the Julia set $J_i$. 
\end{description} 
This choice of $R'$ is indeed possible. Since $F$  and $F_i$, $i\geq 0$,  satisfy the summability
condition  all periodic orbits of  $F$ and $F_i$, $i\geq 0$, are either repelling or attracting, see Corollary~\ref{cor:nosiegel}. Suppose now that we can find   a sequence $c_{i_k}$ of  critical
points of $F_{i_k}$ such that  for some $c\in \Crit_F'$
$$\dist{c_{i_k}, c}\rightarrow 0\;\;\; \mbox{and}\;\;\:   \dist{c_{i_k}, J_{i_k}} \rightarrow 0$$ 
when $k\rightarrow \infty$. The critical point  $c$ is in the basin of an attracting periodic point  $w$.
Let $\eta>0$.  There is $m>0$ such that  $\dist{F^m(c), w}< \eta/3$.    For every $i$ large enough there is an attracting periodic point
$w_i$ of $F_i$ such that $\dist{w_i, w}<\eta/3$. 
Since $\dist{F^m(c_{i_k}), J_{i_k}}\rightarrow 0$ when $k\rightarrow \infty$,  we have that
 for  all $k$ large enough $\dist{J_{i_k},w_{i_k}}<\eta$, a contradiction.


\paragraph{Uniform version of Lemma~\ref{lem:2t}.}
We will argue that after the above choice of the sequences,  the constants and scales $\rm{(i)-(iv')}$,
all  the constants  $R', R_{2t}, L, C(q), K$ from  Lemma~\ref{lem:2t} 
are uniform,   that is  independent from   $F_i$ for $i$ large enough.
\begin{enumerate} 
\item
We say that a  sequence $F^{-n}(z), \dots\, F^{-1}(z), z,$ of  type  $2$ preimages of $z$
is in  the scale $R'$ if   the ball $B_{R'/2}(z)$ can be pulled back univalently 
along the orbit $F^{-n}(z),\cdots, F^{-1}(z), z$   and $\dist{z,J_F}\leq R'/2$.


We claim that for  every 
$L_u'>0$ there exists a neighborhood
of $F$ in the space of rational functions with the topology of uniform convergence
so that for every $G$ from that neighborhood
and every sequence $G^{-L_u'}(z), \cdots, z$ of type $2$ preimages of $z$ in the scale $R'$
and such that $\dist{z,J_G}<R'/16$,  the corresponding sequence
  $F^{-L_u'}(z), \cdots, z$ is of type $2$  in the scale $R'/2$.  Indeed, if $G$ close enough to $F$
  then  $J_F$ is contained in $R'/8$-neighborhood of $J_G$. To see this consider all 
repelling periodic orbits of $F$ with  periods smaller than a constant  and  the property that they  are $R'/16$-dense in $J_F$. By the implicit function theorem $J_G$ is $R'/4$ dense in $J_F$ if $G$ is close to $F$.  

We possibly increase the constant  $L$  from Lemma~\ref{lem:2t}  so that the first two claims
of Lemma~\ref{lem:2t} are satisfied  with the estimates $7$ and $1/37$ 
instead of $6$ and $1/36$, respectively, for all type $2$ preimages of $z$ by $F$
in the scale $R'/2$. We put $L_u:=L$.

Therefore, by $C^1$ continuity  for every  $L_u'>L_u$ there exists a neighborhood of $F$
so that for every $G$ from this neighborhood and every sequence of type $2$ preimages
of $z$ in the scale $R'$ , 
 
$$\inf_{y\in \typeii{(z)},\,L_u'\ge n(y)\ge L_u}~\abs{\br{G^n}'(y)}~
>~6~,$$
and  if the Poincar\'e series $\Sigma_{q}(v)$ for $F$ converges
for some point $v$  then
$$\sum_{y\in \typeii{(z)},\,L_u'\ge n(y)\ge L_u}\abs{\br{G^n}'(y)}^{-q}
~<~\fr1{36}~$$
provided $\dist{z,J_G}<R'/16$.
\item 
The constants $K, R_{2t}$, and $C(q)$ of Lemma~\ref{lem:2t} come from
the Koebe type estimates and  depend only on other uniform constants hence are uniform. 

\end{enumerate}
We formulate a uniform version of Lemma~\ref{lem:comp}.
\begin{lem}\label{lem:comp1}
Let $F$ be a rational function which satisfies the summability condition with an exponent $\alpha\leq 1$
and $(F_i)$ be a sequence of rational maps converging $S(\alpha)$-uniformly to $F$. 
There exists $\epsilon>0$ such that the backward orbit  $\;\cdots , F_i^{-k}(z), \cdots , z\;$ 
of any $z$ from the $\epsilon$-neighborhood of the Julia set $J_i$
stays in the $R'/16$-neighborhood of $J_i$.  The same claim is true for $F$. 
\end{lem}
\begin{proof}
The existence of $\epsilon_i>0$ for every $F_i$ is proved in Lemma~\ref{lem:comp}.
If $\epsilon_i$ tend to $0$ then $J_F\not = \hat{\C}$. Suppose that there is an infinite
set $I$ of positive integers such that  for every $i\in I$ there are $z_i\in \hat{\C}$
  and $n_i\in \N$ so that $\dist{z_i, J_i}\geq R'/16$ and
$\dist{F_i^{n_i}(z_i), J_i}\leq \epsilon_i$. By the compactness, there exists
a converging sequence $z_{i_j}\rightarrow z\in \hat{\C}$  with the property that $B_{R'/32}(z)$ is disjoint with  every  Julia set $J_{i_j}$, $i_j\in I$, and thus also disjoint from $J_F$.  This means that 
there is  $m>0$ such that $F^m(z)$ is in the immediate bassin of attraction of
an attracting  periodic point $w$ of $F$. Let $w_{i_j}$ be the corresponding attracting periodic point
of $F_{i_j}$ for $j$ large enough.  By the continuity,   $F_{i_j}^m(z)$ are uniformly close
to $w_{i_j}$  and therefore the orbit of $z$ by $F_{i_j}$
 must accumulate on the periodic orbit of $w_{i_j}$ for all $j$ large enough, a contradiction.
\end{proof}

\begin{lem}[Uniform version of the Main Lemma]\label{lem:mainunif}
Assume that a rational
function $F$ satisfies the summability condition with an exponent 
$\alpha\leq 1$ and set $\beta=\mmax\alpha/(1-\alpha)$.
Suppose that $F_i$ is a sequence of rational functions tending $S(\alpha)$-uniformly
to $F$. Let $\epsilon$ be supplied  by Lemma~\ref{lem:comp1}
and suppose that a point $z$ belongs to the $\epsilon$-neighborhood of the Julia
set $J_i$ and a ball $B_{\Delta}(z)$ can be pulled
back univalently by a branch of $F_i^{-N}$.

We claim that there exists $i_0$ and  positive constants 
$L'>L, K$ independent of $z, \Delta$, and $\epsilon$
such that for every $i\geq i_0$ the sequence $F_i^{-N}(z),\dots, z$ 
can be decomposed into blocks of types $1$, $2$, and $3$, and
\begin{itemize}
\item  every type $2$ block,
except possibly the leftmost one, 
 has the length contained in $[L,L')$ and yields expansion  $6$,
\item the leftmost type $2$ block has the length contained in $[0,L]$
and  yields expansion  $K>0$, 
\item  all subsequences of the form $1\dots13$,
except possibly the rightmost one,  yield expansion 
$$\gamma_{k_j}\dots\gamma_{k_1}\gamma_{k_0}~,$$
$k_i$ being the lengths of the corresponding blocks,
\item  the rightmost sequence of the form $1\dots 13$ yields
expansion
$$\begin{array}{ll}
 \gamma_{k_j}\dots\gamma_{k_1}\gamma_{k_0}~\Delta^{(1-\mu(c)/\mmax)}&
{\it~if~a~critical~point~}c\in B_{\Delta}(z)~,\\
\gamma_{k_j}\dots\gamma_{k_1}\gamma_{k_0}~\Delta^{(1-1/\mmax)}&
{\it~if~otherwise~.}
\end{array}$$
\end{itemize}
\end{lem}
\begin{proof}
We claim that the inductive construction  from the Main Lemma can be carried out 
for $F_i$ and any $z$ from the $\epsilon$-neighborhood of $J_i$ provided $i$
is large enough.  The constants and scales are given by the conditions $\rm{(i)-(iv')}$. 
Recall that the same sequences
 $\{\alpha_{n}\}$, $\{\gamma_{n}\}$, and $\{\delta_{n}\}$ 
can be used for both  $F_i$ and $F$. 
By the definition of  $ S(\alpha)$-uniform convergence,
there is  a $1-1$ correspondence between the critical points of $F$ and 
these of $F_i$ if only $i$ is large enough. We follow the inductive decomposition
of backward orbits  of $F_i$ of the Main Lemma with one exception,
all the critical points of $F_i$ contained in the $\epsilon$-neighborhood of $J_i$ are included
in the construction. We treat $F_i$ as small perturbations
of $F$ and work in the same scale $0<R'<R$  both for $F_i$ and $F$. 
Using the uniform summability condition for all critical points  inside the  $\epsilon$-neighborhood
of $J_i$,  we obtain the expansion from the last three  claims of Lemma~\ref{lem:mainunif} 
in exactly the same way as in the proof of the Main Lemma.  In the construction of    
all backward   pieces of type $1\dots13$ we use all the critical points of $F_i$
from the $\epsilon$-neighborhood of $J_i$.  

What remains  is the proof that we can find $L$ and $L'>L$
independently from $F_i$ so that  the first two claims of  Lemma~\ref{lem:mainunif}
hold. We put $L:=L_{u}$ and hence $L$
is independent from $F_i$  by the uniform version of Lemma~\ref{lem:2t}.
Next, for a given $F_i$,  the constant $L'$ in the proof of the Main Lemma is defined as 
$ L+L''$ where $L''$ is a function
of $\{\alpha_{n}\}$, $C_{3t}$, and $R_{2t}$ only. These  are the uniform
constants. We conclude the proof by
invoking the uniform version of Lemma~\ref{lem:2t} for $L_{u}$ and 
$L_{u}'=L_{u}+L''$. 
\end{proof}
\paragraph{Uniform self-improving property and main estimates.}
The proof of Proposition~\ref{prop:poincare} is self-contained except for references to  
Lemmas~\ref{lem:main2} and ~\ref{lem:bbb}. Suppose 
that rational functions $F_i$ converge $S(\alpha)$-uniformly to
a rational function $F$ which satifies the summability condition with
an exponent $\alpha \leq 1$.  By the uniform version of the Main Lemma
and Lemma~\ref{lem:comp1}, we obtain the  estimates of Lemma~\ref{lem:main2}
for all $F_i$ sufficiently close to $F$. The decomposition of backward orbits of $F_i$
into pieces of type $1,2$, and $3$ uses,  as in the uniform version of the Main Lemma,
all the critical points of $F_i$ contained in the $\epsilon$-neighborhood of $J_i$
($\epsilon$ comes from Lemma~\ref{lem:comp1}).  Lemma~\ref{lem:bbb} is a reformulation
of two claims of Lemma~\ref{lem:2t} which already has a uniform version.  The constants
$L:=L_u'$ and $L':=L_u'$ are supplied by the uniform version of the Main Lemma. 

Now  suppose as in the hypothesis of Proposition~\ref{prop:poincare} that $F$ satisfies the summability
condition with an exponent $\alpha< q/(\mmax+q)$  for some $q>0$ and that the sequence
$(F_i)$ converges $S(\alpha)$-uniformly to $F$. Also, assume that the Poincar\'e
series for $F$ with an exponent $q>0$ converges for some point $v\in \hat{\C}$,  $\Sigma_q(v)<\infty$.
In the proof of 
Proposition~\ref{prop:poincare} the assumption  $\Sigma_{q}(v)<\infty$ is used
to derive the following property:  there exists
$\eta>0$ so that  for every $z$ such that $\dist{z, J_F} < \eta$,
$$\sum_{y\in\typeii_{l}(z)}
\abs{\br{F^{n(y)}}'(y)}^{-q}~<~\fr1{36}~,$$ 
where $\typeii_{l}(z)$ is  the set of all 
``long'' (of order $L'>n(y)\ge L$) type $2$ preimages $y$ of $z$.
The property in fact holds for all $F_i$ close enough to $F$  by the uniform 
version of Lemma~\ref{lem:2t} with $\eta:=\epsilon$ of Lemma~\ref{lem:comp1}.
This shows that the self-improving property of Proposition~\ref{prop:poincare}
is true for all $F_i$ close enough to $F$.

Taking into account  the uniform self-improving property  and  Corollary~\ref{cor:critexpon},
we obtain the main estimate behind the continuity of the Hausdorff
dimension of  Julia sets for  $S(\alpha)$-uniformly converging rational maps. 
\begin{coro}\label{coro:dim}
Assume that $F$ satisfies the summability condition with an exponent
$$\alpha~<~\fr{p}{\mmax+p}~,$$
$p$, and a sequence of rational maps $(F_i)$ converges $S(\alpha)$-uniformly
to $F$.  If  there exist $\epsilon>0$ and $q>p$ so that for every point
 $z$ from the $\epsilon$-neighborhood of the Julia set $J_F$,
$$\sum_{y\in\typeii_{l}(z)}
\abs{\br{F^{n(y)}}'(y)}^{-q}~<~\fr1{36}~,$$ 
then for every  $i$ large enough 
$\dpoin(J_i)\,<\,q$. 
\end{coro}

\begin{coro}\label{coro:cont}
If a sequence $(F_{i})$ of rational functions
tends $S(\alpha)$-uniformly to $F$ which satisfies the summability
condition with an exponent  $\alpha< \frac{\dpoin(J_{F})}{\mmax+\dpoin(J_{F})}$
 then
\[ \limsup_{i\rightarrow \infty}\dpoin(J_{i})\leq \dpoin(J_{F})\;.\]
\end{coro}
\begin{proof}
By Theorem~\ref{theo:poincare}, for every positive $\eta$ and every critical point of maximal
multiplicity,  the Poincar\'e series for $F$,  $\Sigma_{\dpoin(J_{F})+\eta}(c)<\infty$.
By the uniform version of  Lemma~\ref{lem:2t}, for every $\eta>0$ 
we can find $L_u(\eta)\leq L_u'(\eta)$
so that for every $z$ such that $\dist{z, J_i}< \epsilon$
($\epsilon$ is supplied by Lemma ~\ref{lem:comp1}),
$$\sum_{y\in \typeii{(z)},\,L_u (\eta)'\ge n(y)\ge {\l2t}_u(\eta)}\abs{\br{F_{i}^n}'(y)}^
{-\dpoin(J_{F})-\eta}
~<~\fr1{36}~$$
provided $i$ is large enough.  $\typeii{(z)}$ stands for a set of type $2$ preimages of $z$ 
for $F_{i}$.  In the definition of these type $2$ preimages all  critical points of $F_i$ inside the $\epsilon$-neighborhood of $J_i$ were considered. The constants $L_u'(\eta) >L_u(\eta)$ depend only
on $F$ and  $\eta$. 
We use the  uniform version  of Corollary~\ref{coro:dim}  (set $p:=\dpoin(J)$ and 
$q:=\dpoin(J)+\eta$)  to conclude that 
$\dpoin(J_{i})\leq \dpoin(J)+\eta$ for all $i$ large enough.

\end{proof}

\paragraph{Proof of Theorem~\ref{theo:contin}.}
By Corollary~\ref{coro:cont} and Theorem~\ref{theo:poincare},
\[\limsup_{n\rightarrow \infty}\HD(J_{n})=\limsup_{n\rightarrow \infty}
\dpoin(J_{n})\leq  \dpoin(J_{F})= \HD(J_{F})\;.\]
Since $\HD(J_{F})=\HH(J_{F})$,  for every $\eta >0$  there 
exists a hyperbolic subset $K$ of $J_{F}$ 
so that $\HD(K)\geq \HD(J_{F})-\eta$
(see \cite{denker-urbanski-sullconf, denker-urbanski-existconf}).
By the general theory of hyperbolic sets,  the set $K$ persists
under small perturbations  and therefore every $J_{i}$ for $i$
large enough contains a hyperbolic set $K_{i}$ of the Hausdorff dimension $\HD(K)-2\eta$. Consequently, 
$$\lim\inf_{i\rightarrow \infty}\HD(J_{i})\geq \HD(J_{F})~,$$
which proves the theorem.

\section{Unicritical polynomials}
It is known that the connectedness locus ${\cal M}_{d}$
for unicritical polynomials $z^{d}+c$ 
is a full compact. 
Let $\phi$ be the Riemann map from the unit disk to 
$\C\setminus {\cal M}_{d}$. By Fatou's theorem for almost all
$\xi$ with respect to the Lebesgue measure on the unit circle, 
there exists a radial limit $\lim_{r\rightarrow 1} \phi(r\xi)$.
Therefore,  the harmonic measure $\chi$ is given by,
\[\chi = \phi_{*}(m)\;.\]
By Fatou's theorem, for almost every $c\in \partial{\cal  M}_{d}$ 
with respect to $\chi$ the external radius $\Gamma(c)$ 
(see Definition~\ref{defi:radius}) terminates at $c$.

Denote the Julia set of the polynomial $z^{d}+c$ by  $J_{c}$.
By Shishikura's theorem, $\cite{shishikura}$, there exists a residual set 
$Z \subset \partial{\cal M}_{2}$ with the property that $\forall c \in Z$,
\[ \HH(J_{c})=2\;.\]
Let $c\in \partial {\cal M}_{2}$ corresponds to a Collet-Eckmann
polynomial. By \cite{graczyk-swiatek-ce} and \cite{smirnov-symbce},
Collet-Eckmann parameters are
 typical with respect to the harmonic measure $\chi$ and the corresponding
Julia sets are of Hausdorff dimension $<2$, 
 \cite{ceh}. 
Choose now a sequence $c_{n}\in Z$, $c_{n}\rightarrow c$. Since
$\HH(\cdot)$ is lower semicontinuous, there exist open disks
$D_{n}$ centered at $c_{n}$ 
 so that $\forall c\in D_{n}$,
\[ \HD(J_{c_{n}})> 2-\frac{1}{n}\;.\]
By Yoccoz's theorem, ${\cal M}_{2}$ is locally connected at $c$
and thus there exists a curve $\gamma$ terminating at $c$  so that
\[\limsup_{\gamma \ni c'\rightarrow c}\HD(J_{c'})=2 > \HD(J_{c})\;\]
and $\HD$ as a function of $c'\in \C\setminus {\cal M}_{2}$
does not extend continuously to $\partial {\cal M}_{2}$.

Another type of discontinuity of $\HD(\cdot)$ is caused
by the parabolic implosion.
Let $c\in \partial {\cal M}_{2}$, $c\in \Gamma(c)$, has a parabolic cycle.
The parabolic implosion  means that $\HD(J_{c})$ is strictly contained in
the Hausdorff limit of $J_{c'}, c'\in \Gamma(c)$. 
It was recently shown in \cite{douady-sentenac-zinsmeister} that if $d=2$
and $c> 1/4$ then 
\[\HD(J_{1/4})< \liminf_{c\rightarrow 1/4}\HD(J_{c}) 
\leq  \limsup_{c\rightarrow  1/4}\
\HD(J_{c}) < 2~. \]

\subsection{Renormalizations}
We will start with the observation  unicritical  polynomials satisfying
the summability condition are not infinitely renormalizable. 
\begin{lem}\label{lem:rig}
Suppose that $f_{c}$ satisfies the summability condition
with an exponent $\alpha\leq \frac{1}{1+d}$.
Then $f_{c}$ is only finitely many times renormalizable.
\end{lem}
\begin{proof}
Indeed, suppose that $f:= f_{c}$ is infinitely renormalizable.
Then there is a sequence $n_{j}$ and two topological disks 
$\overline{U_{j}}\subset V_{j}$
so that $f^{n_{j}}:U_{j}\rightarrow V_{j}$ is proper of degree $\ell$ and 
\[J_{j}:= \cap_{i=0}^{\infty}f^{-n_{j}i}(U_{j})\ni 0\]
is a non-trivial continuum ($f^{-n_{j}}$ is a branch which
sends $V_{j}$ onto $U_{j}$). 
We may assume that $\forall j\geq 0$, $U_{j+1}\subset U_{j}$.
Let $R'$ be a constant chosen in Section~\ref{sec:const}.
If for every $j$ positive and every $0\leq k\leq n_{j}$, 
$\diam f^{k}(J_{j})\geq R'$
then  $\cap_{j}^{\infty} J_{j}$ is a non-trivial wandering continuum
and the Julia set of $f_{c}$ would  not be locally connected.
Hence, there exist $j$ and $k$ so that $\diam f^{k}(J_{j})\leq R'$. 
Let $B_{R'}\supset f^{k}(J_{j})$.
We apply Proposition~\ref{prop:shrink} for the  inverse branches
$f^{-k-n_{j}s}$, $s$ is a positive integer,  
$f^{-n_{j}s}(B_{R'})\supset J_{j}$,
and the ball $B_{R'}$. Then
\[\diam J_{j}= \diam f^{-k-n_{j}s}(B_{R'})\leq \omega_{k+n_{j}s}^{-1} \rightarrow 0,\]
a contradiction.
\end{proof}
\subsection{Radial limits}
\paragraph{Proof of Theorem~\ref{theo:cont}.}
Let $f_{c}=z^{d}+c$.
We  use the following fact  stated as Theorem~1.2 in \cite{graczyk-swiatek-ce}.
\begin{fact}
Let $\Gamma(c_{0})$ be an external ray landing at $c_{0}$. 
For every $d\geq 2$
and for almost every $c_{0}\in \partial {\cal M}_{d}$ with
respect to the harmonic measure there exist constants $K>0$ and
$\Lambda>1$ so that for each $c\in \Gamma(c_{0})$ and every $n>0$
\[(f_{c}^{n})'(c)\geq K\Lambda^{n}\;.\]
\end{fact}
This means that for almost all $c_{0}$ with respect
to the harmonic measure, $f_{c}$, $c\in \Gamma(c_{0})$, converges $S(\alpha)$-uniformly to
$f_{c_{0}}$  for every $\alpha>0$. 
To complete the proof of 
Theorem~\ref{theo:cont}, we invoke
Theorem~\ref{theo:contin}, 
\[\limsup_{\Gamma(c_{0})\ni c\rightarrow c_{0}}\HD(J_{c})
=\HD(J_{c_{0}})\;.\] 

\paragraph{Proof of Theorem~\ref{theo:end}.}
Fix a generic $c_{0}\in \partial {\cal M}_{d}$ as in the
proof of Theorem~\ref{theo:cont}. 
Consider $f_{c}$ with $c\in \Gamma(c_{0})$. 
Since $f_{c}$ is hyperbolic,
by \cite{sullivan-rio} it admits 
a unique conformal measure $\nu_{c}$
which is a Hausdorff measure
(restricted to the Julia set $J_{c}$)
normalized by a multiplicative constant
so that its total mass is one. 
Hence, the conformal exponent of $\nu_{c}$ is equal to
$\HD(J_{c})$. By  Theorem~\ref{theo:cont},
$\HD(J_{c})\rightarrow \HD(J_{c_{0}})$ as
$\Gamma(c_{0}) \ni c\rightarrow c_{0}$. 
This means that any weak accumulation point of $\nu_{c}$
is a conformal measure with an exponent $\HD(J_{c_{0}})$. But there is
only one such a measure for $f_{c_{0}}$
by Theorem~\ref{theo:4}, and  we obtain the  convergence of
$\nu_{c}$ to the geometric measure of $f_{c_{0}}$.

\clearpage


\end{document}